\documentclass[10pt]{amsart}
\usepackage{pgf,tikz,pgfplots,color}
\pgfplotsset{compat=1.15}
\usepackage{mathrsfs} %mathscr
\usetikzlibrary{arrows}
\usepackage{fancyhdr}
\usepackage{lastpage} 

\usepackage[margin=1in]{geometry}
\usepackage{hyperref}
\hypersetup{
    colorlinks=true,
    linkcolor=blue,
    citecolor=brown,
    filecolor=magenta,
    urlcolor=blue,
}
\usepackage{amsmath}
\usepackage{amsfonts}
\usepackage{amssymb}
\usepackage{amsthm}
\usepackage{setspace}
\usepackage{inputenc}
\usepackage{comment}
\usepackage{mathtools}
\usepackage[shortlabels]{enumitem}

\usepackage{chngcntr}

\allowdisplaybreaks

\numberwithin{equation}{section} 

\newtheorem{thm}{Theorem}[section]
\newtheorem{cor}[thm]{Corollary} 
\newtheorem{prop}[thm]{Proposition}
\newtheorem{lem}[thm]{Lemma}
\theoremstyle{definition}
\newtheorem{defn}[thm]{Definition}

\newtheorem{note}[thm]{Notation}

\newtheorem{innercustomgeneric}{\customgenericname}
\providecommand{\customgenericname}{}
\newcommand{\newcustomtheorem}[2]{%
  \newenvironment{#1}[1]
  {%
   \renewcommand\customgenericname{#2}%
   \renewcommand\theinnercustomgeneric{##1}%
   \innercustomgeneric
  }
  {\endinnercustomgeneric}
}

\newcustomtheorem{customprop}{Proposition}
\newcustomtheorem{customlemma}{Lemma}

\theoremstyle{remark}
\newtheorem{rem}[thm]{Remark}

\def\XXint#1#2#3{{\setbox0=\hbox{$#1{#2#3}{\int}$ }
\vcenter{\hbox{$#2#3$ }}\kern-.6\wd0}}

\newcommand{\Z}{\mathbb{Z}}
\newcommand{\R}{\mathbb{R}} 
\newcommand{\C}{\mathbb{C}}

\newcommand{\N}{\mathbb{N}}
\newcommand{\1}{\mathbf{1}} 

\newcommand{\dbar}{\overline{\partial}}

\newcommand{\Bc}{\mathcal{B}}
\newcommand{\Co}{\mathscr{C}}
\newcommand{\Ds}{\mathscr{D}}

\newcommand{\Bs}{\mathscr{B}}

\newcommand{\Ec}{\mathcal{E}}

\newcommand{\Fc}{\mathcal{F}}
\newcommand{\Fs}{\mathscr{F}}

\newcommand{\Hc}{\mathcal{H}}
\newcommand{\Kb}{\mathbb{K}}
\newcommand{\Kc}{\mathcal{K}}

\newcommand{\Ms}{\mathscr{M}}

\newcommand{\Pc}{\mathcal{P}}

\newcommand{\Sc}{\mathcal{S}}
\newcommand{\Ss}{\mathscr{S}}
\newcommand{\Uc}{\mathcal{U}}
\newcommand{\Xs}{\mathscr{X}}

\newcommand{\eps}{\varepsilon}
\newcommand{\bmo}{\mathrm{bmo}}
\newcommand{\BMO}{\mathrm{BMO}}
\newcommand{\dist}{\operatorname{dist}}
\newcommand{\supp}{\operatorname{supp}}
\newcommand{\essup}{\mathop{\operatorname{essup}}}
\newcommand{\loc}{\mathrm{loc}}
\newcommand{\Span}{\operatorname{Span}}
\newcommand{\Coorvec}[1]{\frac\partial{\partial#1}}
\newcommand{\re}{\operatorname{Re}}

\newcommand{\Vol}{\operatorname{Vol}}

\makeatletter
\def\@tocline#1#2#3#4#5#6#7{\relax
  \ifnum #1>\c@tocdepth % then omit
  \else
    \par \addpenalty\@secpenalty\addvspace{#2}%
    \begingroup \hyphenpenalty\@M
    \@ifempty{#4}{%
      \@tempdima\csname r@tocindent\number#1\endcsname\relax
    }{%
      \@tempdima#4\relax
    }%
    \parindent\z@ \leftskip#3\relax \advance\leftskip\@tempdima\relax
    \rightskip\@pnumwidth plus4em \parfillskip-\@pnumwidth
    #5\leavevmode\hskip-\@tempdima
      \ifcase #1
       \or\or \hskip 1em \or \hskip 2em \else \hskip 3em \fi%
      #6\nobreak\relax
    \hfill\hbox to\@pnumwidth{\@tocpagenum{#7}}\par% <---- \dotfill -> \hfill
    \nobreak
    \endgroup
  \fi}
\makeatother

	\title[Cauchy-Riemann on Convex Domains of Finite Multitype]{Sobolev and H\"older Estimates for Homotopy Operators of The $\overline\partial$-Equation on Convex Domains of Finite Multitype}          

\author[]{Liding Yao} 
\address{Liding Yao, Department of Mathematics,
	The Ohio State University, Columbus, OH 43210} 
\email{yao.1015@osu.edu}

\thanks{Supported in part by NSF grant DMS-2054989}

\begin{document}

\begin{abstract}
    We construct homotopy formulas for the $\overline\partial$-equation on convex domains of finite type that have optimal Sobolev and H\"older estimates. For a bounded smooth finite type convex domain $\Omega\subset\mathbb C^n$ that has $q$-type $m_q$ for $1\le q\le n$, our  $\overline\partial$ solution operator $\mathcal H_q$ on $(0,q)$-forms has (fractional) Sobolev boundedness $\mathcal H_q:H^{s,p}\to H^{s+1/m_q,p}$ and H\"older-Zygmund boundedness $\mathcal H_q:\mathscr C^s\to\mathscr C^{s+1/m_q}$ for all $s\in\mathbb R$ and $1<p<\infty$. We also show the $L^p$-boundedness $\mathcal H_q:H^{s,p}\to H^{s,pr_q/(r_q-p)}$ for all $s\in\mathbb R$ and $1<p<r_q$, where $r_q:=(n-q+1)m_q+2q$. % The $L^p$-analogy for strong pseudoconvex domains are also provided.
\end{abstract}
\maketitle

\section{Introduction}

 %Here we use $H^{0,1}=h^1$ for the local Hardy space and $H^{0,\infty}=\bmo$ for the small BMO space.

In this paper, we aim to prove the following:
\begin{thm}\label{MainThm}
Assume $\Omega\subset\C^n$ is a bounded smooth convex domain of finite type. Then there are operators $\Hc_q:\Ss'(\Omega;\wedge^{0,q})\to\Ss'(\Omega;\wedge^{0,q-1})$ that maps $(0,q)$-forms to $(0,q-1)$-forms with distributional coefficients, for $1\le q\le n$ (we set $\Hc_{n+1}:=0$), such that
\begin{enumerate}[(i)]
    \item{\normalfont(Homotopy formula)}\label{Item::MainThm::Homo} $f=\dbar\Hc_q f+\Hc_{q+1}\dbar f$ for all $1\le q\le n$ and $(0,q)$-forms $f\in\Ss'(\Omega;\wedge^{0,q})$. 
\end{enumerate}
Moreover suppose $\Omega$ has $q$-type $m_q$ {\normalfont(see Definition~\ref{Defn::Basis::Type})}. Then $\Hc_q$ has the following boundedness properties:
\begin{enumerate}[(i)]\setcounter{enumi}{1}
\item{\normalfont{(Sobolev and H\"older)}}\label{Item::MainThm::Sob} For every $s\in\R$ and $1<p<\infty$, $\Hc_q:H^{s,p}(\Omega;\wedge^{0,q})\to H^{s+1/{m_q},p}(\Omega;\wedge^{0,q-1})$ and $\Hc_q:\Co^s(\Omega;\wedge^{0,q})\to \Co^{s+1/{m_q}}(\Omega;\wedge^{0,q-1})$.
\item{\normalfont{($L^p$-$L^q$ estimates)}}\label{Item::MainThm::Lp} For every $s\in\R$ and $1<p<r_q$, $\Hc_q:H^{s,p}(\Omega;\wedge^{0,q})\to H^{s,{pr_q}/{(r_q-p)}}(\Omega;\wedge^{0,q-1})$. {\normalfont Here $r_q:=(n-q+1)\cdot m_q+2q$.}
\end{enumerate}
\end{thm}
Here for an open subset $\Omega\subseteq\C^n$, we use $\Ss'(\Omega)$ for the space of extendable complex-valued distributions on $\Omega$ (see Notation~\ref{Note::Space::Dist} and Lemma~\ref{Lem::Space::SsUnion}) and $\Co^\infty(\Omega)$ for the space of all bounded smooth complex functions on $\Omega$ (see Definition~\ref{Defn::Space::Hold}).
For $s\in\R$ and $1<p<\infty$, we use $H^{s,p}(\Omega)$ for the Sobolev-Bessel space and $\Co^s(\Omega)$ for the H\"older-Zygmund space (see Definitions \ref{Defn::Space::Sob} and \ref{Defn::Space::Hold}). When $1<p<\infty$ and $k\ge0$, $H^{k,p}=W^{k,p}$ is the usual Sobolev space; and $\Co^s=C^s$ is the usual H\"older space when $s>0$ is not an integer (see Remark~\ref{Rmk::Space::TLRmk}).

In fact, we obtain a stronger estimate via Triebel-Lizorkin spaces (see Theorem~\ref{StrMainThm}). We also prove the corresponding  $L^p$-$L^q$ estimate for strongly pseudoconvex domain in Section~\ref{Section::SPsiCX}, which is new for negative Sobolev spaces (see Theorem~\ref{Thm::SPsiCX}).

\medskip
For a bounded smooth convex domain $\Omega\subset\C^n$ of finite type $m$, Diederich-Fischer-Forn\ae ss \cite{DFFHolder} constructed a solution operator $H_q$ for the $\dbar$-equation from $(0,q)$ closed forms to $(0,q-1)$-forms that has boundedness $H_q:L^\infty\to C^{1/m}$. In other words $\dbar H_qf=f$ for all $L^\infty$ $\dbar$-closed $(0,q)$-forms $f$ on $\Omega$.

Based on their approach, subsequent authors obtained the following $L^p$ and $C^k$-estimates:
\begin{itemize}
    \item Fischer \cite{FischerLp} proved that $H_q:L^p\cap \ker\dbar\to L^\frac{p(mn+2)}{mn+2-p}$ for $1< p<mn+2$.
    \item Hefer \cite{HeferMultitype} improved the preceding two results \cite{DFFHolder,FischerLp} using multitypes: if $\Omega$ has $q$-type $m_q$, then  $H_q:L^\infty\cap\ker\dbar\to C^{1/m_q}$ and $H_q:L^p\cap \ker\dbar\to L^\frac{p\cdot r_q}{r_q-p}$ for $1< p<(n-q+1)\cdot m_q+2q$, where $r_q=(n-q+1)\cdot m_q+2q$. Note that $m=m_1\ge m_q$, and in general $m_q$ is smaller.
    \item Alexandre \cite{AlexandreCk} modified $H_q$ to a new solution operator $\tilde H_q$ such that $\tilde H_q:C^k\cap\ker\dbar\to C^{k+1/m}$ ($\tilde H_q$ depends on $k$).
\end{itemize}

Our Theorem~\ref{MainThm} implies all the results above. And we have the following remarks:
\begin{itemize}
    \item Our $\Hc_q$ is a solution operator to the Cauchy-Riemann equation on $(0,q)$-forms. When $f$ is a $\dbar$ closed $(0,q)$-form, then the $(0,q-1)$ form $u=\dbar\Hc_qf$ solves $\dbar u=f$. Meanwhile for estimates of $\Hc_q$ we do not require the domains to be the subspace of closed forms, while \cite{DFFHolder,FischerLp,HeferMultitype,AlexandreCk} were stated only on closed forms.
    \item Our estimates on $\Co^s$-spaces imply the one in \cite{DFFHolder,AlexandreCk} because $C^k\subsetneq\Co^k$ for $k\ge1$ and $L^\infty\subsetneq\Co^0$ (see for example \cite[(2.5.7/11)]{TriebelTheoryOfFunctionSpacesI}). For $q\ge2$ our result show the gain of $\frac1{m_q}$ derivative, while in \cite{DFFHolder,AlexandreCk} the gain is only $\frac1m=\frac1{m_1}$.
    \item When $1<p<r_q$, by taking $s=0$, we see that Theorem~\ref{MainThm} \ref{Item::MainThm::Lp} contains the $L^p$-$L^q$ estimate in \cite[Theorem~1.3]{HeferMultitype}. We also have the boundedness $\Hc_q:L^{r_q}\to \BMO$ that recovers \cite[Theorem~1.1~(ii)]{FischerLp}, see Remark~\ref{Rmk::Intro::BMOBdd}.
    \item Even for a negative integer $k$, our operator $\Hc_q$ is defined on the distribution space $H^{k,p}$ and has $\frac1m$ gain $\Hc_q:H^{k,p}\to H^{k+\frac1m,p}$ (in fact, to $H^{k+\frac1{m_q},p}$).
    \item The operator $\Hc_q$ is a ``universal solution operator'', in the sense that we have one operator that has $H^{s,p}$ and  $\Co^s$ boundedness for all $s$, rather than just a bounded range of $s$.
\end{itemize}

The estimates on Sobolev space of negative index were first achieved in
\cite{ShiYaoCk} for the case of a smooth strongly pseudoconvex domain, where for each $1\le q\le n$ we obtained a solution operator with $\frac12$-estimate $H^{s,p}\to H^{s+\frac12,p}$ for all $s\in\R$ and $1<p<\infty$. 

In our precious study \cite{ShiYaoCk} and the current paper, our solution operators are non-canonical because they do not come from the solutions of the $\dbar$-Neumann problem. However, for canonical solutions, it is comparably more difficult to discuss the boundedness (or even well-posedness) on negative function spaces, since we need to a version of generalized trace to discuss boundary value condition, see \cite{NegativeBVPBook}. This was mentioned by \cite{NegativeDbar}.
 
Note that the $L^p$-$L^q$ estimates cannot be directly obtained from the $\frac1m$-estimates, since the classical Sobolev estimate only yields $H^{\frac1m,p}\hookrightarrow L^{\frac{2nmp}{2nm-p}}$, which is a larger space than $L^\frac{p(mn+2)}{mn+2-p}$.

\medskip
Our solution operators follow from the construction of \cite{DFFHolder}. Recall that their solution operator $H_q$ from \cite{DFFHolder} has the form
\begin{equation}\label{Eqn::Intro::DFSolOp}
    H_qf(z):=\int_\Omega B_{q-1}(z,\cdot)\wedge f-\int_{b\Omega}K_{q-1}(z,\cdot)\wedge f.
\end{equation}
The first integral is the Bochner-Martinelli integral operator (see \eqref{Eqn::Goal::DefB} for definition of $B_{q-1}$) which is known to gain one derivative. The second integral is the main term. The construction of $K_{q-1}$ is based on the Diederich-Forn\ae ss support function $S(z,\zeta)$ (see \eqref{Eqn::Goal::S} and \eqref{Eqn::Goal::DefK}). We remark that a slight modification for $K(z,\zeta)$ is required in order to make it a bounded function for each $\zeta$, which is mentioned in \cite{AlexandreCk}. See Lemma~\ref{Lem::Goal::Glue} and Remark~\ref{Rmk::Goal::GlueRmk}.

Our solution operator replaces the boundary integral with integration of the commutator $[\dbar,\Ec]$ on the exterior neighborhood. The commutator was introduced by \cite{PeterCommutator} and later be used by \cite{Michel1991} and recently by \cite{GongHolderSPsiCXC2}:
\begin{equation}\label{Eqn::Intro::CommSolOp}
    \Hc_qf(z):=\int_\Uc B_{q-1}(z,\cdot)\wedge f+\int_{\Uc\backslash\overline\Omega}K_{q-1}(z,\cdot)\wedge[\dbar,\Ec] f,
\end{equation}
where $\Uc$ is a small enough neighborhood of $\overline{\Omega}$, and $\Ec$ is a suitably chosen extension operator of $\Omega$ such that the extended functions are supported in $\Uc$. 

In \cite{LiebRange1980,PeterCommutator} the authors used $\Ec$ for the Seeley's half-space extension \cite{Seeley}, which works only on smooth domains and extend $H^{s,p}$ and $\Co^s$ functions for positive $s$. For the case of non-smooth domains, for example \cite{GongHolderSPsiCXC2}, authors used $\Ec$ for the Stein's extension \cite[Chapter VI]{SteinBook} which is defined on Lipschitz domains and extend $H^{s,p}$ and $\Co^s$ for positive $s$ as well.

In our case we choose $\Ec$ to be the Rychkov extension operator, which works on Lipschitz domains and extends $H^{s,p}$ and $\Co^s$ for \textbf{all} $s$ (including $s<0$), see \eqref{Eqn::Space::ExtOp} and \eqref{Eqn::Space::ExtOmega}. The Rychkov's extension operator was first introduced to solve the $\dbar$-equation in \cite{ShiYaoC2}.

To prove the $\frac1m$-estimates, in \cite{DFFHolder}, also in \cite{FischerLp,HeferMultitype}, the second integral in \eqref{Eqn::Intro::DFSolOp} is defined on the boundary, thus we only need to care about the estimate of the tangential part of $K_{q-1}$ with respect to $\zeta$-variable, which is as follows in our notation (see Definition~\ref{Defn::Goal::BotTop}):
\begin{equation*}
    \int_{b\Omega}K_{q-1}(z,\cdot)\wedge f=\int_{b\Omega}K_{q-1}^\top(z,\cdot)\wedge f.
\end{equation*}

On the other hand, to estimate \eqref{Eqn::Intro::CommSolOp} we need to deal with the normal part $K_{q-1}^\bot=K_{q-1}-K_{q-1}^\top$, which contributes to the major loss of the kernel. Roughly speaking, based on the estimates in \cite[Section~5]{DFFHolder}, $K_{q-1}^\bot$ loses $1$ more derivative than $K_{q-1}^\top$. Alexandre \cite{AlexandreCk} gave a better control and showed that at most $\frac12-\frac1m$ derivative is lost.

In this paper we introduce the following decomposition (see Notation~\ref{Note::Goal::KTopBot} and \eqref{Eqn::Goal::TopBotFacts::BotTopfg}), which simplifies Alexandre's approach:
\begin{equation}\label{Eqn::Intro::TopBotDecomp}
    \int_{\Uc\backslash\overline\Omega}K_{q-1}(z,\cdot)\wedge[\dbar,\Ec] f=\int_{\Uc\backslash\overline\Omega}K_{q-1}^\top(z,\cdot)\wedge[\dbar,\Ec] f+\int_{\Uc\backslash\overline\Omega}K_{q-1}^\bot(z,\cdot)\wedge([\dbar,\Ec] f)^\top.
\end{equation}

$[\dbar,\Ec]f$ has one derivative less than $f$, and the $\frac1m$-estimate of $K_{q-1}^\top(z,\cdot)\wedge[\dbar,\Ec]f$ essentially follows from \cite{DFFHolder}. Although $K_{q-1}^\bot$ lose one derivative than $K_{q-1}^\top$, the tangential part of the commutator\\ $[\dbar,\Ec]^\top f:=([\dbar,\Ec]f)^\top$ has the same regularity to $f$ which compensates the estimate that we need, see Proposition~\ref{Prop::LPThm::TangComm} and Remark~\ref{Rmk::LPThm::TangCommRmk} \ref{Item::LPThm::TangCommRmk::FBdd}.

Note that \eqref{Eqn::Intro::TopBotDecomp} is not needed in the case of strongly pseudoconvex domains because for the Leray map $\widehat Q(z,\zeta)$ (see Proposition~\ref{Prop::SPsiCX::LeviMap}) we only need the trivial estimates \eqref{Eqn::SPsiCX::EstQ} and \eqref{Eqn::SPsiCX::EstDK}. See Remark~\ref{Rmk::Basis::EstQ2::SPsiCX}.

For the case $s<1$ in Theorem~\ref{MainThm}, the commutator $[\dbar,\Ec]f$ may give a distribution rather than a classical function. In order to ensure the integral operators make sense, we express the given forms as the derivatives of functions with positive index. For $k\ge1$ we constructed the anti-derivative operators $\{\Sc^{k,\alpha}\}_{|\alpha|\le k}$ in \cite{ShiYaoExt} such that if a function $g$ is supported outside $\Omega$ then $g=\sum_{|\alpha|\le k}D^\alpha\Sc^{k,\alpha}g$ with all summand supported outside $\Omega$ as well. See Proposition \ref{Prop::Space::ATDOp}. Therefore, by integration by parts
\begin{align*}
    \int_{\Uc\backslash\overline\Omega}K_{q-1}^{(\top,\bot)}(z,\zeta)\wedge[\dbar,\Ec]^{(\top)}f(\zeta)d\Vol_\zeta=&\sum_{|\alpha|\le k}\int_{\Uc\backslash\overline\Omega}K_{q-1}^{(\top,\bot)}(z,\zeta)\wedge\big(D^\alpha \Sc^{k,\alpha}\circ[\dbar,\Ec]^{(\top)}f\big)(\zeta)d\Vol_\zeta
    \\
    =&\sum_{|\alpha|\le k}(-1)^{|\alpha|}\int_{\Uc\backslash\overline\Omega}D^\alpha_\zeta K_{q-1}^{(\top,\bot)}(z,\zeta)\wedge  \big(\Sc^{k,\alpha}\circ[\dbar,\Ec]^{(\top)}f\big)(\zeta)d\Vol_\zeta.
\end{align*}
The method of trading derivatives between $K_{q-1}$ and $[\dbar,\Ec]f$ was introduced by \cite{ShiYaoCk} for the estimates of strongly pseudoconvex domains.

To key to prove Theorem~\ref{MainThm} is to do the weighted estimates for $D^k_{z,\zeta}(K^\top_{q-1})(z,\zeta)$ and $D^k_{z,\zeta}(K^\bot_{q-1})(z,\zeta)$. See Theorem~\ref{Thm::WeiEst}. Note that we take derivatives after we take ($\bot$ and $\top$) projections. The reduction from Theorem~\ref{MainThm} to Theorem~\ref{Thm::WeiEst} is achieved by using the Hardy-Littlewood lemma, see Proposition~\ref{Prop::LPThm::HLLem} and Corollary~\ref{Cor::LPThm::BddDom}~\ref{Item::LPThm::BddDom::TangComm}.

In fact, by combining Theorem~\ref{Thm::WeiEst} and Corollary~\ref{Cor::LPThm::BddDom} \ref{Item::LPThm::BddDom::TangComm} we have a stronger estimate of $\Hc_q$ in terms of Triebel-Lizorkin spaces (see Definition~\ref{Defn::Space::TLSpace}):
\begin{thm}\label{StrMainThm}With $\Hc_q$ as in Theorem~\ref{MainThm}, the following boundedness properties hold for $1\le q\le n-1$:
    \begin{align}
        &\Hc_q:\Fs_{p,\infty}^s(\Omega;\wedge^{0,q})\to\Fs_{p,\eps}^{s+\frac1{m_q}}(\Omega;\wedge^{0,q-1}),&\forall\ \eps>0,\quad 1\le p\le\infty; \\
        \label{Eqn::StrMainThm::Lp}
        &\Hc_q:\Fs_{p,\infty}^s(\Omega;\wedge^{0,q})\to\Fs_{\frac{pr_q}{r_q-p},\eps}^s(\Omega;\wedge^{0,q-1}),&\forall\ \eps>0,\quad 1\le p\le r_q.
    \end{align}
\end{thm}
Theorem~\ref{StrMainThm} implies Theorem~\ref{MainThm} \ref{Item::MainThm::Sob} and \ref{Item::MainThm::Lp} automatically for the case $1\le q\le n-1$, see Remark~\ref{Rmk::Space::TLRmk}.

\begin{rem}[Boundedness on Besov spaces]Theorem~\ref{StrMainThm} implies the $\frac1{m_q}$-estimate and higher order $L^p$-$L^q$ estimates on Besov spaces via real interpolations:

By the elementary embedding (see Remark~\ref{Rmk::Space::TLRmk} \ref{Item::Space::TLRmk::Embed}), for every  $s\in\R$ and $t\in(0,\infty]$, we have $\Hc_q:\Fs_{p,t}^s\to\Fs_{p,t}^{s+1/{m_q}}$ for $p\in[1,\infty]$ and $\Hc_q:\Fs_{p,t}^s\to\Fs_{\frac{pr_q}{r_q-p},t}^s$ for $p\in[1,r_q]$. On the other hand we have real interpolations (see for example \cite[Corollary~1.111]{TriebelTheoryOfFunctionSpacesIII}): 
\begin{align*}
    (\Fs_{p,t_0}^{s_0}(\Omega),\Fs_{p,t_1}^{s_1}(\Omega))_{\theta,t}=\Bs_{p,t}^{\theta s_1+(1-\theta)s_0}(\Omega),&\quad\forall p\in[1,\infty),\ t_0,t_1,t\in(0,\infty],\ \theta\in(0,1)\text{ and } s_0\neq s_1;
    \\
    (\Fs_{\infty,\infty}^{s_0}(\Omega),\Fs_{\infty,\infty}^{s_1}(\Omega))_{\theta,t}=\Bs_{\infty,t}^{\theta s_1+(1-\theta)s_0}(\Omega),&\quad\forall t\in(0,\infty],\ \theta\in(0,1)\text{ and } s_0\neq s_1.
\end{align*}

See \cite[(1.368) and (1.369)]{TriebelTheoryOfFunctionSpacesIII}. Therefore (see \cite[Definition~1.2.2/2 and Theorem~1.3.3]{TriebelInterpolation} for example) for every  $s\in\R$ and $t\in(0,\infty]$, we have
$\Hc_q:\Bs_{p,t}^s(\Omega;\wedge^{0,q})\to\Bs_{p,t}^{s+1/m_q}(\Omega;\wedge^{0,q-1})$ for $p\in[1,\infty]$, and $\Hc_q:\Bs_{p,t}^s(\Omega;\wedge^{0,q})\to\Bs_{\frac{pr_q}{r_q-p},t}^s(\Omega;\wedge^{0,q-1})$ for $p\in[1,r_q]$.
\end{rem}

\begin{rem}[Boundedness on BMO]\label{Rmk::Intro::BMOBdd}
As a special case of \eqref{Eqn::StrMainThm::Lp} we recover the endpoint $L^p$-$L^q$ boundedness $\Hc_q:L^{r_1}(\Omega;\wedge^{0,q})\to\BMO(\Omega;\wedge^{0,q-1})$ from \cite[Theorem~1.1 (ii)]{FischerLp} (cf. \cite[Theorem~1.3]{HeferMultitype}). The definition of $\BMO(\Omega)$ used by Fischer \cite{FischerLp} comes from \cite[Section~4, Definition~3]{McNealStein}. We recall that for an arbitrary open subset $U\subset\R^N$, $\BMO(U)$ and $\bmo(U)$ (see \cite[Definition 1.2]{ChangHardy}) are spaces consisting of $f\in L^1_\loc(U)$ such that: 
\begin{gather*}
    \|f\|_{\BMO(U)}:=\sup_{B\subseteq U}\frac1{|B|}\int_{B}\big|f-{\textstyle\frac1{|B|}\int_{B}f}\big|<\infty,\quad
    \|f\|_{\bmo(U)}:=\|f\|_{\BMO(U)}+\sup_{B\subseteq U}\frac{1}{|B|}\int_{B}|f|<\infty.
\end{gather*}
Here $B$ denotes the balls in $\R^N$.

Clearly, $\bmo(U)\subset L^1_\loc(U)$, while $\BMO(U)=\bmo(U)/\{c\cdot\1_U:c\in\C\}$ ignores the constant functions.

By \cite[Theorem~1.4]{ChangHardy} (since $\Omega$ is bounded smooth), we have $\bmo(\Omega)=\{\tilde f|_\Omega:\tilde f\in\bmo(\C^n)\}$, and by \cite[Theorem~2.5.8/2]{TriebelTheoryOfFunctionSpacesI} we have $\bmo(\C^n)=\Fs_{\infty2}^0(\C^n)$. Therefore, by Definition \ref{Defn::Space::TLSpace} for spaces on domains, we get $\bmo(\Omega)=\Fs_{\infty2}^0(\Omega)$.

In addition, by Remark~\ref{Rmk::Space::TLRmk} \ref{Item::Space::TLRmk::Embed} and \ref{Item::Space::TLRmk::SobHold} we have $\Fs_{\infty,\eps}^0\subset\Fs_{\infty,2}^0$ and $\Fs_{r_q,2}^0=L^{r_q}\subset \Fs_{r_q,\infty}^0$. Therefore, we obtain the boundedness $\Hc_q:L^{r_q}(\Omega;\wedge^{0,q})\to\bmo(\Omega;\wedge^{0,q-1})$. By taking the quotient of constant functions we obtain a stronger one $\Hc_q:L^{r_q}(\Omega;\wedge^{0,q})\to\BMO(\Omega;\wedge^{0,q-1})$. (Recall that $r_1\ge r_q$ from Theorem \ref{MainThm} \ref{Item::MainThm::Lp} since $m_1\ge m_q$.)
\end{rem}

\medskip
The estimates for the $\dbar$-equation is a fundamental question in several complex variables. There are two major approaches. The first one is the \textit{$\dbar$-Neumann problems} which defines the \textit{canonical solutions}. This was proposed in \cite{GarabedianSpencer1952}. The $\dbar$-Neumann problem was proposed in \cite{GarabedianSpencer1952} and developed by Morrey \cite{Morrey1958} and Kohn \cite{Kohn1963}. We refer \cite{ChenShawBook} for a detailed discussion. Note that the canonical solutions is the solution with minimal $L^2$-norm, which follows H\"ormander's $L^2$-estimates \cite{Hormander1965}.

We use the second approach called \textit{integral representations}, which yield non-canonical solutions but the expressions can be more explicit. The method was introduced to the $\dbar$-equation by Henkin \cite{Henkin1969} and Grauert \& Lieb \cite{GrauertLieb1971} in the study of strongly pseudoconvex domains. We refer \cite{RangeSCVBook} and \cite{LiebMichelBook} for a general discussion.

We briefly review the estimates for convex domains of finite type below. Studies in complex or real pseudo-ellipsoids were conducted by Range \cite{Range1976}, Diederich-Forn\ae ss-Wiegerinck \cite{DFW1986}, Chen-Krantz-Ma \cite{ChenKrantzMa} and Fleron \cite{Fleron1996}, and in the domain of real-analytic boundaries by Range \cite{Range1977} and Bruna-Castillo \cite{BrunaCastillo}. These are all special cases to general convex domains of finite type. The $\frac1m$-regularity was shown to be optimal by \cite{ChenKrantzMa}.

For the type conditions in convex domains, McNeal \cite{McNealFiniteType} introduced the \textit{$\eps$-extremal basis} and showed the equivalence between the line type and D'Angelo 1-type on convex domains (also see \cite{BoasStraubeFiniteType} for a short proof), and it was later used to show the boundedness of $\dbar$-Neumann solutions in \cite{McNealStein}. McNeal's approach was used by Cumenge \cite{Cumenge1997,Cumenge2001} and Wang \cite{WangStrictType} to obtain estimates of the $\dbar$-equation. For the type where $q\ge2$, Yu \cite{YuMultitype} introduced a different basis from McNeal, called the \textit{$\eps$-minimal basis}, and showed the equivalence of the line $q$-type, D'Angelo $q$-type and Catlin's $q$-type. See also \cite{HeferBases} for the connections between McNeal's $\eps$-extremal basis and Yu's $\eps$-minimal basis.

As mentioned in the beginning, the solution operators on convex domains of finite type are mainly derived from Diederich-Fischer-Forn\ae ss \cite{DFFHolder} with the holomorphic supported function constructed in \cite{DFSupport}. The H\"older estimate $L^\infty\to C^{\frac1m}$ was obtained \cite{DFFHolder}, and the anisotropic version was later obtained by Fischer \cite{FischerNonisotropic} and Diederich-Fischer \cite{DiederichFischerNonisotropic} on lineally convex domains of finite type. The $L^p$-estimate was first obtained by Fischer \cite{FischerLp} and later there were partial progress by \cite{AhnCho,Ahn} and by \cite{CharpentierDupain2018}. The $C^k\to C^{k+\frac1m}$ estimate was achieved by Alexandre \cite{AlexandreCk}. The notion of multitype was used by Hefer \cite{HeferMultitype} and showed that on $(0,q)$-forms the $\frac1m$-estimate could be improved to the $\frac1{m_q}$-estimate automatically if one consider the multitype of the domain, see also \cite{Alexandre2011}. 

For the convex domains of infinite type, there are studies by Range \cite{Range1992}, Forn\ae ss-Lee-Zhang \cite{FornaessLeeZhang}, Ha-Khanh-Raich \cite{HaKhanhRaich} and recently by Ha
\cite{Ha2019,Ha2021}. Some of their constructions also used integral representations. It should be possible to improve their results using the Rychkov's extension operator, as applied in the present study.

For general finite type domains that is not necessarily convex, it is known that in $\C^2$ one can have optimal $\frac1m$-estimate in general, see \cite{FeffermanKohn1988} and \cite{ChangeNagelStein1992}. See also \cite{Range1990} for an approach using integral representations.
In higher dimensions Catlin \cite{Catlin83,Catlin84,Catlin87} showed that the canonical solution has boundedness $L^2\to H^{\eps,2}$ for some $\eps>0$ if and only if the domain has finite D'Angelo type. The general lower bounds for $\eps$ with respect to the type $m$ remain highly unknown. For more discussions between finite types and subelliptic estimates we refer the reader to the survey by \cite{DAngeloKohnFiniteTypeSurvey}.

% \und{More the below later }
% \gra{Here our result do not give the $L^1$-Sobolev estimate: it is not known to the author whether $\Hc_q:W^{k,1}\to W^{k+\frac1{m_q},1}$ or $\Hc_q:W^{k,1}\to W^{k,\frac{r_q}{r_q-1}}$ itself has boundedness for $k\ge0$. Nevertheless, this is true for a (possible) different homotopy operator $\widetilde\Hc_q$, see Proposition~??. However $\widetilde\Hc_q$ may not be able to defined for general distributions, see Remark~??.}

\medskip
The paper is organized as follows. In Section~\ref{Section::Goal}, we recall the construction of Diederich-Forn\ae ness support function and the corresponding integral kernel, and we introduce the tangential and vertical projections for  $d\bar\zeta$-forms. In Section~\ref{Section::Basis} we review the $\eps$-minimal basis and prove Theorem~\ref{Thm::WeiEst}. In Section~\ref{Section::Space}, we summarize the properties of function spaces and Rychkov's construction of extension operator. In Section~\ref{Section::LPThm} we prove the boundedness of tangential commutator, Proposition~\ref{Prop::LPThm::TangComm}, and strong Hardy-Littlewood lemma, Proposition~\ref{Prop::LPThm::HLLem}. In Section~\ref{Section::PfThm}, we complete the proof of Theorems \ref{MainThm} and \ref{StrMainThm} using Theorem~\ref{Thm::WeiEst} and Corollary~\ref{Cor::LPThm::BddDom}. In Section~\ref{Section::SPsiCX}, we apply the proof technics of  Theorems \ref{MainThm} and \ref{StrMainThm} to the case of strongly pseudoconvex domains and prove Theorem~\ref{Thm::SPsiCX}.

\smallskip
In the following we use $\N=\{0,1,2,\dots\}$ as the set of non-negative integers. 

On a complex coordinate system $(z_1,\dots,z_n)$, $\partial^\alpha_z$ denotes the derivative on holomorphic part $\frac{\partial^{|\alpha|}}{\partial z_1^{\alpha_1}\dots\partial z_n^{\alpha_n}}$ where $\alpha\in\N^n$, and $D^\beta_z$ denotes the total derivative $\frac{\partial^{|\beta|}}{\partial z_1^{\beta_1}\dots\partial z_n^{\beta_n}\partial\bar z_1^{\beta_{n+1}}\dots\partial\bar z_n^{\beta_{2n}}} $ where $\beta\in\N^{2n}$.

We use the notation $x \lesssim y$ to denote that $x \leq Cy$, where $C$ is a constant independent of $x,y$, and $x \approx y$ for ``$x \lesssim y$ and $y \lesssim x$''. We use $x\lesssim_\eps y$ to emphasize the dependence of $C$ on the parameter $\eps$.

For a function class $\Xs$ and a domain $U$, we use $\Xs(U)=\Xs(U;\C)$ as the space of complex-valued functions in $U$ that have regularity $\Xs$. We use $\Xs(U;\R)$ if the functions are restricted to being real-valued. We use $\Xs(U;\wedge^{p,q})$ for the space of (complex-valued) $(p,q)$-forms on $U$ that have regularity $\Xs$.

In the following, $U_1=\{-T_1<\varrho<T_1\}$ denotes a fixed neighborhood of $b\Omega$ (see Lemma~\ref{Lem::Goal::Glue}).

\subsection*{Acknowledgement} The author would like to thank Xianghong Gong and Kenneth Koenig for the valuable supports and comments.

\section{The Construction of Homotopy Formulas}\label{Section::Goal}

Let $\Omega\subset\C^n$ be a smooth convex domain that has finite type $m$. We fix a defining function $\varrho\in C^\infty(\C^n;\R)$ of $\Omega$ (that is $\Omega=\{\varrho<0\}$ and $\nabla\varrho(\zeta)\neq0$ for all $\zeta\in b\Omega$) such that:
\begin{enumerate}[label=(\thesection.\arabic*)]
    \item\label{Item::Goal::DefFunAssump} There is a $T_0>0$, such that for every $-T_0<t<T_0$, the domain $\Omega_t:=\{\zeta:\varrho(\zeta)<t\}$ is convex and has the same complex affine $q$-type (see Definition~\ref{Defn::Basis::Type}) to $\Omega=\Omega_0$, for all $1\le q\le n$. \setcounter{equation}{\value{enumi}}
\end{enumerate}

This can be achieved by assuming $0\in\Omega$ (which can be done by passing to a translation) and require $\varrho$ to have homogeneity condition (see also \cite[(2.1)]{HeferMultitype}):
\begin{equation}\label{Eqn::Goal::DefFun}
    \varrho(\lambda\zeta)+1=\lambda(\varrho(\zeta)+1)=\lambda,\quad\text{ for all }\zeta\in b\Omega\text{ and all }\lambda\in\R_+\text{ closed to }1.
\end{equation}

In this setting $\Omega_t$ is simply the dilation of $\Omega$, which in particular shares the same (line, D'Angelo, or Catlin) type conditions.

We let $U_0:=\{\zeta:|\varrho(\zeta)|<T_0\}$ be a corresponding open neighborhood of $b\Omega$.
%  \begin{equation}\label{Eqn::Goal::AssumpU}
%      U_0\Subset 2\Omega\quad\text{and}\quad U_0\cap\tfrac12\Omega=\varnothing.
%  \end{equation}
 
% In practice we can view $U_0$ as the cylinder $(b\Omega)\times (-T_0,T_0)$ if we impose the assumption \eqref{Eqn::Goal::DefFun}. See the proof of ?? in ??.

% have $|\nabla\varrho|\big|_{b\Omega}\ge1$ since we have assumed $0\in\Omega\subset\{|z|<1\}$. Thus there is a neighborhood $U\subset\C^n$ of $b\Omega$ such that

We recall the Diederich-Forn\ae ss holomorphic support function $S\in C^\infty(\C^n\times U_0;\C)$ from \cite{DFSupport}:

Fix suitably large constants $M_1,M_2,M_3>1$. For each $\zeta\in U_0$ we take a unitary matrix $\Phi(\zeta)\in \C^{n\times n}$ that is locally defined and smoothly depend on $\zeta$,
%\blu{G: It is good that you give a warning here, but locally $\Phi(\zeta)$ is smooth and this may be needed.}\bro{Yes I add the smoothness dependence}
%(there is no continuity dependence assumption on $\zeta$)
such that $\Phi(\zeta)\frac{\dbar\varrho(\zeta)}{|\dbar\varrho(\zeta)|}=[1,0,\dots,0]^\intercal$. We define\footnote{For a complex matrix $A$, we use $A^\dagger=\overline A^\intercal$ for the conjugate transpose. Thus $A^\dagger=A^{-1}$ when $A$ is unitary.} for $\zeta\in U_0$ and $w=[w_1,\dots,w_n]^\intercal\in \C^n(\simeq\C^{n\times 1})$:
\begin{gather}
    % \label{Eqn::Goal::RhoZeta}
    % \varrho_\zeta(w):=\varrho(\zeta+\Phi(\zeta)\cdot w);
    % \\
    \label{Eqn::Goal::SZeta}
    S^{\Phi(\zeta)}_\zeta(\omega):=3\omega_1+M_1\omega_1^2-\tfrac1{M_2}\sum_{j=1}^{m/2}M_3^{4^j}(-1)^j\sum_{|\alpha|=2j;\alpha_1=0}\frac{\partial^\alpha \varrho(\zeta+\Phi(\zeta)^\dagger\cdot w)}{\partial w^\alpha}\bigg|_{w=0}\cdot\frac{\omega^\alpha}{\alpha!};
    \\
    % \label{Eqn::Goal::QZeta}
    % Q^1_\zeta(w):=3+M_1w_1;\quad Q^k_\zeta(w):=-\tfrac1{M_2}\sum_{j=1}^{m/2}M_3^{4^j}(-1)^j\sum_{|\alpha|=2j;\alpha_1=0}\frac{\partial^\alpha \varrho_\zeta}{\partial w^\alpha}(0)\cdot\frac{\alpha_k}{2j\cdot\alpha!}\cdot\frac{w^\alpha}{w_k},\quad1\le k\le n;
    % \\
    \label{Eqn::Goal::S}
    S(z,\zeta):=S^{\Phi(\zeta)}_\zeta\big(\Phi(\zeta)(z-\zeta)\big)
    % ,\quad Q(z,\zeta):=\Phi(\zeta)^\dagger\cdot Q_\zeta\big(\Phi(\zeta)\cdot(z-\zeta)\big)
    \quad z\in\Omega.
\end{gather}

\begin{lem}The $S(z,\zeta)$ in \eqref{Eqn::Goal::S} with suitable constants $M_1,M_2,M_3>0$ satisfies the following:
\begin{enumerate}[(i)]
    \item {\normalfont(\cite[Lemma~2.1]{DFFHolder})} $S(z,\zeta)$ is a smooth function, holomorphic in $z$, and does not depend on the choice of the family $\{\Phi(\zeta):\zeta\in U_0\}$.
    \item{\normalfont(\cite[Corollary~2.4]{DFSupport} and \cite[Theorem~2.1]{FischerLp})} There is an $M_4>1$ such that
    \begin{equation}\label{Eqn::Goal::SBdd}
        \re S(z,\zeta)\le M_4\cdot\max(0,\varrho(z)-\varrho(\zeta))-\tfrac1{M_4}|z-\zeta|^m,\quad\forall \zeta\in U_0,\quad z\in \Omega\cup U_0\text{ such that }|z-\zeta|<\tfrac1{M_4}.
    \end{equation}
\end{enumerate}
\end{lem}
% Here \cite[Theorem~2.3]{DFSupport} and \cite[Theorem~4.2]{HeferMultitype} shows that there are $0<c<1<C$ such that $\re S(z,\zeta)\le C\max(0,\varrho(z)-\varrho(\zeta))-c|z-\zeta|^m$ for all $\zeta\in b\Omega$ and $z\in\Omega=\Omega_0$. The homogeneity of $\varrho+1$ from \eqref{Eqn::Goal::DefFun} ensures that for suitable $M_4>\frac1c$ (ii) is true.

As mentioned in \cite{AlexandreCk}, $S(z,\zeta)$ may have zeroes in $\big(\Omega\times (U_0\backslash\overline\Omega)\big)\cap\{|z-\zeta|\ge\frac1{M_4}\}$. We can take the following standard modification.
\begin{lem}\label{Lem::Goal::Glue}
%\blu{G: Why is $S$ smooth in $\zeta$?}\bro{I put the smoothness assumption on top.}
Let $S\in C^\infty_\loc(\C^n\times U_0;\C)$ be as in \eqref{Eqn::Goal::S}. There are a $T_1\in(0,T_0]$ associated with the neighborhood $U_1:=\{\zeta:|\varrho(\zeta)|<T_1\}$ of $b\Omega$, a constant $M_5>1$ and a $\widehat S\in \Co^\infty(\Omega\times U_1;\C)$ such that
\begin{enumerate}[(i)]
    \item $\widehat S(\cdot,\zeta)$ is holomorphic in $z\in\Omega$ for all $\zeta\in U_1$.
    \item\label{Item::Goal::Glue::2} $|\widehat S(z,\zeta)|\ge\frac1{M_5}$ for all $(z,\zeta)\in \Omega\times(U_1\backslash\overline\Omega)$ such that $|z-\zeta|\ge\frac1{2M_4}$.
    \item\label{Item::Goal::Glue::A} There is a $A\in \Co^\infty(\Omega\times U_1;\C)$ such that $\widehat S(z,\zeta)=A(z,\zeta)\cdot S(z,\zeta)$ and $\frac1{M_5}\le |A(z,\zeta)|\le M_5$ for all $(z,\zeta)\in \Omega\times(U_1\backslash\overline\Omega)$ such that $|z-\zeta|\le\frac1{2M_4}$.
\end{enumerate}
\end{lem}
\begin{rem}\label{Rmk::Goal::GlueRmk}
Lemma~\ref{Lem::Goal::Glue} was not mentioned in \cite{DFFHolder}, which might make a gap in estimating the last integral in \cite[Section~6]{DFFHolder}. There is a different modification $\widehat S(z,\zeta)$ in \cite[Section~6]{HeferMultitype} but it may not work in our situation.
\end{rem}
\begin{proof}We use the same construction from \cite[Theorem~2.4.3]{HenkinLeitererBook}.

Let $\delta_1:=\min\big(T_0,(2M_4)^{-m-2}(1+\|\nabla\varrho\|_{L^\infty(U_0)})^{-1}\big)\in(0,1)$. By \eqref{Eqn::Goal::SBdd}  we see that
%\blu{G: You don't want to use the absolute value to show $\log$ is well-defined.} \bro{I change it to $\re$}
\begin{equation*}
    -\re S(z,\zeta)>\delta_1,\quad\text{whenever }\varrho(z),\varrho(\zeta)\in(-\delta_1,\delta_1)\text{ and }\tfrac1{2M_4}\le|z-\zeta|\le\tfrac1{M_4}.
\end{equation*}

Let $\chi_1\in C_c^\infty((-\delta_1,\delta_1);[0,1])$ %and $\chi_2\in C^\infty(\R;[0,1])$
be such that $\chi_1\big|_{[-\frac12\delta_1,\frac12\delta_1]}\equiv1$.  %$\chi_2\big|_{(-\infty,\frac12\delta_1]}\equiv1$ and $\chi_2\big|_{[\delta_1,\infty)}\equiv0$.
Let $U_1':=\{\zeta:|\varrho(\zeta)|<\delta_1\}$, we define a $(0,1)$-form $f(z,\zeta)=\sum_{j=1}^nf_j(z,\zeta)d\bar z_j$ for $z\in\Omega_{\delta_1}=\{\varrho<\delta_1\}$ and $\zeta\in U_1'$ by
\begin{equation*}
    f(z,\zeta):=\begin{cases}\dbar_z\big(\chi_1(|z-\zeta|)\cdot\log(- S(z,\zeta))\big),&\text{if }\tfrac1{2M_4}\le|z-\zeta|\le\tfrac1{M_4}\\0,&\text{otherwise}\end{cases}.
\end{equation*}
Since $S(z,\zeta)$ is smooth and holomorphic in $z$, we see that $f$ is bounded 
%\blu{G: Check if $f$ is  smooth in $\zeta$. }\bro{Mentioned in the previous comment}
smooth in the domain $\Omega_{\delta_1}\times U_1'$, and $f(\cdot,\zeta)$ is $\dbar$-closed for each $\zeta\in U_1'$.

Therefore, either by applying \cite[Theorem~11.2.7 and Lemma~11.2.6]{ChenShawBook} since $\Omega_{\delta_1}$ is convex, or by applying \cite[Theorem~2.3.5]{HenkinLeitererBook} since we can find a strongly convex domain $\widetilde\Omega$ such that $\Omega_{\frac12\delta_1}\subset \widetilde\Omega\subset\Omega_{\delta_1}$, there exists a continuous solution operator $T:\Co^\infty(\Omega_{\delta_1};\wedge^{0,1})\cap\ker\dbar\to \Co^\infty(\Omega_{\frac12\delta_1})$ such that $\dbar Tg=g$ in $\Omega_{\frac12\delta_1}$ for every bounded smooth $\dbar$-closed form $g$ in $\Omega_{\delta_1}$.

Now for $z\in\Omega_{\frac12\delta_1}$ and $\zeta\in U_1'$, we define
\begin{gather*}
    u(z,\zeta):=(Tf(\cdot,\zeta))(z);\qquad A(z,\zeta):=\exp(-u(z,\zeta));
    \\
    \widehat S(z,\zeta):=\begin{cases}A(z,\zeta)S(z,\zeta),&\text{if }|z-\zeta|\le\tfrac1{2M_4}
    \\
    -\exp\big(\chi_1(|z-\zeta|)\log(-S(z,\zeta))-u(z,\zeta)\big),&\text{if }|z-\zeta|\ge\tfrac1{2M_4}
    \end{cases}.
\end{gather*}
We see that $\widehat S:\Omega_{\frac12\delta_1}\times U_1'\to\C$ is holomorphic in $z$ and is bounded from below in $\{|z-\zeta|\ge\frac1{2M_4}\}$. 

Taking $T_1:=\frac12\delta_1$, $U_1:=\{|\varrho|<T_1\}$ and $M_5:=\max\Big(\frac1{\delta_1}\cdot\exp\big(\sup\limits_{\Omega_{T_1}\times U_1}u\big),\|S\|_{L^\infty(\Omega_{T_1}\times U_1)}\cdot\exp\big(\sup\limits_{\Omega_{T_1}\times U_1}(-u)\big)\Big)$, we get the estimates in \ref{Item::Goal::Glue::2} and \ref{Item::Goal::Glue::A}, completing the proof.
\end{proof}
% In particular for $S(z,\zeta)$ given in \eqref{Eqn::Goal::S}, $\widehat S$ is holomorphic in $z$ and $\widehat S(z,\zeta)\neq0$ for all $(z,\zeta)\in \Omega\times(U_0\backslash\overline\Omega)$.

We now use $\widehat S(z,\zeta)$ to define the corresponding Leray map $\widehat Q=(\widehat Q_1,\dots,\widehat Q_n)\in C^\infty(\Omega\times U_1;\C^n)$ by the following: for $\zeta\in U_1$,
\begin{equation}\label{Eqn::Goal::HatQ}
    \begin{aligned}
    \widehat S^{\Phi(\zeta)}_\zeta(w)&:=\widehat S(\zeta+\Phi(\zeta)^\dagger\cdot w,\zeta);& \widehat Q_{\zeta,j}^{\Phi(\zeta)}(w)&\textstyle:=\int_0^1\frac{\partial\widehat S^{\Phi(\zeta)}_\zeta}{\partial w_j}(tw)dt,\quad 1\le j\le n;
    \\
    \widehat Q_\zeta^{\Phi(\zeta)}(w)&:=[\widehat Q_{\zeta,1}^{\Phi(\zeta)}(w),\dots,\widehat Q_{\zeta,n}^{\Phi(\zeta)}(w)]^\intercal;&\widehat Q(z,\zeta)&:=\Phi(\zeta)^\intercal\cdot \widehat Q_\zeta^{\Phi(\zeta)}\big(\Phi(\zeta)\cdot(z-\zeta)\big).
    \end{aligned}
\end{equation}
By the same argument to \cite[Lemma~2.1]{DFSupport} $\widehat Q$ does not depend on the choice of unitary maps $\{\Phi(\zeta)\}$. In fact we have $\widehat Q_j(z,\zeta)=\int_0^1\frac{\partial\widehat S}{\partial z_j}(\zeta+t(z-\zeta),\zeta)dt$, and thus $\widehat S(z,\zeta)=\sum_{j=1}^n\widehat Q_j(z,\zeta)\cdot(z_j-\zeta_j)$.

We now identify the vector-valued function $\widehat Q(z,\zeta)$ with the 1-form $\sum_{j=1}^n\widehat Q_j(z,\zeta)d\zeta_j$ and we denote $b(z,\zeta):=\sum_{j=1}^n(\bar\zeta_j-\bar z_j)d\zeta_j$. The following notations of differential forms on $(z,\zeta)\in \Omega\times U_1$ are adapted from \cite{ChenShawBook}: in the following $\dbar=\dbar_{z,\zeta}$,
\begin{gather}
    \label{Eqn::Goal::DefB}
    B(z,\zeta):=\frac{b\wedge(\dbar b)^{n-1}}{(2\pi i)^n|z-\zeta|^{2n}}=:\sum_{q=0}^{n-1}B_q(z,\zeta);
    \\
    \label{Eqn::Goal::DefK}
    K(z,\zeta):=\frac{ b\wedge \widehat Q}{(2\pi i)^n}\wedge\sum_{k=1}^{n-1}(-1)^k\frac{(\dbar b)^{n-1-k}\wedge(\dbar \widehat Q)^{k-1}}{|z-\zeta|^{2(n-k)}\widehat S^k}=:\sum_{q=0}^{n-2}K_q(z,\zeta).
\end{gather}
Here $B$ is a $(n,n-1)$ form where $B_q$ is the component  that has degree $(0,q)$ in $z$ and $(n,n-1-q)$ in $\zeta$; $K$ is a $(n,n-2)$ form where $K_q$ is the component that has degree $(0,q)$ in $z$ and $(n,n-2-q)$ in $\zeta$.

\begin{lem}\label{Lem::Goal::HomotopyFormula}
% Let $B(z,\zeta)$ and $K(z,\zeta)$ be as in \eqref{Eqn::Goal::DefB} and \eqref{Eqn::Goal::DefK}. Then
% \begin{equation}\label{Eqn::Goal::HomotopyOP}
%     H_qf(z):=\int_{\Omega} B_{q-1}(z,\cdot)\wedge f-\int_{b\Omega}K_{q-1}(z,\cdot)\wedge f,\quad 1\le q\le n,\quad f\in \Co^\infty(\Omega;\wedge^{0,q}),\quad z\in\Omega.
% \end{equation}
% are pointwise defined, such that $f=\dbar H_qf+H_{q+1}\dbar f$ for all $f\in \Co^\infty(\Omega;\wedge^{0,q})$.

Let $\Ec:\Co^\infty(\Omega)\to C^1_c(\Omega\cup U_1)$ be an extension operator such that $\supp \Ec f\Subset \Omega\cup U_1$ for all functions $f\in\Co^\infty(\Omega)$. Then the following integral is pointwisely defined:
\begin{equation}\label{Eqn::Goal::HomotopyOP}
    \Hc_qf(z):=\int_{\Omega\cup U_1} B_{q-1}(z,\cdot)\wedge \Ec f+\int_{U_1\backslash\overline\Omega}K_{q-1}(z,\cdot)\wedge[\dbar,\Ec]f,\quad 1\le q\le n,\quad f\in \Co^\infty(\Omega;\wedge^{0,q}),\quad z\in\Omega.
\end{equation}
Moreover $f=\dbar H_qf+H_{q+1}\dbar f$ for all $f\in \Co^\infty(\Omega;\wedge^{0,q})$.
\end{lem}
See \cite[Proposition~2.1]{GongHolderSPsiCXC2} or \cite[Theorem~11.2.2]{ChenShawBook} for a proof. For both references we have correspondence of notations $K=\Omega^{01}(b,\widehat Q)$. Here for integration of bi-degree forms we use the convention $\int_xu(x,y)dx^I\wedge dy^J:=(\int_xu(x,y)dx^I)dy^J$. Note that this is different from \cite[Section~III.1.9]{RangeSCVBook}.
% This is the same as  with change of notations, except we do not require $E$ to be a chosen extension operator. We give a proof in Appendix ??.

Here Lemma~\ref{Lem::Goal::HomotopyFormula} does not guarantee that $f=\dbar H_qf+H_{q+1}\dbar f$ holds for distributions, since $\Ec$ may not be defined on the space of distributions. % By choosing a suitable $E$ Lemma~\ref{Lem::Goal::HomotopyFormula} holds for  general $f\in\Ss'(\Omega;\wedge^{0,q})$, see Remark~\ref{Rmk::Space::ExtBdd}.
\begin{defn}\label{Defn::Goal::DefT}
    We construct the operator $\Hc_q$ from \eqref{Eqn::Goal::HomotopyOP} by taking $\Ec$ to be Rychkov's extension operator given in Definition~\ref{Defn::Space::ExtOmega}.
\end{defn}

Note that the Rychkov's extension operator is defined on the space $\Ss'(\Omega)$ of all extensible distributions. % Using some tricks of anti-derivatives we can reduce the problem to some weight estimates of $K_{q-1}(z,\zeta)$, see Proposition \ref{Prop::Space::ATDOp}. The anti-derivative method is developed in \cite{ShiYaoExt} and has been used in \cite{ShiYaoCk}.

The boundedness of $\Hc_q$ follows from the weighted estimates of the derivatives of the tangential part and the vertical part of $K_{q-1}(z,\zeta)$ with respect to $\zeta$-variable:
\begin{defn}\label{Defn::Goal::BotTop}
    Let $\varrho:U_1\to(-T_1,T_1)$ be a defining function of $\Omega$ with non-vanishing gradient and let $b\Omega_t=\{\varrho=t\}$ (for $|t|<T_1$) be as above. Let $1\le p,q\le n$ and $\zeta_0\in U$, we define the $\dbar$-vertical projection $(-)^\bot_{\zeta_0}$ and $\dbar$-tangential projection $(-)^\top_{\zeta_0}$ at $\zeta_0$ to be the following surjective orthonormal projections:
  %  \blu{G: This can be defined near the boundary of $\Omega$. Check that both can be globally defined near the boundary.}\bro{I mention the $\nabla\varrho\neq0$ here now}
    \begin{gather*}
        \textstyle(-)^\bot_{\zeta_0}:\bigwedge^{p,q}\C^n\twoheadrightarrow\bigwedge^p\C^n\otimes_\C \big(\Span\langle\dbar\varrho(\zeta_0)\rangle\wedge\bigwedge^{q-1}\C^n\big),\quad
        \textstyle(-)^\top_{\zeta_0}:\bigwedge^{p,q}\C^n\twoheadrightarrow\bigwedge^p\C^n\otimes_\C\bigwedge^{q} T^{*0,1}_{\zeta_0}(b\Omega_{\varrho(\zeta_0)}).
    \end{gather*}
    
    For a $(p,q)$-form $f:U_1\to\bigwedge^{p,q}\C^n$ we define $f^\bot(\zeta):=f(\zeta)^\bot_\zeta$ and $f^\top(\zeta):=f(\zeta)^\top_\zeta$ for $\zeta\in U_1$ naturally.
    % \begin{equation}\label{Eqn::Goal::BotTop::DbarDef}
    %     \dbar^\bot f:=(\dbar (f^\top))^\bot,\qquad\dbar^\top f:=(\dbar (f^\top))^\top+\dbar(f^\bot).
    % \end{equation}
\end{defn} %\bro{I delete the $\dbar^\top$}

Here for a real hypersurface $M\subset\C^n$ and a $\zeta\in M$, $T^{*0,1}_\zeta M:=T^{*0,1}_\zeta\C^n\cap\C T_\zeta^*M$ is the anti-holomorphic cotangent space of $M$ at $\zeta$.

\begin{note}\label{Note::Goal::KTopBot}
For the bidegree form $K(z,\zeta)$, we use $K^\top(z,\zeta)$ and $K^\bot(z,\zeta)$ for the projections with respect to $\zeta$-variable but not to $z$-variable, i.e. $K_q^\top(z,\zeta):=K_q(z,\cdot)^\top(\zeta)$ and $K_q^\bot(z,\zeta):=K_q(z,\cdot)^\bot(\zeta)$ for each $q$.
\end{note}
\begin{rem}
\label{Rmk::Goal::TopBotFacts}
Let $\overline\theta_1,\dots,\overline\theta_n$ be $(0,1)$-forms defined on an open subset $U\subset U_1$, that form an orthonormal frame, such that $\overline\theta_1=\dbar\varrho/|\dbar\varrho|$. Let $(\overline Z_1,\dots,\overline Z_n)$ be the dual basis, which are $(0,1)$ vector fields on $U$. Therefore,
\begin{gather*}
    \textstyle\Span(\overline\theta_2,\dots,\overline\theta_n)=\coprod_{\zeta\in U}T^{*0,1}_\zeta(b\Omega_{\varrho(\zeta)})\ (\subset T^{*0,1}U),\quad\textstyle \Span(\overline Z_2,\dots,\overline Z_n)=\coprod_{\zeta\in U}T^{0,1}_\zeta(b\Omega_{\varrho(\zeta)})\ (\subset T^{0,1}U).
\end{gather*}

We see that $\overline Z_1$ is uniquely determined by $\varrho$ (that does not depend on $(\overline\theta_2,\dots,\overline\theta_n)$) and is globally defined on $U_1$:
\begin{equation}\label{Eqn::Goal::TopBotFacts::Z1}
    \overline Z_1=\frac1{|\dbar\varrho|}\sum_{j=1}^n\frac{\partial\varrho}{\partial\zeta_j}\Coorvec{\overline\zeta_j}.
\end{equation}
Let $f=\sum_{|J|=p,|K|=q}f_{J, K}\theta^J\wedge\overline{\theta}^K$ be a $(p,q)$-form on $U$, where $f_{J,K}=\langle Z_J\wedge \overline Z_K,f\rangle$, we see that
\begin{equation*}
     \textstyle f^\bot=\sum_{|J|=p,|K'|=q-1}f_{J,1K'}\theta^J\wedge\overline\theta_1\wedge\overline{\theta}^{K'},\quad f^\top=\sum_{\substack{|J|=p,|K|=q;\,\min K\ge2}}f_{J,K}\theta^J\wedge\overline{\theta}^K.
\end{equation*}
Therefore, $f^\bot$ and $f^\top$ are still defined when $f$ has distributional coefficients, and we have the following:
\begin{equation}\label{Eqn::Goal::TopBotFacts::BotTopf}
    f^\bot=\overline\theta_1\wedge\iota_{\overline Z_1}f=(f^\bot)^\bot,\quad f^\top=f-f^\bot=(f^\top)^\top.
\end{equation}

% Let $g=\sum_{|J|=p,|K|=q}g_{JK}d\zeta^J\wedge d\overline{\zeta}^K$ be another $(p',q')$-form on $U$, we see that  \blu{Notice that $\dbar^\top$ may have $\bar\theta_1$ components.}
% \begin{equation}\label{Eqn::Goal::TopBotFacts::BotTopg}
%     \begin{aligned}
%     \dbar^\bot g&\textstyle=(-1)^{p'}\sum_{|J|=p',|K|=q'}(\overline Z_1g_{JK})d\zeta^J\wedge\overline\theta_1\wedge  d\overline{\zeta}^K=-\overline\theta_1\wedge \overline Z_1g;
%     \\
%     \dbar^\top g&\textstyle=(-1)^{p'}\sum_{|J|=p',|K|=q'}\sum_{j=2}^n(\overline Z_jg_{JK})d\zeta^J\wedge\overline\theta_j\wedge  d\overline{\zeta}^K=\dbar g-\dbar^\bot g.
% \end{aligned}
% \end{equation}
% In particular $\dbar^\top g$ and $\dbar^\bot g$ are still defined when $g$ has distributional coefficients.

Moreover, for a $(p',q')$-form $g$ on $U$, one can see that  
\begin{equation}\label{Eqn::Goal::TopBotFacts::BotTopfg}
    (f\wedge g)^\top=f^\top\wedge g^\top,\qquad f^\bot\wedge g^\bot=0,\qquad\text{ and thus}\quad f^\bot\wedge g=f^\bot\wedge g^\top.
\end{equation}

We leave the proof to the reader.
\end{rem}

The weighted estimates that we need are as follows.
\begin{thm}[Weighted estimates for $K(z,\zeta)$]\label{Thm::WeiEst}
% Let $\Omega\subset\C^n$ the convex domain of finite type $m$ as above. Let $K(z,\zeta)=\sum_{q=0}^{n-2}K_q(z,\zeta)$ be defined in \eqref{Eqn::Goal::DefK}. Let the projections $K^\top$ and $K^\bot$ be as in Notation~\ref{Note::Goal::KTopBot} via Definition~\ref{Defn::Goal::BotTop}. And
Let $\dist(w):=\dist(w,b\Omega)$.
 Let $1\le q\le n$. Assume $\Omega$ has $q$-type $m_q<\infty$. Let $r_q:=(n-q+1)\cdot m_q+2q$ and $\gamma_q=\frac{r_q}{r_q-1}$.

Then for any $k\ge2$ and $0<s<k-1-1/m_q$, there is a $C=C(\Omega,U_1,\widehat S,q,m_q,k,s)>0$ such that
\begin{align}
    \label{Eqn::WeiEst::Top+1}
    \int_{U_1\backslash\overline\Omega}\dist(\zeta)^s |D^k_{z,\zeta}(K_{q-1}^\top)(z,\zeta)|d\Vol(\zeta)&\le C\dist(z)^{s+1+\frac1{m_q}-k},&&\forall z\in \Omega;
    \\
    \label{Eqn::WeiEst::Top+2}
    \int_{\Omega}\dist(z)^s |D^k_{z,\zeta}(K_{q-1}^\top)(z,\zeta)|d\Vol(z)&\le C\dist(\zeta)^{s+1+\frac1{m_q}-k},&&\forall \zeta\in U_1\backslash\overline\Omega;
    \\
    \label{Eqn::WeiEst::Bot+1}
    \int_{U_1\backslash\overline\Omega}\dist(\zeta)^s|D^k_{z,\zeta}(K_{q-1}^\bot)(z,\zeta)|d\Vol(\zeta)&\le C\dist(z)^{s+\frac2{m_q}-k},&&\forall z\in \Omega;
    \\
    \label{Eqn::WeiEst::Bot+2}
    \int_{\Omega}\dist(z)^s |D^k_{z,\zeta}(K_{q-1}^\bot)(z,\zeta)|d\Vol(z)&\le C\dist(\zeta)^{s+\frac2{m_q}-k},&&\forall\zeta\in U_1\backslash\overline\Omega;
    \\
    \label{Eqn::WeiEst::Top01}
    \int_{U_1\backslash\overline\Omega}|\dist(\zeta)^s D^k_{z,\zeta}(K_{q-1}^\top)(z,\zeta)|^{\gamma_q}d\Vol(\zeta)&\le C\dist(z)^{(s+1-k)\gamma_q},&&\forall z\in \Omega;
    \\
    \label{Eqn::WeiEst::Top02}
    \int_{\Omega}|\dist(z)^s D^k_{z,\zeta}(K_{q-1}^\top)(z,\zeta)|^{\gamma_q}d\Vol(z)&\le C\dist(\zeta)^{(s+1-k)\gamma_q},&&\forall \zeta\in U_1\backslash\overline\Omega;
    \\
    \label{Eqn::WeiEst::Bot01}
    \int_{U_1\backslash\overline\Omega}|\dist(\zeta)^sD^k_{z,\zeta}(K_{q-1}^\bot)(z,\zeta)|^{\gamma_q}d\Vol(\zeta)&\le C\dist(z)^{(s-k+\frac1{m_q})\gamma_q},&&\forall z\in \Omega;
    \\
    \label{Eqn::WeiEst::Bot02}
    \int_{\Omega}|\dist(z)^s D^k_{z,\zeta}(K_{q-1}^\bot)(z,\zeta)|^{\gamma_q}d\Vol(z)&\le C\dist(\zeta)^{(s-k+\frac1{m_q})\gamma_q},&&\forall\zeta\in U_1\backslash\overline\Omega.
\end{align}
\end{thm}
Here we use $D^k_{z,\zeta}=\{\frac{\partial^{|\alpha+\beta+\gamma+\delta|}}{\partial z^\alpha\partial\bar z^\beta\partial\zeta^\gamma\partial\bar\gamma^\delta}:|\alpha+\beta+\gamma+\delta|\le k\}$ for the total derivatives among all variables, acting on their coordinate components. Notice that we take derivatives after we take ($\bot$ and $\top$) projections.
 We will prove Theorem~\ref{Thm::WeiEst} in Section~\ref{Section::Basis}.

The estimates \eqref{Eqn::WeiEst::Bot+1}, \eqref{Eqn::WeiEst::Bot+2}, \eqref{Eqn::WeiEst::Bot01} and \eqref{Eqn::WeiEst::Bot02} are all not optimal. In practice, to prove Theorems \ref{MainThm} and \ref{StrMainThm} it is enough replace the $\frac2{m_q}$-factors in \eqref{Eqn::WeiEst::Bot+1} and \eqref{Eqn::WeiEst::Bot+2} by any $\eps+\frac1{m_q}$, and the $\frac1{m_q}$-factors in \eqref{Eqn::WeiEst::Bot01} and \eqref{Eqn::WeiEst::Bot02} by any $\eps$, for all $\eps>0$. See Remark~\ref{Rmk::Basis::FinalBasisEstRmk} \ref{Item::Basis::FinalBasisEstRmk::NIso} for their improvements.

One can see later in Corollary~\ref{Cor::LPThm::BddDom} \ref{Item::LPThm::BddDom::TangComm} that $[\dbar,\Ec]^\top$ does not lose derivative, see also Remark~\ref{Rmk::LPThm::TangCommRmk}. This technique is not necessary in the estimates for strongly pseudoconvex domains, see Remark~\ref{Rmk::Basis::EstQ2::SPsiCX}.

By expanding $K_{q-1}(z,\zeta)$ from \eqref{Eqn::Goal::DefK}, we see that its coefficients are the (constant) linear combinations of 
\begin{equation}\label{Eqn::Goal::MainKernel}
    \frac{b(z,\zeta)\wedge\widehat Q(z,\zeta)\wedge\big(\dbar \widehat Q(z,\zeta)\big)^{k-1}}{\widehat S(z,\zeta)^k|z-\zeta|^{2(n-k)}},\quad 1\le k\le n-q.
\end{equation}

We start with the estimates of the components in \eqref{Eqn::Goal::MainKernel} in Section~\ref{Section::Basis}.

% By Lemma~\ref{Lem::Goal::Glue}, $|K_{q-1}(z,\zeta)|$ is obviously bounded from below when $|z-\zeta|\ge(2M_4)^{-1}$, and we can replace $\widehat S$ by $S$ in \eqref{Eqn::Goal::MainKernel} and \eqref{Eqn::Goal::HatQ}, and restrict our focus to $|z-\zeta|<(2M_4)^{-1}$.

% Finally by a stronger version of Hardy-Littlewood lemma (Corollary~\ref{Cor::LPThm::BddDom}), 

\section{Estimates via $\eps$-Minimal Bases}\label{Section::Basis}
We recall some notations and definitions from \cite{McNealFiniteType,YuMultitype}.

\begin{defn}\label{Defn::Basis::Type}
    Let $\Omega\subset\C^n$ be an open set and let $\zeta\in b\Omega$. For $1\le q\le n$, the (complex affine) $q$-type of $\Omega$ at $\zeta$ is
    \begin{equation*}
        L_q(b\Omega,\zeta):=\sup\Big\{m\in\R_+:\varliminf_{w\to 0;w\in H}\frac{\dist(\zeta+w,b\Omega)}{|w|^m}=0\text{ for all $q$-dim $\C$-linear subspace }H\le \C^n\Big\}.
    \end{equation*}
    
    The (affine) $q$-type of $\Omega$ is the minimum of $L_q(b\Omega,\zeta)$ among all $\zeta\in b\Omega$, which we denoted by $m_q$.
\end{defn}
As mentioned in \cite[Theorem~2.1]{HeferMultitype}, on convex domains the affine types, the D'Angelo types and the regular D'Angelo types all coincide. Moreover, if $\Omega$ has affine $q$-type $m_q<\infty$ for $1\le q\le n$, then $(m_n,m_{n-1},\dots,m_1)$ is the Catlin's multitype of $\Omega$. See \cite{McNealFiniteType,BoasStraubeFiniteType,YuMultitype}. In particular $m_1\ge\dots\ge m_{n-1}\ge2$ are all even integers and $m_n=1$.

To study the $\frac1{m_q}$ gain on $(0,q)$-forms, especially for $q\ge2$, we use the approach of $\eps$-minimal basis, which is introduced by \cite{YuMultitype} and used in \cite{HeferMultitype}.
\begin{defn}\label{Defn::Basis::Basis}
    Let $\Omega\subset\C^n$ be a finite type convex domain where the defining function $\varrho$ is as before. For $\zeta\in U_1$, $v\in\C^n$ and $\eps>0$ let
    \begin{equation*}
        \tau(\zeta,v,\eps):=\sup\{c>0:|\varrho(\zeta+\lambda v)-\varrho(\zeta)|\le\eps,\quad\forall\lambda\in\C,\ |\lambda|\le c\}.
    \end{equation*}
    % For $c>0$, let $cP_\eps(\zeta)$
    
    An \textit{$\eps$-minimal basis} (or a \textit{Yu-basis at the scale $\eps$}) $(v_1,\dots,v_n)$ at $\zeta\in U_1$ is given recursively as follows: for $1\le k\le n$, $v_k\in \C^n$ is a unit vector minimizing to the following quantity  of $v$:
    \begin{equation*}
        \dist\big(\zeta,\{z\in (\zeta+\C\cdot v):\varrho(z)=\varrho(\zeta)+\eps\}\big),\quad\text{where }|v|=1\text{ and }v\bot\Span_\C\langle v_1,\dots,v_{k-1}\rangle.
    \end{equation*}
    Here for $k=1$ we use $\Span_\C\varnothing=\{0\}$.
    
    % We also call $(v_1,\dots,v_n)$ a \textit{$\eps$-minimal coordinate system} at $\zeta$. 
    
    For an $\eps$-minimal basis $(v_1,\dots,v_n)$ at $\zeta$, we define $\tau_j(\zeta,\eps):=\tau(\zeta,v_j,\eps)$ for $1\le j\le n$, and the ellipsoid $P_\eps(\zeta):=\big\{\zeta+\sum_{j=1}^na_jv_j:\sum_{j=1}^n\frac{|a_j|^2}{\tau_j(\zeta,\eps)^2}<1\big\}\subset\C^n$. For $c>0$, we set $cP_\eps(\zeta):=\{z\in\C^n:\zeta+\frac{z-\zeta}c\in P_\eps(\zeta)\}$ for a dilation of $P_\eps(\zeta)$ with the same center.    
    % We define $\tau'_1(\zeta,\eps)=\eps^\frac12$ and $\tau'_j(\zeta,\eps)=\tau_j(\zeta,\eps)$ for $2\le j\le n$.
    
    % at $\zeta$ is a linear coordinate system $\omega=(\omega_1,\dots,\omega_n)$ such that $(\frac\partial{\partial\omega_1},\dots,\frac\partial{\partial\omega_n})$ is a $\eps$-minimal basis at $\zeta$.
 
\end{defn}
Here we use ellipsoid rather than rectangle to define $P_\eps(\zeta)$ (cf. \cite[Section 3]{DFFHolder} for example). One can see that $\{\tau_j(\zeta,\eps)\}_{j=1}^n$ and $P_\eps(\zeta)$ do not depend on the choice of the $\eps$-minimal basis. 

% Indeed, where there are two $\eps$-minimal basis $(v_1,\dots,v_n)$ and $(v'_1,\dots,v'_n)$ at one fixed $\zeta$, say $v_j$ is the non-zero linear combinations of $\{v'_k\}_{k\in I}$, then by the minimality assumption in the construction we must have $\tau(\zeta,v_j,\eps)=\tau(\zeta,v'_k,\eps)$ for all $k\in I$

%   In what follows we denote by $\omega^{\zeta,\eps}=(\omega_1^{\zeta,\eps},\dots,\omega_n^{\zeta,\eps})$ a $\eps$-minimal coordinate system at $\zeta$. That is, $\omega^{\zeta,\eps}$ is a linear coordinate system such that $(\frac\partial{\partial\omega^{\zeta,\eps}_1},\dots,\frac\partial{\partial\omega^{\zeta,\eps}_n})$ is a $\eps$-minimal basis at $\zeta$. We use $\omega=\omega^{\zeta,\eps}$ if the parameters are clear.
    
We recall the following from \cite{McNealFiniteType,YuMultitype,DFFHolder,HeferMultitype}. Recall Lemma~\ref{Lem::Goal::Glue} for $M_5>0$ and $U_1\supset b\Omega$.
\begin{lem}\label{Lem::Basis::Prem}
Assume the finite type convex domain $\Omega\subset\C^n$ has $q$-type $m_q<\infty$. 
Then there are a $C_0>1$, an $\eps_0\in(0,\frac1{M_5})$, and for every multi-index $\beta=(\beta',\beta'')\in\N^{2n}$  a $C_\beta>0$, such that
\begin{enumerate}[(i)]
    \item\label{Item::Basis::Prem1} For every $\zeta\in \{|\varrho|<\eps_0\}$ we have $P_{\eps_0}(\zeta)\subseteq U_1$. Moreover for every $0<\eps\le\eps_0$ and every $\eps$-minimal basis $(v_1,\dots,v_n)$ at $\zeta$ {\normalfont(recall that $\tau_j(\zeta,\eps):=\tau(\zeta,v_j,\eps)$)}:
\begin{gather}
    \label{Eqn::Basis::Prem::Set}
    P_{\eps}(\zeta')\subseteq C_0P_{\eps/2}(\zeta)\text{ and }2P_\eps(\zeta')\subseteq P_{C_0\eps}(\zeta),\quad\forall\zeta'\in P_\eps(\zeta);
    \\
    \label{Eqn::Basis::Prem::Order1}
    \tau_1(\zeta,\eps)\le\tau_2(\zeta,\eps)\le\tau_3(\zeta,\eps)\le\dots\le\tau_n(\zeta,\eps);
    \\
    \label{Eqn::Basis::Prem::Order2}
    \tau_1(\zeta,\eps)\ge\tfrac1{C_0}\eps,\quad \tau_2(\zeta,\eps)\ge\tfrac1{C_0}\eps^\frac12;
    \\
    \label{Eqn::Basis::Prem::TypeBdd}
    \tau_q(\zeta,\eps)\le C_0\eps^{1/m_{n+1-q}},\quad\forall 1\le q\le n;
    % \\
    % \label{Eqn::Basis::Prem::Order2}
    % \frac1{C_0^2}\sum_{j=1}^n\frac{|a_j|^2}{\tau_j(\zeta,\eps)^2}\le \frac1{\tau\big(\zeta,\sum_{j=1}^na_jv_j,\eps\big)^2}\le C_0^2\sum_{j=1}^n\frac{|a_j|^2}{\tau_j(\zeta,\eps)^2},\quad\forall (a_1,\dots,a_n)\in\C^n;
\end{gather}
\item\label{Item::Basis::Prem2} For every $\zeta\in U_1$, $0<\eps\le\eps_0$ and $\eps$-minimal basis $(v_1,\dots,v_n)$ at $\zeta$,
\begin{equation}\label{Eqn::Basis::Prem::DefFunBdd}
    \Big|\frac{\partial^{|\beta|}}{\partial w^{\beta'}\partial\bar w^{\beta''}}\varrho(\zeta+w_1v_1+\dots w_nv_n)\Big|\le C_\beta\frac \eps{\prod_{j=1}^n\tau_j(\zeta,\eps)^{\beta'_j+\beta''_j}},\quad \forall w\in \C^n\text{ such that }\sum_{j=1}^n\frac{|w_j|^2}{\tau_j(\zeta,\eps)^2}<1.
\end{equation}
\end{enumerate}
\end{lem}

\begin{rem}\label{Rmk::Basis::PremRmk}
\begin{enumerate}[(i)]
    \item \eqref{Eqn::Basis::Prem::TypeBdd} is useful particularly for $q\ge2$. If one only care about $(0,1)$-forms, we could use the $\eps$-extremal basis (a.k.a. the McNeal-basis at the scale $\eps$) to define $\tau_1,\dots,\tau_n$, where \eqref{Eqn::Basis::Prem::Set} and \eqref{Eqn::Basis::Prem::Order2} remain true, but \eqref{Eqn::Basis::Prem::TypeBdd} is replaced by $\eps^{1/m_{q-1}}\lesssim\tau_q\lesssim \eps^{1/m_1}$. See \cite[Theorem~2.3]{HeferMultitype} for example.
    
    \item\label{Item::Basis::PremRmk::TrivialEst} In \eqref{Eqn::Basis::Prem::DefFunBdd}, $\sum_{j=1}^n\frac{|w_j|^2}{\tau_j(\zeta,\eps)^2}<1$ is the same as saying $\zeta+w_1v_1+\dots+w_nv_n\in P_\eps(\zeta)$. This estimate is only useful when $\sum_{j=1}^n\frac1{m_{n+1-j}}(\beta'_j+\beta''_j)<1$. Otherwise the right hand side of \eqref{Eqn::Basis::Prem::DefFunBdd} is bounded from below (or even goes to $\infty$ as $\eps\to0$) whereas the left hand side is always uniformly bounded.
\end{enumerate}

% \eqref{Eqn::Basis::Prem::DefFunBdd} does not include the case $\beta'+\beta''=(1,0,\dots,0)$, where we only have trivial estimate $|D^\beta_{\varphi}\varrho|\lesssim1$.
\end{rem}

% We now define a neighborhood of $\overline\Omega$ throughout Sections \ref{Section::Basis} and \ref{Section::PfThm}:
% \begin{equation}\label{Eqn::Basis::Uc}
%     \Uc:=\{\zeta\in\C^n:\varrho(\zeta)<\eps_0\}\subset U_1.
% \end{equation}

To prove Theorem~\ref{Thm::WeiEst} we need to estimate $|K_{q-1}^\top(z,\zeta)|$ and $|K_{q-1}^\bot(z,\zeta)|$ inside an ellipsoid $P_\eps(\zeta)$. Recall the $\eps_0$ in Lemma~\ref{Lem::Basis::Prem} and the $\widehat Q$ in \eqref{Eqn::Goal::HatQ}.
\begin{lem}\label{Lem::Basis::EstSQ} Keeping the notations in Lemma~\ref{Lem::Basis::Prem}, there is a $C_1>0$ that satisfies the following:
\begin{enumerate}[(i)]
    \item \label{Item::Basis::EstSQ::S}
    $|\widehat S(z,\zeta)|\ge \tfrac1{C_1}\eps$, for every $\zeta\in U_1$, $0<\eps\le\eps_0$ and $z\in\Omega_{\varrho(\zeta)}\backslash P_\eps(\zeta)$.
    \item \label{Item::Basis::EstSQ::Q}Let $\zeta_0\in U_1$, $0<\eps\le\eps_0$ and let $\Psi_0\in\C^{n\times n}$ be a unitary matrix such that its $n$ column vectors (with order) form an $\eps$-minimal basis at $\zeta_0$. Let $\widehat Q_{\Psi_0}(z,\zeta):=\overline\Psi_0\cdot \widehat Q(\Psi_0\cdot z,\Psi_0\cdot\zeta)$. Then for $1\le j,k\le n$,
\begin{equation}\label{Eqn::Basis::EstSQ::Q}
        |\widehat Q_{\Psi_0,j}(z,\zeta_0)|\le\frac{C_1\eps}{\tau_j(\zeta_0,\eps)},\quad\Big|\Coorvec{\overline\zeta_k}\widehat Q_{\Psi_0,j}(z,\zeta_0)\Big|\le\frac{C_1\eps}{\tau_j(\zeta_0,\eps)\tau_k(\zeta_0,\eps)},\quad\text{for } z\in P_\eps(\zeta_0).
\end{equation}
\end{enumerate}

\end{lem}
\begin{proof}See \cite[Lemma~4.2]{DFFHolder}, \cite[Proposition~4.1]{HeferMultitype} or \cite[Lemma~4(ii)]{AlexandreCk} for \ref{Item::Basis::EstSQ::S}. Note that the modification from $S$ to $\widehat S$ in Lemma~\ref{Lem::Goal::Glue} ensures that $|\widehat S(z,\zeta)|$ is bounded from below when $|z-\zeta|$ is large.

For \eqref{Eqn::Basis::EstSQ::Q}, since we fix the point $\zeta_0$ and the number $\eps$, by passing to a unitary coordinate change we can assume that $\Psi_0=I_n$. In particular $\widehat Q_{\Psi_0}=\widehat Q$.

\eqref{Eqn::Basis::EstSQ::Q} is a weaker version of \cite[Lemma~8]{AlexandreCk} where (under the assumption $\Psi_0=I_n$) it was proved that $|\partial_{\overline\zeta_k}\widehat Q_j(z,\zeta_0)|\lesssim\frac{\eps}{\tau_j\tau'_k}$. In the statement, $\tau'_1=\eps^{\frac12}\gg\eps\approx\tau_1$ and $\tau'_k=\tau_k$ for $k\ge2$. \cite[Lemma~8]{AlexandreCk} is stated with $\eps$-extremal bases, but the result is still true if we replace it by $\eps$-minimal bases. There is no additional change to the proof.

Alternatively, we define $Q=[Q_1,\dots,Q_n]^\intercal$ as the $\widehat Q$ from \eqref{Eqn::Goal::HatQ} with $\widehat S(z,\zeta)$ replaced by $S(z,\zeta)$, i.e. $Q_j(z,\zeta)=\int_0^1\frac{\partial S}{\partial z_j}(\zeta+t(z-\zeta),\zeta)dt$. 
Then we have 
\begin{equation*}
    \widehat Q_j=A\cdot Q_j+\partial_{z_j}A\cdot S,\quad\text{and}\quad\partial_{\overline\zeta_k}\widehat Q_j=A\cdot\widehat Q_j+\partial_{\overline\zeta_k}A\cdot\widehat Q_j+\partial_{z_j}A\cdot\partial_{\overline\zeta_k}S+\partial_{z_j\overline\zeta_k}^2A\cdot S,
\end{equation*}
where $A\in C^\infty$ is in Lemma~\ref{Lem::Goal::Glue} \ref{Item::Goal::Glue::A}. By \cite[Lemma~5.1]{DFFHolder} we have $|Q_j|\lesssim\eps/\tau_j$ and $|\partial_{\overline\zeta_k}Q_j|\lesssim\eps/(\tau_j\tau_k)$, thus \eqref{Eqn::Basis::EstSQ::Q} follows from the fact that $A\in C^2$ and $S(z,\zeta)=\sum_{j=1}^nQ_j(z,\zeta)(z_j-\zeta_j)$. Again \cite[Lemma~5.1]{DFFHolder} is stated with $\eps$-extremal bases, but the result is still true if we replace it by $\eps$-minimal bases. There is no additional change to the proof.
\end{proof}

\begin{cor}\label{Cor::Basis::EstQ2}Keeping the notations from Lemma~\ref{Lem::Basis::EstSQ} \ref{Item::Basis::EstSQ::Q}, and we identify the column vector function $\widehat Q_{\Psi_0}=[\widehat Q_{\Psi_0,1},\dots,\widehat Q_{\Psi_0,n}]$ with the $(1,0)$-form $\widehat Q_{\Psi_0}=\sum_{j=1}^n\widehat Q_{\Psi_0,j}(z,\zeta)d\zeta_j$. Then the estimates \eqref{Eqn::Basis::EstQ2::Top} and \eqref{Eqn::Basis::EstQ2::Bot} imply the following:

There is a $C'_1>0$ that does not depend on $\zeta_0\in U_1$ and $0<\eps\le\eps_0$, such that for every $a\in\{0,1\}$, $0\le b\le n-1$ and every $z\in P_\eps(\zeta_0)$,
\begin{align}
\label{Eqn::Basis::EstQ2::Top}
    \big|(\widehat Q_{\Psi_0})^a\wedge(\dbar\widehat Q_{\Psi_0})^b(z,\zeta_0)\mod d\bar\zeta_1\big|&\le C'_1\frac{\eps^{a+b-1}}{\prod_{l=2}^{b+1}\tau_l(\zeta_0,\eps)^2};
    \\
    \label{Eqn::Basis::EstQ2::Bot}
    \big|(\widehat Q_{\Psi_0})^a\wedge(\dbar\widehat Q_{\Psi_0})^b(z,\zeta_0)\big|&\le C'_1\frac{\eps^{a+b-2}\tau_{b+1}(\zeta_0,\eps)}{\prod_{l=2}^{b+1}\tau_l(\zeta_0,\eps)^2}.
\end{align}
\end{cor}
Here by \eqref{Eqn::Basis::EstQ2::Top} we mean that, if $(\widehat Q_{\Psi_0})^a\wedge(\dbar\widehat Q_{\Psi_0})^b(z,\zeta)=\sum_{|J|=a+b;|K|=v}f_{JK}(z,\zeta)d\zeta^J\wedge d\overline\zeta^K$, then\\ $|f_{JK}|\lesssim\eps^{a+b-1}/\prod_{l=2}^{b+1}\tau_l^2$ for all index sets $J$ and $K=(k_1,\dots,k_b)$ such that $k_1,\dots,k_b\ge2$.

\begin{rem}This is essentially the \cite[Lemma~5.5]{DFFHolder} or \cite[Lemma~4.2]{HeferMultitype}, where they use the $(1,0)$-form $Q(z,\zeta)$ (see the proof of Lemma \ref{Lem::Basis::EstSQ}) instead of the $\widehat Q(z,\zeta)$ in the statement.
% \begin{enumerate}[(i)]
%     \item 
%     \item \eqref{Eqn::Basis::EstQ2::Top} and \eqref{Eqn::Basis::EstQ2::Bot} require no other than \eqref{Eqn::Basis::EstSQ::Q}. We are going to recycle the estimates from Corollary~\ref{Cor::Basis::EstQ2} to Lemma~\ref{Lem::Basis::IntEst} later in Section~\ref{Section::SPsiCX}.
% \end{enumerate}
\end{rem}
\begin{rem}\label{Rmk::Basis::EstQ2::SPsiCX}
In the case of strongly pseudoconvex domains we informally have $\tau_2\approx\dots\approx\tau_n\approx\eps^\frac12$, which means $|(\widehat Q_{\Psi_0})^a\wedge(\dbar\widehat Q_{\Psi_0})^b\mod d\bar\zeta_1|\lesssim\eps^{a-1}$ and $|(\widehat Q_{\Psi_0})^a\wedge(\dbar\widehat Q_{\Psi_0})^b|\lesssim\eps^{a-\frac32}$. Since $\eps^{a-1},\eps^{a-\frac32}\gtrsim1$, recall from Remark~\ref{Rmk::Basis::PremRmk} \ref{Item::Basis::PremRmk::TrivialEst} that these estimates become unnecessary.
\end{rem}
    
\begin{proof}[Proof of Corollary~\ref{Cor::Basis::EstQ2}]
Again we assume $\Psi_0=I_n$.
    By writing $\widehat Q^a\wedge(\dbar\widehat Q)^b=\sum_{|J|=a+b;|K|=v}f_{JK}d\zeta^J\wedge d\overline\zeta^K$, we have $ f_{j_1\dots j_{a+b},k_1\dots k_b}=\pm\prod_{p=1}^b\frac{\partial \widehat Q_{j_p}}{\partial \overline\zeta_{k_p}}\prod_{q=b+1}^{a+b}\widehat Q_{j_q}$. By \eqref{Eqn::Basis::EstSQ::Q} we have
    \begin{equation}\label{Eqn::Basis::EstQ::Tmp2}
        \big|f_{j_1\dots j_{a+b},k_1\dots k_b}\big|\lesssim\frac{\eps^{a+b}}{\prod_{p=1}^b\tau_{j_p}\tau_{k_p}\prod_{q=b+1}^{a+b}\tau_{j_q}}.
    \end{equation}

    In the differential form $(j_1,\dots,j_{a+b})$ and $(k_1,\dots,k_b)$ are two collections of distinct indices. Therefore we can assume $(1\le)j_{b+1}< \dots<j_{a+b}<j_1<\dots<j_b$ and $k_1<\dots<k_b$. In \eqref{Eqn::Basis::FinalBasisEst::Top} we modular $d\overline\zeta_1$, where we only need the case $k_1\ge2$.
    
    By \eqref{Eqn::Basis::Prem::Order1} and \eqref{Eqn::Basis::Prem::Order2}, $\eps\approx\tau_1\le\dots\le \tau_n$, therefore for the case $k_1\ge2$,
    \begin{equation*}
        |f_{j_1\dots j_{a+b},k_1\dots k_{a+b}}|\lesssim\begin{cases}
        \displaystyle\frac{\eps^{b+1}}{\tau_{j_{b+1}}\prod_{p=1}^b\tau_{j_p}\tau_{k_p}}\Big|_{\substack{j_{b+1}=1\\j_p=k_p=p+1}}\approx \frac{\eps^b}{\prod_{l=2}^{b+1}\tau_l^2}&\text{when }a=1
        \\
        \displaystyle\frac{\eps^b}{\prod_{p=1}^b\tau_{j_p}\tau_{k_p}}\Big|_{\substack{j_p=p\\k_p=p+1}}\approx\frac{\eps^{b-1}}{\prod_{l=2}^b\tau_l^2\cdot\tau_{b+1}}&\text{when }a=0
        \end{cases}\le\frac{\eps^{a+b-1}}{\prod_{l=2}^{b+1}\tau_l^2}.
    \end{equation*}
    This proves \eqref{Eqn::Basis::EstQ2::Top}.
    
    Similarly for the case $k_1\ge1$,
    \begin{equation*}
        |f_{j_1\dots j_{a+b},k_1\dots k_{a+b}}|\lesssim\begin{cases}
        \displaystyle\frac{\eps^{b+1}}{\tau_{j_{b+1}}\prod_{p=1}^b\tau_{j_p}\tau_{k_p}}\Big|_{\substack{j_{b+1}=1\\j_p=p+1\\k_p=p}}\approx \frac{\eps^{b-1}}{\prod_{l=2}^b\tau_l^2\cdot \tau_{b+1}}&\text{when }a=1
        \\
        \displaystyle\frac{\eps^b}{\prod_{p=1}^b\tau_{j_p}\tau_{k_p}}\Big|_{\substack{j_p=k_p=p}}\approx\frac{\eps^{b-2}}{\prod_{l=2}^b\tau_l^2}&\text{when }a=0
        \end{cases}\le\frac{\eps^{a+b-2}\tau_{b+1}}{\prod_{l=2}^{b+1}\tau_l^2}.
    \end{equation*}
    This proves \eqref{Eqn::Basis::EstQ2::Bot}.
\end{proof}

Taking pullback from $\zeta\mapsto\Psi_0\cdot\zeta$, we have the following:
\begin{lem}\label{Lem::Basis::FinalBasisEst}
For every $j\ge0$ there is a $C_j>0$ such that for every $0\le k\le n-1$, $0<\eps\le\eps_0$, $\zeta\in U_1\backslash\overline\Omega$ and $z\in \Omega_{\varrho(\zeta)}\cap P_\eps(\zeta)\backslash P_{\eps/2}(\zeta)$,
\begin{align}\label{Eqn::Basis::FinalBasisEst::Top}
    \bigg|D^j_{z,\zeta}\Big(\frac{\widehat Q\wedge(\dbar\widehat Q)^{k}}{\widehat S^{k+1}}\Big)^\top(z,\zeta)\bigg|&\le C_j\frac{\eps^{-1-j}}{\prod_{l=2}^{k+1}\tau_l(\zeta,\eps)^2};
    \\
    \label{Eqn::Basis::FinalBasisEst::Bot}
    \bigg|D^j_{z,\zeta}\Big(\frac{\widehat Q\wedge(\dbar\widehat Q)^{k}}{\widehat S^{k+1}}\Big)(z,\zeta)\bigg|&\le C_j\frac{\eps^{-2-j}\tau_{k+1}(\zeta,\eps)}{\prod_{l=2}^{k+1}\tau_l(\zeta,\eps)^2}.
\end{align}Here $D^j=\{\partial^\alpha_z\partial^\beta_\zeta\partial^\gamma_{\overline\zeta}\}_{|\alpha+\beta+\gamma|\le j}$ is the collection of differential operators acting on the components.
\end{lem}
\begin{rem}\label{Rmk::Basis::FinalBasisEstRmk}
\begin{enumerate}[(i)]
    \item There is no $\bar z$-derivative since the fractions are holomorphic in $z\in\Omega_{\varrho(\zeta)}(\supset\Omega)$.
    \item In fact \eqref{Eqn::Basis::FinalBasisEst::Top} and \eqref{Eqn::Basis::FinalBasisEst::Bot} correspond to the term $\frac{\eps^{-j}}{\prod_{i=0}^k\tau_{\nu_i}\prod_{i=1}^k\tau_{\mu_i}}$ and $\frac{\eps^{-j-\frac12}}{\prod_{i=0}^k\tau_{\nu_i}\prod_{i=1}^{k-1}\tau_{\mu_i}}$ in \cite[Lemma~9]{AlexandreCk} respectively. By doing a refined estimate on the normal direction, where \cite{AlexandreCk} introduced a notion $\tau_1':=\eps^\frac12$, one can improve \eqref{Eqn::Basis::FinalBasisEst::Bot} by a factor of $\eps^\frac12$. However \eqref{Eqn::Basis::FinalBasisEst::Bot} is enough for our proof. % since the main term is \eqref{Eqn::Basis::FinalBasisEst::Top}.
    \item\label{Item::Basis::FinalBasisEstRmk::NIso} If we consider the anisotropic estimates, one can show that for $\alpha,\beta,\gamma\in\N^n$, $\zeta\in U_1\backslash\overline\Omega$ and $z\in\Omega\cap P_\eps(\zeta)$, if the standard coordinate basis is $\eps$-minimal at $\zeta$, then
    \begin{align*}
    \bigg|\partial^\alpha_z\partial^\beta_\zeta\partial^\gamma_{\overline\zeta}\Big(\frac{\widehat Q\wedge(\dbar\widehat Q)^{k}}{\widehat S^{k+1}}\Big)^\top(z,\zeta)\bigg|&\lesssim_{\alpha,\beta,\gamma} \frac{\eps^{-1-\alpha_1-\beta_1-\frac12\gamma_1}}{\prod_{l=2}^{k+1}\tau_l(\zeta,\eps)^2}\prod_{j=2}^{n}\tau_j(\zeta,\eps)^{-\alpha_j-\beta_j-\gamma_j};
    \\
    \bigg|\partial^\alpha_z\partial^\beta_\zeta\partial^\gamma_{\overline\zeta}\Big(\frac{\widehat Q\wedge(\dbar\widehat Q)^{k}}{\widehat S^{k+1}}\Big)(z,\zeta)\bigg|&\lesssim_{\alpha,\beta,\gamma}\frac{\eps^{-\frac32-\alpha_1-\beta_1-\frac12\gamma_1}}{\tau_{k+1}(\zeta,\eps)\prod_{l=2}^{k}\tau_l(\zeta,\eps)^2}\prod_{j=2}^{n}\tau_j(\zeta,\eps)^{-\alpha_j-\beta_j-\gamma_j}.
\end{align*}
The proof requires Alexandre's estimate for normal derivatives in \cite[Section~2]{AlexandreCk}.
\end{enumerate} 
\end{rem}
\begin{proof}[Proof of Lemma \ref{Lem::Basis::FinalBasisEst}]
% Both \eqref{Eqn::Basis::FinalBasisEst::Top} and \eqref{Eqn::Basis::FinalBasisEst::Bot} are invariant under a change of unitary coordinate system. Thus when we estimate can assume that the standard basis is an $\eps$-minimal 
For convenience we write\footnote{On a $(1,0)$-form $f$ we have $(\dbar f)^\top=\dbar_bf$ on the boundary $b\Omega$. For orthonormal $(0,1)$-forms $(\overline\theta_1,\dots,\overline\theta_n)$ and its dual basis $(\overline Z_1,\dots,\overline Z_n)$ in Remark~\ref{Rmk::Goal::TopBotFacts}, we have $\dbar_b(f_kd\zeta_k)=\sum_{j=2}^n(\overline Z_jf)\overline\theta_j\wedge d\zeta_k$. One can interpret $\dbar^\top|_{b\Omega_t}$ as $\dbar_b$ on each leaf $b\Omega_t\subset U_1$.} $\dbar^\top\widehat Q:=(\dbar\widehat Q)^\top$ throughout the proof. 

By Lemma~\ref{Lem::Basis::EstSQ} \ref{Item::Basis::EstSQ::S} and Lemma~\ref{Lem::Goal::Glue} \ref{Item::Goal::Glue::A} we have $|\widehat S|\gtrsim\eps$ for $z\in\Omega_{\varrho(\zeta)}\backslash P_{\eps/2}(\zeta)$. Applying the trivial estimate $|D^j\widehat S|\lesssim_j1$ for $j\ge1$ we see that 
\begin{equation}\label{Eqn::Basis::FinalBasisEst::EstS}
    |D^j (\widehat S^{-k})(z,\zeta)|\lesssim_j|\widehat S(z,\zeta)|^{-k-j}\lesssim_j\eps^{-k-j}\text{ when }z\in (\Omega\cap P_\eps(\zeta))\backslash P_{\eps/2}(\zeta).
\end{equation}

Therefore, by applying product rules, for
$z\in (\Omega _{\varrho (\zeta )}\cap P_{\varepsilon }(\zeta ))
\backslash P_{\varepsilon /2}(\zeta )$, we uniformly have
\begin{align*}
    &\bigg|D^j_{z,\zeta}\frac{\widehat Q\wedge(\dbar^\top\widehat Q)^{k}}{\widehat S^{k+1}}\bigg|    \lesssim_j\sum_{q=0}^{k}\sum_{\substack{j_0+\dots+j_{q+1}=j\\j_2,\dots,j_q\ge1}}\Big|D^{j_0}\frac1{\widehat S^{k+1}}\Big|\big|(D^{j_1}\widehat Q)\wedge(\dbar^\top\widehat Q)^{k-q}\big|\prod_{p=2}^{q+1}|D^{j_p}(\dbar\widehat Q)|\hspace{-1in}
    \\
    \lesssim&_j\sum_{q=0}^{k}\bigg(\sum_{j_0=0}^{j-q}\Big|D^{j_0}\frac1{\widehat S^{k+1}}\Big|\cdot \big|\widehat Q\wedge(\dbar^\top\widehat Q)^{k-q}\big|+\sum_{\substack{1\le j_1\le j-q\\j_0\le j-q-j_1}}\Big|D^{j_0}\frac1{\widehat S^{k+1}}\Big|\cdot |D^{j_1}\widehat Q|\cdot\big|(\dbar^\top\widehat Q)^{k-q}\big|\bigg),\hspace{-0.8in}&(|D^{j_p}(\dbar\widehat Q)|\lesssim_j1)
    \\
    \lesssim&_{k,\widehat S}\sum_{q=0}^k\Big(\sum_{j_0=1}^{j-q}\eps^{-k-j_0-1}\big|\widehat Q\wedge(\dbar^\top\widehat Q)^{k-q}\big|+\sum_{j_0=1}^{j-q-1}\eps^{-k-j_0-1}\big|(\dbar^\top\widehat Q)^{k-q}\big|\Big),
    &(\text{by }\eqref{Eqn::Basis::FinalBasisEst::EstS}\text{ and }|D^{j_1}\widehat Q|\lesssim_j1).
\end{align*}

Now fix $\zeta=\zeta_0$ and $\eps$. Since the left hand side of \eqref{Eqn::Basis::FinalBasisEst::Top} is invariant under a change of unitary coordinate system, we can assume that  the standard basis is $\eps$-minimal at $\zeta_0$, thus $\dbar^\top\widehat Q_j(z,\zeta_0)=\sum_{k=2}^n\frac{\partial\widehat Q_j}{\partial \overline\zeta_k}(z,\zeta_0)d\overline\zeta_k$. Applying \eqref{Eqn::Basis::EstQ2::Top} we get
\begin{equation*}
    \bigg|D^j_{z,\zeta}\Big(\frac{\widehat Q\wedge(\dbar^\top\widehat Q)^{k}}{\widehat S^{k+1}}\Big)\bigg|\lesssim\sum_{q=0}^k\Big(\sum_{j_0=1}^{j-q}\eps^{-k-j_0-1}\frac{\eps^{k-q}}{\prod_{l=2}^{k-q+1}\tau_l^2}+\sum_{j_0=1}^{j-q-1}\eps^{-k-j_0-1}\frac{\eps^{k-q-1}}{\prod_{l=2}^{k-q+1}\tau_l^2}\Big)\lesssim\frac{\eps^{-1-j}}{\prod_{l=2}^{k+1}\tau_l^2}.
\end{equation*}
This complete the proof of \eqref{Eqn::Basis::FinalBasisEst::Top}.

Replacing $\dbar^\top \widehat Q$ by $\dbar \widehat Q$ and \eqref{Eqn::Basis::EstQ2::Top} by \eqref{Eqn::Basis::EstQ2::Bot}, the above argument yields \eqref{Eqn::Basis::FinalBasisEst::Bot}.
\end{proof}
The same estimates hold if we swap $z$ and $\zeta$.
\begin{cor}\label{Cor::Basis::FinalBasisEst}
By enlarging the constant $C_j>0$ in Lemma~\ref{Lem::Basis::FinalBasisEst} if necessary, the estimates \eqref{Eqn::Basis::FinalBasisEst::Top} and \eqref{Eqn::Basis::FinalBasisEst::Bot} with $\tau_j(\zeta,\eps)$ replaced by $\tau_j(z,\eps)$, hold for all $z\in\Omega$ and $\zeta\in P_\eps(z)\backslash(P_{\eps/2}(z)\cup\Omega)$.
\end{cor}
\begin{proof}
    Indeed we still have $z\in\Omega$ and $\zeta\in U_1\backslash\overline\Omega$. By \eqref{Eqn::Basis::Prem::Set}, $\zeta\in P_\eps(z)\backslash P_{\eps/2}(z)$ implies $z\in P_{C_0\eps}(\zeta)\backslash P_{\eps/(2C_0)}(\zeta)$ and $\tau_j(\zeta,\eps)\ge\frac1{C_0}\tau_j(z,\eps)$ where $C_0$ is in Lemma~\ref{Lem::Basis::Prem}. The results then follow from Lemma~\ref{Lem::Basis::FinalBasisEst}.
\end{proof}

Recall $K_{q-1}(z,\zeta)$ in \eqref{Eqn::Goal::DefK}. Let $r_q:=(n-q+1) m_q+2q$ and $\gamma_q:=\frac{r_q}{r_q-1}$. Since $m_1\ge\dots\ge m_{n-1}\ge2>m_n=1$, we see that
\begin{equation}\label{Eqn::Basis::OrderGamma}
    r_1\ge r_2\ge\dots\ge r_n,\qquad\text{thus }1<\gamma_1\le \gamma_2\le\dots\le \gamma_n.
\end{equation}

We can now integrate $K_{q-1}$ on some $\eps$-minimal ellipsoids.

\begin{lem}\label{Lem::Basis::IntEst} For every $j\ge0$ there is a $C_j\ge0$, such that
for  every $1\le q\le n-1$, $z\in \Omega$, $\zeta\in U_1\backslash\overline\Omega$ and $0<\eps\le\eps_0$,
\begin{gather}
\label{Eqn::Basis::IntEst::Top+}
    \int_{\Omega \cap P_\eps(\zeta)\backslash P_{\frac\eps2}(\zeta)}|D^j(K^\top_{q-1})(w,\zeta)|d\Vol_w+\int_{P_\eps(z)\backslash (P_{\frac\eps2}(z)\cup \Omega)}|D^j(K^\top_{q-1})(z,w)|d\Vol_w\le C_j\eps^{\frac1{m_q}+1-j};
    \\
\label{Eqn::Basis::IntEst::Bot+}
    \int_{\Omega \cap P_\eps(\zeta)\backslash P_{\frac\eps2}(\zeta)}|D^jK_{q-1}^\bot(w,\zeta)|d\Vol_w+\int_{P_\eps(z)\backslash(P_{\frac\eps2}(z)\cup \Omega)}|D^jK_{q-1}^\bot(z,w)|d\Vol_w\le C_j\eps^{\frac2{m_q}-j};
    \\
\label{Eqn::Basis::IntEst::Top0}
    \int_{\Omega \cap P_\eps(\zeta)\backslash P_{\frac\eps2}(\zeta)}|D^jK^\top_{q-1}(w,\zeta)|^{\gamma_q}d\Vol_w+\int_{P_\eps(z)\backslash (P_{\frac\eps2}(z)\cup \Omega)}|D^jK^\top_{q-1}(z,w)|^{\gamma_q}d\Vol_w\le (C_j\eps^{1-j})^{\gamma_q};
    \\
\label{Eqn::Basis::IntEst::Bot0}
    \int_{\Omega \cap P_\eps(\zeta)\backslash P_{\frac\eps2}(\zeta)}|D^jK^\bot_{q-1}(w,\zeta)|^{\gamma_q}d\Vol_w+\int_{P_\eps(z)\backslash (P_{\frac\eps2}(z)\cup \Omega)}|D^jK^\bot_{q-1}(z,w)|^{\gamma_q}d\Vol_w\le (C_j\eps^{\frac1{m_q}-j})^{\gamma_q}.
\end{gather}
Here $D^j=\{\partial_z^{\alpha'}\partial_{\bar z}^{\alpha''}\partial_\zeta^{\beta'}\partial_{\bar\zeta}^{\beta''}:|\alpha'|+|\alpha''|+|\beta'|+|\beta''|\le j\}$ is the collection of derivatives of all variables.
\end{lem}
\begin{proof}
Since $K_{q-1}=K^\top_{q-1}+K^\bot_{q-1}$ (see Remark~\ref{Rmk::Goal::TopBotFacts}) and the right hand side of \eqref{Eqn::Basis::IntEst::Top+} is smaller than the right hand side of \eqref{Eqn::Basis::IntEst::Bot+}. It is enough to estimate \eqref{Eqn::Basis::IntEst::Bot+} with $K^\bot_{q-1}$ replaced by $K_{q-1}$. The same replacement also works for \eqref{Eqn::Basis::IntEst::Bot0}.

    Recall from \eqref{Eqn::Goal::DefK} that $K_{q-1}$ is the linear combinations of \eqref{Eqn::Goal::MainKernel}. Recall from Definition~\ref{Defn::Goal::BotTop} and Remark~\ref{Rmk::Goal::TopBotFacts} that $K_{q-1}^\top$ is also the linear combinations of \eqref{Eqn::Goal::MainKernel} with $\dbar\widehat Q$ replaced by $\dbar^\top\widehat Q=(\dbar\widehat Q)^\top$. Therefore by Lemma~\ref{Lem::Basis::FinalBasisEst} and Corollary~\ref{Cor::Basis::FinalBasisEst}, for $z$ and $\zeta$ in the assumption and $w$ in the integrands (for $z$ and $\zeta$ respectively),
    \begin{gather*}
        |D^j K^\top_{q-1}(w,\zeta)|\lesssim\sum_{k=1}^{n-q}\frac{\eps^{-1-j}}{\prod_{l=2}^{k}\tau_l(\zeta,\eps)^2}\frac1{|w-\zeta|^{2n-2k-1}},\quad
        |D^j K^\top_{q-1}(z,w)|\lesssim\sum_{k=1}^{n-q}\frac{\eps^{-1-j}}{\prod_{l=2}^{k}\tau_l(z,\eps)^2}\frac1{|w-z|^{2n-2k-1}};
        \\
        |D^j K_{q-1}(w,\zeta)|\lesssim\sum_{k=1}^{n-q}\frac{\eps^{-2-j}}{\prod_{l=2}^{k}\tau_l(\zeta,\eps)^2}\frac{\tau_{k+1}(\zeta,\eps)}{|w-\zeta|^{2n-2k-1}},\quad
        |D^j K_{q-1}(z,w)|\lesssim\sum_{k=1}^{n-q}\frac{\eps^{-2-j}}{\prod_{l=2}^{k}\tau_l(z,\eps)^2}\frac{\tau_{k+1}(\zeta,\eps)}{|w-z|^{2n-2k-1}}.
    \end{gather*}
    
    Taking an $\eps$-minimal basis at $\zeta$ and at $z$ respectively, $P_\eps(\zeta)$ and $P_\eps(z)$ are mapped into the subset\\ $\{u:|u_1|<\tau_1,\dots,|u_n|<\tau_n\}$ where $\tau_l=\tau_l(\zeta,\eps)$ or $\tau_l(z,\eps)$. By \eqref{Eqn::Basis::Prem::TypeBdd} we have $\tau_l\lesssim\eps^{1/m_{n-l+1}}$. Recall that $1=m_n<m_{n-1}\le\dots \le m_1$ from Definition~\ref{Defn::Basis::Type}. 
    
    Therefore \eqref{Eqn::Basis::IntEst::Top+} and \eqref{Eqn::Basis::IntEst::Bot+} are given by the following
    (see also \cite[Section~3]{HeferMultitype}):
    \begin{equation}\label{Eqn::Basis::PfIntEst::Pf+}
        \begin{aligned}
        &\int_{P_\eps(\zeta)}\sum_{k=1}^{n-q}\frac{d\Vol(w)}{|w-\zeta|^{2n-2k-1}\prod_{l=2}^{k}\tau_l(\zeta,\eps)^2}+\int_{P_\eps(z)}\sum_{k=1}^{n-q}\frac{d\Vol(w)}{|w-z|^{2n-2k-1}\prod_{l=2}^{k}\tau_l(\zeta,\eps)^2}
        \\
        \approx&\sum_{k=1}^{n-q}\int_{|w_1|<\tau_1,\dots,|w_n|<\tau_n}\frac{d\Vol(w_1,\dots,w_n)}{\big(\prod_{l=2}^{k}\tau_l^2\big)\cdot\big(\sum_{l=1}^n|w_l|\big)^{2n-2k-1}}
        \\
        \lesssim&\sum_{k=1}^{n-q}\tau_1^2\int_{|w_{k+1}|<\tau_{k+1},\dots,|w_n|<\tau_n}\frac{d\Vol(w_{k+1},\dots,w_n)}{\big(\sum_{l=k+1}^n|w_l|\big)^{2n-2k-1}}
        \\
        \lesssim&\eps^2\sum_{k=1}^{n-q}\int_0^{\tau_{k+1}}tdt\int_0^\infty\frac{s^{2n-2k-3}ds}{(t+s)^{2n-2k-1}}
        &\hspace{-2.5in}(\tau_1\approx\eps,\ t=|w_{k+1}|,\ s=|(w_{k+2},\dots,w_n)|)
        \\
        \lesssim&\eps^2\sum_{k=1}^{n-q} \int_0^{\tau_{k+1}}dt=\eps^2\sum_{k=1}^{n-q} \tau_{k+1}\lesssim\eps^2\sum_{k=1}^{n-q}\eps^{\frac1{m_{n-k}}}\approx \eps^{2+\frac1{m_q}}&\hspace{-2in}(\text{by }\eqref{Eqn::Basis::Prem::TypeBdd}).
    \end{aligned}
    \end{equation}
    Multiplying $\eps^{-1-j}$ we get \eqref{Eqn::Basis::IntEst::Top+}. Note that by \eqref{Eqn::Basis::Prem::TypeBdd} $\eps^{-2-j}\max\limits_{1\le k\le n-q}\tau_{k+1}^2\lesssim\eps^{\frac2{m_q}-2- j}$, thus \eqref{Eqn::Basis::IntEst::Bot+} follows.
    
    Similarly, \eqref{Eqn::Basis::IntEst::Top0} and \eqref{Eqn::Basis::IntEst::Bot0} are given by the following (cf. the control of $L_k$ in \cite[Section~4]{HeferMultitype}):
    \begin{equation}\label{Eqn::Basis::PfIntEst::Pf0}
        \begin{aligned}
        &\int_{P_\eps(\zeta)}\sum_{k=1}^{n-q}\frac{d\Vol(w)}{|w-\zeta|^{(2n-2k-1)\gamma_q}\prod_{l=2}^{k}\tau_l(\zeta,\eps)^{2\gamma_q}}+\int_{P_\eps(z)}\sum_{k=1}^{n-q}\frac{d\Vol(w)}{|w-z|^{(2n-2k-1)\gamma_q}\prod_{l=2}^{k}\tau_l(\zeta,\eps)^{2\gamma_q}}\hspace{-1.1in}
        \\
        \lesssim&\sum_{k=1}^{n-q}\int_{|w_1|<\tau_1,\dots,|w_n|<\tau_n}\frac{d\Vol(w_1,\dots,w_n)}{\big(\prod_{l=2}^{k}\tau_l^{2\gamma_q}\big)\cdot\big(\sum_{l=1}^n|w_l|\big)^{(2n-2k-1)\gamma_q}}
        \\
        \lesssim&\eps^{2}\sum_{k=1}^{n-q}\prod_{l=2}^k\frac1{\tau_l^{2(\gamma_q-1)}}\int_{|w_{k+1}|<\tau_{k+1},\dots,|w_n|<\tau_n}\frac{d\Vol(w_{l+1},\dots,w_n)}{\big(\sum_{l=1}^n|w_l|\big)^{(2n-2k-1)\gamma_q}}
        \\
        \lesssim&\eps^{2}\sum_{k=1}^{n-q}\eps^{-\frac12\cdot 2(k-1)(\gamma_q-1)}\int_0^{\tau_{k+1}}tdt\int_0^\infty\frac{s^{2n-2k-3}ds}{(t+s)^{(2n-2k-1)\gamma_q}}
        &(\eps^\frac12\lesssim\tau_2\le\tau_3\le\dots)
        \\
        \lesssim&\eps^{2\gamma_q}\sum_{k=1}^{n-q}\eps^{2-2\gamma_q-(k-1)(\gamma_q-1)}\int_0^{\tau_{k+1}}t^{(2n-2k-1)(1-\gamma_q)}dt
        \\
        \lesssim&\eps^{2\gamma_q}\sum_{k=1}^{n-q}\eps^{(k+1)(1-\gamma_q)}\tau_{k+1}^{(2n-2k-1)(1-\gamma_q)+1}
        \\
        \lesssim&\eps^{2\gamma_q}\sum_{p=q}^{n-1}\eps^{(n-p+1)(1-\gamma_q)+\frac1{m_p}(2p-1)(1-\gamma_q)+\frac1{m_p}}&\hspace{-0.4in}(\tau_{k+1}\lesssim\eps^\frac1{m_{n-k}},\ p=n-k)
        \\
        \lesssim&\eps^{2\gamma_q}\sum_{p=q}^{n-1}\eps^{(n-p+1)(1-\gamma_p)+\frac1{m_p}(2p-1)(1-\gamma_p)+\frac1{m_p}}&(\text{by }\eqref{Eqn::Basis::OrderGamma}).
    \end{aligned}
    \end{equation}

    The last inequality above equals to $\eps^{2\gamma_q}$ itself, because for every $1\le p\le n-1$,
    \begin{equation}\label{Eqn::Basis::PfIntEst::Pf0Num}
        \textstyle (n-p+1)(1-\gamma_p)+\frac1{m_p}(2p-1)(1-\gamma_p)+\frac1{m_p}=\frac{(n-p+1)m_p+2p-1}{m_p}(1-\gamma_p)+\frac1{m_p}=-\frac{r_p-1}{m_p(r_p-1)}+\frac1{m_p}=0.
    \end{equation}
    Thus multiplying $\eps^{-(1+j)\gamma_q}$ to \eqref{Eqn::Basis::PfIntEst::Pf0} we obtain \eqref{Eqn::Basis::IntEst::Top0}. Again by \eqref{Eqn::Basis::Prem::TypeBdd} $\eps^{-1-j}\max\limits_{1\le k\le n-q}\tau_{k+1}\lesssim\eps^{-1-j+\frac1{m_q}}$. Multiplying $\eps^{(\frac1{m_q}-1-j)\gamma_q}$ to \eqref{Eqn::Basis::PfIntEst::Pf0} we get \eqref{Eqn::Basis::IntEst::Bot0}. 
\end{proof}

We can now prove Theorem \ref{Thm::WeiEst} by taking the sums over $\eps$-minimal ellipsoids.

\begin{proof}[Proof of Theorem \ref{Thm::WeiEst}]Note that the constants $\eps_0,C_0>0$ in Lemma~\ref{Lem::Basis::Prem} depend only on $\Omega$, $\varrho$ and $\widehat S$ but not on $z$ and $\zeta$. We can replace the domains of the integrals \eqref{Eqn::WeiEst::Top+1} - \eqref{Eqn::WeiEst::Bot02} by $z\in\Omega\cap P_{\eps_0}(\zeta)$ and $\zeta\in P_{\eps_0}(z)\backslash\overline\Omega$: indeed by construction
\begin{equation*}
    \sup_{z\in\Omega,\zeta\in U_1\backslash\overline\Omega;\,|z-\zeta|\ge\eps_0/C_0}|D^k_{z,\zeta}K^\top_{q-1}(z,\zeta)|+|D^k_{z,\zeta}K^\bot_{q-1}(z,\zeta)|<\infty,\quad k\ge0.
\end{equation*}

% In particular we can assume $\zeta\in U_1\backslash\overline\Omega$ in \eqref{Eqn::WeiEst::Top+2}, \eqref{Eqn::WeiEst::Bot+2}, \eqref{Eqn::WeiEst::Top02} and \eqref{Eqn::WeiEst::Bot02}. 

The proof of \eqref{Eqn::WeiEst::Top+1} and \eqref{Eqn::WeiEst::Top+2} both follow from the same argument. Similarly for the rest, thus we only need to prove \eqref{Eqn::WeiEst::Top+1}, \eqref{Eqn::WeiEst::Bot+1}, \eqref{Eqn::WeiEst::Top01} and \eqref{Eqn::WeiEst::Bot01}.

Let $z\in\Omega$ with $\dist(z)<\eps_0$. Let $J\in\Z$ be the unique number such that $2^{-J}\eps_0\le\varrho(z)<2^{1-J}\eps_0$. Therefore $P_{2^{-J}\eps_0}(z)\subseteq\Omega$ and $\zeta\in P_\eps(z)\Rightarrow \dist(\zeta)\lesssim\eps$ for all $0<\eps\le\eps_0$.

Applying \eqref{Eqn::Basis::IntEst::Top+} we get \eqref{Eqn::WeiEst::Top+1}:
\begin{equation}\label{Eqn::Basis::PfWeiEst::Pf+}
    \begin{aligned}
    &\int_{P_{\eps_0}(z)\backslash\overline\Omega}\dist(\zeta)^s|D^kK^\top_{q-1}(z,\zeta)|d\Vol_\zeta\lesssim_{k}\sum_{j=1}^J\int_{P_{2^{1-j}\eps_0}(z)\backslash(P_{2^{-j}\eps_0}(z)\cup\Omega)}(2^{-j}\eps_0)^s|D^kK^\top_{q-1}(z,\zeta)|d\Vol_\zeta
    \\
    &\qquad\lesssim_k\sum_{j=1}^J(2^{-j}\eps_0)^s(2^{-j}\eps_0)^{\frac1{m_q}+1-k}\lesssim_{\eps_0}2^{-J(s+1+\frac1{m_q}-k)}\approx\dist(z)^{s+1+\frac1{m_q}-k}.
\end{aligned}
\end{equation}

Applying \eqref{Eqn::Basis::IntEst::Top0} we get \eqref{Eqn::WeiEst::Top01}:
\begin{equation}\label{Eqn::Basis::PfWeiEst::Pf0}
    \begin{aligned}
    &\int_{P_{\eps_0}(z)\backslash\overline\Omega}\hspace{-0.1in}|\dist(\zeta)^sD^kK^\top_{q-1}(z,\zeta)|^{\gamma_q}d\Vol_\zeta\lesssim_{k}\sum_{j=1}^J\int_{P_{2^{1-j}\eps_0}(z)\backslash(P_{2^{-j}\eps_0}(z)\cup\Omega)}\hspace{-0.3in}(2^{-j}\eps_0)^{s\gamma_q}|D^kK^\top_{q-1}(z,\zeta)|^{\gamma_q}d\Vol_\zeta
    \\
    &\qquad\lesssim_k\sum_{j=1}^J(2^{-j}\eps_0)^{s\gamma_q}(2^{-j}\eps_0)^{(1-k)\gamma_q}\lesssim_{\eps_0}2^{-J(s+1-k)\gamma_q}\approx\dist(z)^{(s+1-k)\gamma_q}.
\end{aligned}
\end{equation}

Using \eqref{Eqn::Basis::IntEst::Bot+} and \eqref{Eqn::Basis::IntEst::Bot0}, the same arguments show that
\begin{align*}
    \int_{P_{\eps_0}(z)\backslash\overline\Omega}\dist(\zeta)^s|D^kK^\bot_{q-1}(z,\zeta)|d\Vol(\zeta)&\lesssim 2^{-J(s+\frac2{m_q}-k)}\approx\dist(z)^{s+\frac2{m_q}-k};\\
    \int_{P_{\eps_0}(z)\backslash\overline\Omega}|\dist(\zeta)^sD^kK^\bot_{q-1}(z,\zeta)|^{\gamma_q}d\Vol(\zeta)&\lesssim 2^{-J(s-k+\frac1{m_q})\gamma_q}\approx\dist(z)^{(s-k+\frac1{m_q})\gamma_q}.
\end{align*}
These two equations give \eqref{Eqn::WeiEst::Bot+1} and \eqref{Eqn::WeiEst::Bot01}. The proof is now complete.
\end{proof}

Theorem~\ref{Thm::WeiEst} implies the following weighted boundedness between Sobolev spaces:
\begin{cor}\label{Cor::PfThm::WeiSobEst}
Let $1\le q\le n-1$ and $\alpha\in\N^{2n}(=\N^n_\zeta\times\N^n_{\bar \zeta})$.
We define integral operators\\ $\Kc_{q,\alpha}^\top,\Kc_{q,\alpha}^\bot:L^1(U_1\backslash\overline\Omega;\wedge^{0,q+1})\to C_\loc^0(\Omega;\wedge^{0,q-1})$ by
\begin{equation}\label{Eqn::PfThm::WeiSobEst::DefK}
    \Kc_{q,\alpha}^\top g(z):=\int_{U_1\backslash\overline\Omega} (D^\alpha_\zeta (K_{q-1}^\top))(z,\cdot)\wedge g;\quad \Kc_{q,\alpha}^\bot g(z):=\int_{U_1\backslash\overline\Omega} (D^\alpha_\zeta (K_{q-1}^\bot))(z,\cdot)\wedge g.
\end{equation}
Let $\dist(w):=\dist(w,b\Omega)$. Then for every $k\ge0$ and $1<s<k+|\alpha|-\frac1{m_q}$ {\normalfont(in particular $k+|\alpha|\ge2$)},
\begin{align}
\label{Eqn::PfThm::WeiSobEst::Top+}
    \Kc_{q,\alpha}^\top&:L^p(U_1\backslash\overline\Omega,\dist^{1-s};\wedge^{0,q+1})\to W^{k,p}(\Omega,\dist^{k+|\alpha|-\frac1{m_q}-s};\wedge^{0,q-1}),
    & \forall 1\le p\le\infty;
    \\
\label{Eqn::PfThm::WeiSobEst::Bot+}
    \Kc_{q,\alpha}^\bot&:L^p(U_1\backslash\overline\Omega,\dist^{\frac1{m_q}-s};\wedge^{0,q+1})\to W^{k,p}(\Omega,\dist^{k+|\alpha|-\frac1{m_q}-s};\wedge^{0,q-1}),
    &\forall 1\le p\le\infty;
    \\
\label{Eqn::PfThm::WeiSobEst::Top0}
    \Kc_{q,\alpha}^\top&:L^p(U_1\backslash\overline\Omega,\dist^{1-s};\wedge^{0,q+1})\to W^{k,\frac{pr_q}{r_q-p}}(\Omega,\dist^{k+|\alpha|-s};\wedge^{0,q-1}),
 &\forall 1\le p\le r_q;
    \\
\label{Eqn::PfThm::WeiSobEst::Bot0}
    \Kc_{q,\alpha}^\bot&:L^p(U_1\backslash\overline\Omega,\dist^{\frac1{m_q}-s};\wedge^{0,q+1})\to W^{k,\frac {pr_q}{r_q-p}}(\Omega,\dist^{k+|\alpha|-s};\wedge^{0,q-1}),
    &\forall 1\le p\le r_q.
\end{align}
\end{cor}
\begin{rem}
    Using integration by parts we get the relation $$\Kc_{q,\alpha}^{(\top,\bot)}g=(-1)^{|\alpha|}\Kc_{q,0}^{(\top,\bot)}\circ D^\alpha g,\quad\text{for all }g\in C_c^\infty(U_1\backslash\overline\Omega;\wedge^{0,q+1}).$$ Therefore \eqref{Eqn::PfThm::WeiSobEst::Top+} - \eqref{Eqn::PfThm::WeiSobEst::Bot0} can be restated as, for every $k,l\ge0$ and $1<s<k+l-\frac1{m_q}$ (in particular $k+l\ge2$),
    \begin{align*}
    \Kc_{q,0}^\top&:\widetilde W^{l,p}(\overline U_1\backslash\Omega,\dist^{1-s};\wedge^{0,q+1})\to W^{k,p}(\Omega,\dist^{k+l-\frac1{m_q}-s};\wedge^{0,q-1}),
    & \forall 1\le p\le\infty;
    \\
    \Kc_{q,0}^\bot&:\widetilde W^{l,p}(\overline U_1\backslash\Omega,\dist^{\frac1{m_q}-s};\wedge^{0,q+1})\to W^{k,p}(\Omega,\dist^{k+l-\frac1{m_q}-s};\wedge^{0,q-1}),
    &\forall 1\le p\le\infty;
    \\
    \Kc_{q,0}^\top&:\widetilde W^{l,p}(\overline U_1\backslash\Omega,\dist^{1-s};\wedge^{0,q+1})\to W^{k,\frac{pr_q}{r_q-p}}(\Omega,\dist^{k+l-s};\wedge^{0,q-1}),
 &\forall 1\le p\le r_q;
    \\
    \Kc_{q,0}^\bot&:\widetilde W^{l,p}(\overline U_1\backslash\Omega,\dist^{\frac1{m_q}-s};\wedge^{0,q+1})\to W^{k,\frac {pr_q}{r_q-p}}(\Omega,\dist^{k+l-s};\wedge^{0,q-1}),
    &\forall 1\le p\le r_q.
    \end{align*}
    Here $\widetilde W^{l,p}(\overline U,\varphi):=\{ g\in W^{l,p}(\R^N,\varphi):g|_{\overline U^c}=0 \}$ follow the notations in Definition \ref{Defn::Space::TLSpace}.
\end{rem}

Corollary~\ref{Cor::PfThm::WeiSobEst} follows almost immediately by Schur's test:
\begin{lem}[Schur's test]\label{Lem::PfThm::Schur}
    Let $(X,\mu)$ and $(Y,\nu)$ be two measure spaces. Let $G\in L^1_\loc(X\times Y,\mu\otimes \nu)$, $1\le \gamma<\infty$ and $A>0$ that satisfy
    \begin{equation*}
        \essup_{y\in Y}\int_X|G(x,y)|^\gamma d\mu(x)\le A^\gamma;\quad \essup_{x\in X}\int_Y|G(x,y)|^\gamma d\nu(y)\le A^\gamma.
    \end{equation*}
    
    Then the integral operator $Tf(y):=\int_XG(x,y)f(x)d\mu(x)$ has boundedness $T:L^p(X,d\mu)\to L^q(Y,d\nu)$, with operator norm $\|T\|_{L^p\to L^q}\le A$ for all $1\le p,q\le\infty$ such that $\frac1q=\frac1p+\frac1\gamma-1$.
\end{lem}
See for example \cite[Appendix B]{RangeSCVBook}. Note that the norm of $L^p(X,\mu)$ here is $(\int_X|f|^pd\mu)^\frac1p$, while the norm of $L^p(\Omega,\varphi)$ in Definition~\ref{Defn::Space::WeiSob} is $(\int_\Omega|\varphi f|^pd\Vol)^{1/p}$. When $p<\infty$ we have the correspondence $d\mu=|\varphi|^p\cdot d\Vol$.

\begin{proof}[Proof of Corollary~\ref{Cor::PfThm::WeiSobEst}]
For $\beta\in\N^{2n}_{z,\bar z}$, we have $D^\beta_z\Kc_{q,\alpha}^\top g(z)=\int (D^\beta_zD^\alpha_\zeta K^\top_{q-1})(z,\cdot)\wedge g$.

Applying Lemma~\ref{Lem::PfThm::Schur} to \eqref{Eqn::WeiEst::Top+1} and \eqref{Eqn::WeiEst::Top+2} with $(X,\mu)=(U_1\backslash\overline\Omega,\Vol_\zeta)$, $(Y,\nu)=(\Omega,\Vol_z)$, $\gamma=1$ and $G(\zeta,z)=\dist_{\Omega^c}(z)^{|\alpha|+k-s-\frac1{m_q}}\cdot(D^k_zD^\alpha_\zeta K_{q-1}^\top)(z,\zeta)\cdot\dist_\Omega(\zeta)^{s-1}$, we see that for every $k\ge0$ and\\ $1<s<k+|\alpha|-1/m_q$,
\begin{equation*}
    \big[g\mapsto \dist^{|\alpha|+|\beta|-s-\frac1{m_q}}\cdot\big( D^k_z\Kc^\top_{q,\alpha}(\dist^{s-1}\cdot g)\big)\big]:L^p(U_1\backslash\overline\Omega)\to L^p(\Omega),\quad\forall 1\le p\le\infty.
\end{equation*}
This is the same as saying $D^k_z\Kc_{q,\alpha}^\top:L^p(U_1\backslash\overline\Omega,\dist^{1-s})\to L^p(\Omega,\dist^{|\alpha|+k-s-1/m_q})$. Thus \eqref{Eqn::PfThm::WeiSobEst::Top+} follows.

Similarly, applying Lemma~\ref{Lem::PfThm::Schur} to \eqref{Eqn::WeiEst::Top01} and \eqref{Eqn::WeiEst::Top02} with $(X,\mu)=(U_1\backslash\overline\Omega,\Vol_\zeta)$, $(Y,\nu)=(\Omega,\Vol_z)$, $\gamma=\frac{r_q}{r_q-1}$ and $G(\zeta,z)=\dist_{\Omega^c}(z)^{|\alpha|+k-s}\cdot(D^k_zD^\alpha_\zeta K_{q-1}^\top)(z,\zeta)\cdot\dist_\Omega(\zeta)^{s-1}$ we get \eqref{Eqn::PfThm::WeiSobEst::Top0}.

Repeating the same arguments while replacing $\dist(\zeta)^{s-1}$ by $\dist(\zeta)^{s-\frac1{m_q}}$ and $K^\top$ by $K^\bot$, we get \eqref{Eqn::PfThm::WeiSobEst::Bot+} and \eqref{Eqn::PfThm::WeiSobEst::Bot0}.
\end{proof}

\begin{rem}
    By keeping track of the proof, the implied constants in Theorem~\ref{Thm::WeiEst} and the operator norms in Corollary~\ref{Cor::PfThm::WeiSobEst} depend only on the $C_0$ in Lemma~\ref{Lem::Basis::Prem}, the $C_1$ in Lemma~\ref{Lem::Basis::EstSQ}, and the upper bound of $\|\varrho\|_{C^{m+k+2}}$.

    More generally, whenever the smooth holomorphic support function $\widehat S(z,\zeta)$, along with the corresponding Leray map $\widehat Q_j(z,\zeta):=\int_0^1\frac{\partial S}{\partial z_j}(\zeta+t(z-\zeta),\zeta)dt$ ($j=1,\dots,n$), satisfies the estimates in Lemma~\ref{Lem::Basis::EstSQ}, then the kernel $K(z,\zeta)$ given by \eqref{Eqn::Goal::DefK} would fulfill the same weighted estimates as in Theorem~\ref{Thm::WeiEst}.
\end{rem}

\section{Function Spaces and Extension Operators}\label{Section::Space}
In this section we focus on the real domain $\R^N\simeq\C^n$ where $N=2n$. 
\begin{note}\label{Note::Space::Dist}
We denote by $\Ss'(\R^N)$ the space of tempered distributions, and for an arbitrary open subset $U\subseteq\R^n$, we denote by $\Ss'(U):=\{\tilde f|_U:\tilde f\in\Ss'(\R^N)\}\subsetneq\Ds'(U)$ the space of distributions in $U$ which can be extended to tempered distributions in $\R^N$. (See also \cite[(3.1) and Propotion 3.1]{RychkovExtension}.)
\end{note}

We first recall the classical Sobolev and H\"older spaces. The characterizations in Definitions \ref{Defn::Space::Sob} and \ref{Defn::Space::Hold} are not directly used in the paper.

\begin{defn}[Weighted Sobolev]\label{Defn::Space::WeiSob}
Let $U\subseteq\R^N$ be an arbitrary open set. Let $\varphi:U\to[0,\infty)$ be a non-negative continuous function, we define for $k\ge0$ and $1\le p\le\infty$,
\begin{equation}\label{Eqn::Intro::WeiSobo}
\begin{gathered}
W^{k,p}(U,\varphi):=\{f\in W^{k,p}_\loc(U):\|f\|_{W^{k,p}(U,\varphi)}<\infty\},
\\
\|f\|_{W^{k,p}(U,\varphi)}:=\bigg(\sum_{|\alpha|\le k}\int_U |\varphi\partial^\alpha f|^p\bigg)^\frac1p \quad 1 \le p < \infty;\quad
\|f\|_{W^{k,\infty}(U,\varphi)}:=\sup\limits_{|\alpha|\le k}\|\varphi\partial^\alpha f\|_{L^\infty(U)}.
\end{gathered}
\end{equation}
We define $W^{k,p}(U):=W^{k,p}(U,\1)$ where $\1=\1_{\R^N}$ is the constant function.
\end{defn}
\begin{defn}[Sobolev-Bessel]\label{Defn::Space::Sob}
    Let $s\in\R$. For $1<p<\infty$, we define the Bessel potential space $H^{s,p}(\R^N)$ to be the set of all tempered distributions $f\in\Ss'(\R^N)$ such that
    \begin{equation*}
        \|f\|_{H^{s,p}(\R^N)}:=\|(I-\Delta)^\frac s2f\|_{L^p(\R^N)}<\infty.
    \end{equation*}
    Here we use the standard (negative) Laplacian $\Delta=\sum_{j=1}^N\partial_{x_j}^2$.
    
    % We define $H^{s,1}(\R^N)=\{f:\|(I-\Delta)^\frac s2f\|_{h^1(\R^N)}<\infty\}$ and $H^{s,\infty}(\R^N)=\{f:\|(I-\Delta)^\frac s2f\|_{\bmo(\R^N)}<\infty\}$ via local Hardy space $h^1$ and local BMO space $\bmo$, which consist of $f\in\Ss'\cap L^1_\loc(\R^N)$ such that the following norms are finite respectively (see \cite[Theorems 1 and 4(3)]{GoldbergLocalHardy}):
    % \begin{gather*}
    %     \|f\|_{h^1(\R^N)}:=\int_{\R^N}\sup_{0<t<1;|y-x|\le t}|e^{-t\sqrt{-\Delta}}f(y)|dx;
    %     \\
    %     \|f\|_{\bmo(\R^N)}:=\sup_{x\in\R^N;0<r<1}\frac1{|B(0,r)|}\int_{B(x,r)}\Big|f-{\textstyle\frac1{|B(0,r)|}\int_{B(x,r)}f}\Big|+\sup_{x\in\R^N;r>1}\frac1{|B(0,r)|}\int_{B(x,r)}|f|.
    % \end{gather*}
    
    On an open subset $U\subseteq\R^N$, we define $H^{s,p}(U):=\{\tilde f|_U:\tilde f\in H^{s,p}(\R^N)\}$ for $s\in\R$, $1< p<\infty$ with norm $\|f\|_{H^{s,p}(U)}:=\inf_{\tilde f|_U=f}\|\tilde f\|_{H^{s,p}(\R^N)}$.

\end{defn}
\begin{defn}[H\"older-Zygmund]\label{Defn::Space::Hold}Let $U\subseteq\R^N$ be an open subset. We define the H\"older-Zygmund space $\Co^s(U)$ for $s\in\R$ by the following:
\begin{itemize}
    \item For $0<s<1$, $\Co^s(U)$ consists of all $f\in C^0(U)$ such that $\|f\|_{\Co^s(U)}:=\sup\limits_U|f|+\sup\limits_{x,y\in U}\frac{|f(x)-f(y)|}{|x-y|^s}<\infty$.
    \item $\Co^1(U)$ consists of all $f\in C^0(U)$ such that $\|f\|_{\Co^1(U)}:=\sup\limits_U|f|+\sup\limits_{x,y\in U;\frac{x+y}2\in U}\frac{|f(x)+f(y)-2f(\frac{x+y}2)|}{|x-y|}<\infty$.
    \item For $s>1$ recursively, $\Co^s(U)$ consists of all $f\in \Co^{s-1}(U)$ such that $\nabla f\in\Co^{s-1}(U;\C^N)$. We define $\|f\|_{\Co^s(U)}:=\|f\|_{\Co^{s-1}(U)}+\sum_{j=1}^N\|D_j f\|_{\Co^{s-1}(U)}$.
    \item For $s\le0$ recursively, $\Co^s(U)$ consists of all distributions that have the form $g_0+\sum_{j=1}^N\partial_jg_j$ where $g_0,\dots,g_N\in \Co^{s+1}(U)$. We define $\|f\|_{\Co^s(U)}:=\inf\{\sum_{j=0}^N\|g_j\|_{\Co^{s+1}(U)}:f=g_0+\sum_{j=1}^N\partial_j g_j\in\Ds'(U)\}$.
    \item We define $\Co^\infty(U):=\bigcap_{s>0}\Co^s(U)$ be the space of bounded smooth functions.
\end{itemize}
    
\end{defn}

In the paper we consider a more general version of the function spaces: the Triebel-Lizorkin spaces.
\begin{defn}[Triebel-Lizorkin]\label{Defn::Space::TLSpace}
    Let $\lambda=(\lambda_j)_{j=0}^\infty$ be a sequence of Schwartz functions satisfying:
\begin{enumerate}[label=(\thesection.\arabic*)]\setcounter{enumi}{\value{equation}}
    \item\label{Item::Space::LambdaFour} The Fourier transform $\hat\lambda_0(\xi)=\int_{\R^n}\lambda_0(x)2^{-2\pi ix\xi}dx$ satisfies $\supp\hat\lambda_0\subset\{|\xi|<2\}$ and $\hat\lambda_0|_{\{|\xi|<1\}}\equiv1$.
    \item\label{Item::Space::LambdaScal}  $\lambda_j(x)=2^{jn}\lambda_0(2^jx)-2^{(j-1)n}\lambda_0(2^{j-1}x)$ for $j\ge1$.\setcounter{equation}{\value{enumi}}
\end{enumerate}

Let $0< p,q\le\infty$ and $s\in\R$, we define the Triebel-Lizorkin norm $\|\cdot\|_{\Fs_{pq}^s(\lambda)}$ as
\begin{align}%\label{Eqn::Space::BsNorm}
%\|f\|_{\Bs_{pq}^s(\lambda)}&:=\|(2^{js}\lambda_j\ast f)_{j=0}^\infty\|_{\ell^q(\N;L^p(\R^N))}=\bigg(\sum_{j=0}^\infty2^{jsq}\Big(\int_{\R^N}|\lambda_j\ast f(x)|^pdx\Big)^\frac qp\bigg)^\frac1q;
%\\
\label{Eqn::Space::TLNorm1}
    \|f\|_{\Fs_{pq}^s(\lambda)}&:=\|(2^{js}\lambda_j\ast f)_{j=0}^\infty\|_{L^p(\R^N;\ell^q(\N))}=\bigg(\int_{\R^N}\Big(\sum_{j=0}^\infty|2^{js}\lambda_j\ast f(x)|^q\Big)^\frac pqdx\bigg)^\frac1p,&p<\infty;
    \\
    \label{Eqn::Space::TLNorm2}
    \|f\|_{\Fs_{\infty q}^s(\lambda)}&:=\sup_{x\in\R^N,J\in\Z}2^{NJ\frac1q}\|(2^{js}\lambda_j\ast f)_{j=\max(0,J)}^\infty\|_{L^q(B(x,2^{-J});\ell^q)},&p=\infty.
\end{align}
Here for $q=\infty$ we take the usual modifications: we replace the $\ell^q$ sum by the supremum over $j$.

We define $\Fs_{pq}^s(\R^N)$ with its norm given by a fixed choice of $\lambda$.

For an arbitrary open subset $U\subseteq\R^N$, we define
\begin{align*}
    \Fs_{pq}^s(U)&:=\{\tilde f|_U:\tilde f\in\Fs_{pq}^s(\R^N)\}&\text{with}\quad \|f\|_{\Fs_{pq}^s(U)}:=\inf_{\tilde f|_U=f}\|\tilde f\|_{\Fs_{pq}^s(\R^N)};
    \\
    \widetilde\Fs_{pq}^s(\overline U)&:=\{f\in \Fs_{pq}^s(\R^N):f|_{\overline{U}^c}=0\}&\text{ as a closed subspace of }\Fs_{pq}^s(\R^N).
\end{align*}
\end{defn}
\begin{rem}\label{Rmk::Space::TLRmk}
\begin{enumerate}[(i)]
    \item When $p$ or $q<1$, \eqref{Eqn::Space::TLNorm1} and \eqref{Eqn::Space::TLNorm2} are only quasi-norms. For convenience we still use the terminology ``norms'' to refer them.
    \item Different choices of $\lambda$ results in equivalent norms. See \cite[Proposition~2.3.2]{TriebelTheoryOfFunctionSpacesI} and \cite[Propositions 1.3 and 1.8]{TriebelTheoryOfFunctionSpacesIV}.
    \item\label{Item::Space::TLRmk::Embed} We have embedding  $\Fs_{pq_1}^s(\R^N)\hookrightarrow\Fs_{pq_2}^s(\R^N)\hookrightarrow\Fs_{pq_1}^{s-\delta}(\R^N)$ for all $0<p\le\infty$, $s\in\R$, $q_1\le q_2$, $\delta>0$. 
    \\
    We have embedding $\Fs_{p_1q_1}^s(\R^N)\hookrightarrow\Fs_{p_1q_2}^{s-N(\frac1{p_1}-\frac1{p_2})}(\R^N)$ for all $p_1<p_2\le\infty$, $s\in\R$, $0< q_1,q_2\le\infty$.
    
    See \cite[Corollary~2.7]{TriebelTheoryOfFunctionSpacesIV} for an illustration. Taking restrictions to an arbitrary open subset $U$ we have $\Fs_{pq_1}^s(U)\hookrightarrow\Fs_{pq_2}^s(U)\hookrightarrow\Fs_{pq_1}^{s-\delta}(U)$ for $q_1\le q_2$ and $\delta>0$, and $\Fs_{p\infty}^s(U)\hookrightarrow\Fs_{r\eps}^{s-1}(U)$ for $\eps>0$ and $0<p\le r\le Np/\max(N-p,0)$.
    
    % In Section~\ref{Section::SPsiCX} we use Sobolev embedding $\Fs_{p,\infty}^{s+1}(\R^N)\hookrightarrow\Fs_{\frac{pN}{N-p},\eps}^{s}(\R^N)$ for $1\le p\le N$, also see \cite[Corollary~2.7]{TriebelTheoryOfFunctionSpacesIV} for the proof. Thus $\Fs_{p,\infty}^{s+1}(U)\hookrightarrow\Fs_{\frac{pN}{N-p},\eps}^{s}(U)$.
    \item When $p=q=\infty$, \eqref{Eqn::Space::TLNorm2} can be written as $\|f\|_{\Fs_{\infty\infty}^s(\lambda)}=\sup_{j\ge0}\|2^{js}\lambda_j\ast f\|_{L^\infty(\R^N)}$. It is more common to be referred as the \textit{Besov norm} $\Bs_{\infty\infty}^s$. See also \cite[(1.15)]{TriebelTheoryOfFunctionSpacesIV} and \cite[Remark~2.3.4/3]{TriebelTheoryOfFunctionSpacesI}.
    \item For an arbitrary open subset $U\subseteq\R^N$, we have $\Fs_{pq}^s(\overline{U}^c)=\Fs_{pq}^s(\R^N)/\widetilde\Fs_{pq}^s(\overline U)$, or equivalently \\$\Fs_{pq}^s(U)=\Fs_{pq}^s(\R^N)/\widetilde\Fs_{pq}^s(U^c)$. See also \cite[Remark~4.3.2/1]{TriebelInterpolation}.
    \item\label{Item::Space::TLRmk::SobHold} When $\Omega\subseteq\R^N$ is either a bounded Lipschitz domain or the total space, we have following (see \cite[Sections 2.5.6 and 2.5.7]{TriebelTheoryOfFunctionSpacesI} and  \cite[Theorem~1.122]{TriebelTheoryOfFunctionSpacesIII}),
    \begin{itemize}
        \item $H^{s,p}(\Omega)=\Fs_{p2}^s(\Omega)$ for $s\in\R$ and $1< p<\infty$;
        \item $W^{k,p}(\Omega)=\Fs_{p2}^k(\Omega)$ for $k\in\N$ and $1<p<\infty$;
        \item $\Co^s(\Omega)=\Fs_{\infty\infty}^s(\Omega)$ for $s\in\R$;
        \item $\Co^{k+s}(\Omega)=C^{k,s}(\Omega)$ for $k\in\N$ and $0<s<1$.
    \end{itemize}
    We sketch the proof of $\Co^s=\Fs_{\infty\infty}^s$ of $s<0$, on $\R^N$ and on bounded Lipschitz domains below:
\end{enumerate}
\end{rem}
\begin{proof}[Proof of $\Co^s=\Fs_{\infty\infty}^s$ when $s\le0$] We fix an $s\le0$ below.

Let $k>-s/2$ be an integer. For a $f\in\Fs_{\infty\infty}^s(\R^N)$ we have $f=(I-\Delta)^k(I-\Delta)^{-k}f$ where by \cite[Theorem~2.3.8]{TriebelTheoryOfFunctionSpacesI} $(I-\Delta)^{-k}f\in\Fs_{\infty\infty}^{s+2k}(\R^N)=\Co^{s+2k}(\R^N)$ thus $f\in\Co^s(\R^N)$. Conversely for a $f=\sum_{|\alpha|\le\lceil -s\rceil+1}D^\alpha g_\alpha\in\Co^s(\R^N)$ where $g\in\Co^{s+\lceil -s\rceil+1}(\R^N)=\Fs_{\infty\infty}^{s+\lceil -s\rceil+1}(\R^N)$, by \cite[Theorem~2.3.8]{TriebelTheoryOfFunctionSpacesI} $D^\alpha g_\alpha\in\Fs_{\infty\infty}^s(\R^N)$ thus $f\in\Fs_{\infty\infty}^s(\R^N)$.
    
    For bounded Lipschitz $\Omega$, since we have $\Co^{s+2k}(\Omega)=\Fs_{\infty\infty}^{s+2k}(\Omega)$ and $\Co^{s+\lceil -s\rceil+1}(\Omega)=\Fs_{\infty\infty}^{s+\lceil -s\rceil+1}(\Omega)$, taking restrictions on both sides we get $\Co^s(\Omega)=\Fs_{\infty\infty}^s(\Omega)$.
\end{proof}

Remark~\ref{Rmk::Space::TLRmk} \ref{Item::Space::TLRmk::SobHold} can be illustrated via the extension operator. In the paper our convex domain $\Omega\subset\C^n$ is smooth and one can use the so-called \textit{half-space extension}, see Remark \ref{Rmk::LPThm::TangCommRmk} and \cite[Sections 2.9 and 3.3.4]{TriebelTheoryOfFunctionSpacesI} for example. In our case it is more preferable to use the \textit{Rychkov's extension}, which can also work on Lipschitz domains.
\begin{defn}\label{Defn::Space::ExtOp}
Let $\omega\subset\R^N$ be a \textit{special Lipschitz domain}\footnote{In literature like \cite[Definition~1.103]{TriebelTheoryOfFunctionSpacesIII} the condition $\|\nabla\sigma\|_{L^\infty}<1$ is not required. This can be achieved through an invertible linear transformation.}, that is $\omega=\{(x_1,x'):x_1>\sigma(x')\}$ for some $\sigma:\R^{N-1}\to\R$ such that $\|\nabla\sigma\|_{L^\infty}<1$.

The \textit{Rychkov's universal extension operator} $E=E_\omega$ for $\omega$ is given by the following:
\begin{equation}\label{Eqn::Space::ExtOp}
E_\omega f: =\sum_{j=0}^\infty\psi_j\ast(\1_{\omega}\cdot(\phi_j\ast f)),\qquad f\in\Ss'(\omega).
\end{equation}
Here $(\psi_j)_{j=0}^\infty$ and $(\phi_j)_{j=0}^\infty$ are families of Schwartz functions that satisfy the following properties: 
\begin{enumerate}[label=(\thesection.\arabic*)]\setcounter{enumi}{\value{equation}}
    \item\label{Item::Space::PhiScal} \textit{Scaling condition}: $\phi_j(x)=2^{(j-1)N}\phi_1(2^{j-1}x)$ and $\psi_j(x)=2^{(j-1)N}\psi_1(2^{j-1}x)$ for $j\ge2$.
	\item\label{Item::Space::PhiMomt} \textit{Moment condition}: $\int\phi_0=\int\psi_0=1$, $\int x^\alpha\phi_0(x)dx=\int x^\alpha\psi_0(x)dx=0$ for all multi-indices $|\alpha|>0$, and $\int x^\alpha\phi_1(x)dx=\int x^\alpha\psi_1(x)dx=0$ for all $|\alpha|\ge0$.
	\item\label{Item::Space::PhiApprox}\textit{Approximate identity}: $\sum_{j=0}^\infty\phi_j=\sum_{j=0}^\infty\psi_j\ast\phi_j=\delta_0$ is the Dirac delta measure.
	\item\label{Item::Space::PhiSupp} \textit{Support condition}: $\phi_j,\psi_j$ are all supported in the negative cone $-\Kb:=\{(x_1,x'):x_1<-|x'|\}$.\setcounter{equation}{\value{enumi}}
\end{enumerate}
\end{defn}

Such family $(\phi_j,\psi_j)_{j=0}^\infty$ exists, see \cite[Theorem~4.1(b) and Proposition~2.1]{RychkovExtension}.
\begin{prop}\label{Prop::Space::RychFact}Let $0<p,q\le\infty$ and $s\in\R$. Let $(\phi_j)_{j=0}^\infty$ be as in Definition \ref{Defn::Space::ExtOp}.
\begin{enumerate}[(i)]
    \item\label{Item::Space::RychFact::ExtBdd}{\normalfont(\cite[Theorem~4.1]{RychkovExtension})} $E_\omega:\Fs_{pq}^s(\omega)\to\Fs_{pq}^s(\R^N)$ is bounded, provided $(p,q)\notin\{\infty\}\times(0,\infty)$.
    \item\label{Item::Space::RychFact::EqvNorm}{\normalfont(\cite[Theorem~3.2]{RychkovExtension} and \cite[Theorem~1]{YaoExtensionMorrey})} There are intrinsic norms
    \begin{align*}
        %\|f\|_{\Bs_{pq}^s(\omega)}\approx&_{p,q,s}\|(2^{js}\phi_j\ast f)_{j=0}^\infty\|_{\ell^q(\N;L^p(\omega))};
        %\\
        \|f\|_{\Fs_{pq}^s(\omega)}\approx&_{p,q,s}\|(2^{js}\phi_j\ast f)_{j=0}^\infty\|_{L^p(\omega;\ell^q(\N))},&\text{provided }p<\infty;
        \\
        \|f\|_{\Fs_{\infty q}^s(\omega)}\approx&_{q,s}\sup_{x\in\R^N;J\in\Z}2^{NJ\frac1q}\|(2^{js}\phi_j\ast f)_{j=\max(0,J)}^\infty\|_{L^q(\omega\cap B(x,2^{-J});\ell^q)},&\text{for }p=\infty.
    \end{align*}
    \item\label{Item::Space::RychFact::EqvNorm2}{\normalfont(\cite[Proposition~6.6]{ShiYaoExt} and \cite[Theorem~2]{YaoExtensionMorrey})} For every $m\ge0$ there is an equivalent norm via derivatives $\|f\|_{\Fs_{pq}^s(\omega)}\approx_{p,q,s,m}\sum_{|\alpha|\le m}\|D^\alpha f\|_{\Fs_{pq}^{s-m}(\omega)}$.
\end{enumerate}
\end{prop}
In fact $E_\omega:\Fs_{\infty q}^s(\omega)\to\Fs_{\infty q}^s(\R^N)$ is also bounded, see \cite{ZhuoTriebelType}. We do not need this result in our paper.

In our case we work on the bounded domains instead of special type domains. For a bounded domain $\Omega$ we define its extension operator $\Ec_\Omega$ via partition of unity. For completeness we give the construction below.

\begin{note}[Objects for partition of unity]\label{Note::Space::PartUnity}
Let $\Omega\subset\R^N$ be a bounded Lipschitz domain and let $\Uc\Supset\Omega$ be a fixed open neighborhood. We use the following objects, which can all be obtained by standard partition of unity argument:
\begin{itemize}
    \item $(U_\nu)_{\nu=0}^M$ are finitely many bounded smooth open sets in $\Uc\subseteq\R^n$.
    \item$(\Phi_\nu:\R^N\to\R^N)_{\nu=1}^M$ as invertible affine linear transformations.
    \item $(\chi_\nu)_{\nu=0}^M$ are $C_c^\infty$-functions on $\R^N$ that take values in $[0,1]$.
    \item $(\omega_\nu)_{\nu=1}^M$ are special Lipschitz domains on $\R^N$.
\end{itemize}

They have the following properties: 
    \begin{enumerate}[label=(\thesection.\arabic*)]\setcounter{enumi}{\value{equation}}
        \item $b\Omega\subset\bigcup_{\nu=1}^M U_\nu$ and $U_0\Subset\Omega\Subset\bigcup_{\nu=0}^MU_\nu \Subset \Uc$.
        \item\label{Item::PartitionUnity::chi} $\chi_\nu\in C_c^\infty(U_\nu)$ for $0\le \nu\le M$, and $\sum_{\nu=0}^M\chi_\nu^2\equiv1$ in a neighborhood of $\overline\Omega$.
        \item\label{Item::PartitionUnity::Phi} For each $1\le \nu\le M $, $U_\nu=\Phi_\nu(B^N(0,1))$ and $U_\nu\cap\Omega=U_\nu\cap\Phi_\nu(\omega_\nu)$.\setcounter{equation}{\value{enumi}}
    \end{enumerate}
\end{note}

The partition of unity argument requires no other than the following fact:
\begin{lem}\label{Lem::Space::PartUnityLem}
    Let $U\subseteq\R^N$ be an arbitrary open subset, let $\Phi$ be an invertible linear transform and let $\chi\in\Co^\infty(\R^N)$ be a bounded smooth function.
    
    Then $Tf(x):=\chi(x)f(\Phi(x))$ defines a bounded linear map $T:\Fs_{pq}^s(U)\to\Fs_{pq}^s(\Phi^{-1}(U))$ for all $s\in\R$ and $0<p,q\le\infty$.
\end{lem}
\begin{proof}
We have boundedness $[g\mapsto g\circ\Phi]:\Fs_{pq}^s(\R^N)\to\Fs_{pq}^s(\R^N)$ from \cite[Theorem~2.25]{TriebelTheoryOfFunctionSpacesIV} and\\ $[g\mapsto \chi g]:\Fs_{pq}^s(\R^N)\to\Fs_{pq}^s(\R^N)$ from \cite[Theorem~2.28]{TriebelTheoryOfFunctionSpacesIV}. For $f\in\Fs_{pq}^s(U)$, let $\tilde f\in\Fs_{pq}^s(\R^N)$ be an extension of $f$, we see that $\chi\cdot(\tilde f\circ\Phi)$ is an extension to $Tf$ with respect to the domain $\Phi^{-1}(U)$. Therefore taking restrictions to $U$ and $\Phi^{-1}(U)$ respectively we get the boundedness of $T$.
\end{proof}
\begin{defn}\label{Defn::Space::ExtOmega}
    Let $\Omega$ be a bounded Lipschitz domain and let $\Uc\Supset\Omega$ be an open subset. Let $(\chi_\nu)_{\nu=0}^M$, $(\omega_\nu)_{\nu=1}^M$ and $(\Phi_\nu)_{\nu=1}^M$ be from above. We define the extension operator $\Ec=\Ec_\Omega$ for $\Omega$ as
       \begin{equation}\label{Eqn::Space::ExtOmega}
       \Ec_\Omega f:=\chi_0^2\cdot f+\sum_{\nu=1}^M\chi_\nu\cdot (E_{\omega_\nu}[(\chi_\nu f)\circ\Phi_\nu]\circ\Phi_\nu^{-1}),
   \end{equation}
   where $E_{\omega_\nu}$ is given by \eqref{Eqn::Space::ExtOp}.
\end{defn}

\begin{rem}\label{Rmk::Space::ExtBdd}Combining Proposition~\ref{Prop::Space::RychFact} and Lemma~\ref{Lem::Space::PartUnityLem},
$\Ec_\Omega:\Fs_{pq}^s(\Omega)\to\Fs_{pq}^s(\R^N)$ is also bounded for all $0<p,q\le\infty$ and $s\in\R$ such that $(p,q)\notin\{\infty\}\times(0,\infty)$ (for its proof see \cite[Lemma~6.3]{ShiYaoExt} for example). The $\Fs_{\infty q}^s$-boundedness is also true (see \cite{ZhuoTriebelType} and \cite[Remark~20]{YaoExtensionMorrey}), but we do not need it in our paper, since in applications we only need the extensions on $\Fs_{p\infty}^s$.
\end{rem}

Before the weighted estimates, if $f$ has low regularity then we need to express $[\dbar,\Ec]f$ and $([\dbar,\Ec]f)^\top$ as the sum of derivatives of good functions, and move those derivatives to $K_{q-1}(z,\zeta)$ via integration by parts. To make sure there are no boundary terms in the integration by part we need the following:
\begin{prop}[Anti-derivatives with support]\label{Prop::Space::ATDOp}
Let $\Omega\subset\R^N$ be a bounded Lipschitz domain, for any $k\ge1$ there are operators $\Sc^{k,\alpha}_\Omega:\Ss'(\R^N)\to\Ss'(\R^N)$, $|\alpha|\le k$, such that
\begin{enumerate}[(i)]
    \item\label{Item::Space::ATDOp::Bdd} $\Sc^{k,\alpha}_\Omega:\Fs_{pq}^s(\R^N)\to\Fs_{pq}^{s+k}(\R^N)$ for all $0<p,q\le\infty$ and $s\in\R$ such that $(p,q)\notin\{\infty\}\times(0,\infty)$.
    \item\label{Item::Space::ATDOp::Der} $g=\sum_{|\alpha|\le k}D^\alpha\Sc^{k,\alpha}g$ for all $g\in\Ss'(\R^N)$.
    \item\label{Item::Space::ATDOp::Supp} If $g\in\Ss'(\R^N)$ satisfies $g|_\Omega=0$, then $\Sc^{k,\alpha}g|_\Omega=0$ for all $|\alpha|\le k$.
\end{enumerate}
In particular $\Sc^{k,\alpha}_\Omega:\widetilde\Fs_{pq}^s(\Omega^c)\to\widetilde\Fs_{pq}^{s+k}(\Omega^c)$.
\end{prop}See \cite[Proposition~1.7]{ShiYaoExt}, whose proof is in \cite[Section~6.3]{ShiYaoExt}. 

In fact the result is also true for $\Fs_{\infty q}^s$-cases. This can be done by replacing \cite[Theorem~1.5 (i)]{ShiYaoExt} with \cite[Proposition~17]{YaoExtensionMorrey} in its proof. We do not need this result in our paper.

The condition \ref{Item::Space::ATDOp::Supp} is non-trivial here. If one only wants conditions \ref{Item::Space::ATDOp::Bdd} and \ref{Item::Space::ATDOp::Der} to be satisfied, we can consider the decomposition $f=(I-\Delta)^m((I-\Delta)^{-m}f)=\sum_{|\gamma|\le 2m}c_\gamma D^\gamma ((I-\Delta)^{-m}f)$ on $\R^N$.

\begin{rem}\label{Rmk::Space::ATDOp}
In practice we consider the composition $\Sc^{k,\alpha}\circ[\dbar,\Ec]f$, where $\Ec$ is an extension operator of $\Omega$ such that $\supp\Ec f\Subset\Uc$ for some fixed open bounded (smooth) neighborhood $\Uc\Supset\Omega$. Clearly $[\dbar,\Ec]f$ is supported in $\Uc\backslash\overline\Omega$ thus $\supp( \Sc^{k,\alpha}_\Omega\circ[\dbar,\Ec]f)\subseteq\Omega^c$. In order to obtain a better support condition $$\supp(\Sc^{k,\alpha}\circ[\dbar,\Ec]f)\subseteq\overline\Uc\backslash\Omega,$$ we can apply the proposition to the domain $\Omega\cup(\mathcal V\backslash\overline\Uc)$, where $\mathcal V\Supset\Uc$ is a larger bounded smooth domain that makes $\mathcal V\backslash\overline\Uc$ a bounded Lipschitz domain, and then taking a smooth cutoff outside $\mathcal V$.

In other words there exist $\{\Sc^{k,\alpha}\}_{|\alpha|\le k}:\widetilde\Fs_{pq}^s(\overline\Uc\backslash\Omega)\to \widetilde\Fs_{pq}^{s+k}(\overline\Uc\backslash\Omega)$ such that $g=\sum_{|\alpha|\le k}D^\alpha \Sc^{k,\alpha}g$.
\end{rem}

In Theorem~\ref{MainThm} we claim that $\Hc_q$ is defined on the large space $\Ss'(\Omega;\wedge^{0,q})$. This can be implied by its definedness on all H\"older spaces with negative indices.
\begin{lem}\label{Lem::Space::SsUnion}
Let $\Omega\subset\R^N$ be a bounded Lipschitz domain. Then $\Ss'(\Omega)=\bigcup_{s<0}\Co^s(\Omega)$.
\end{lem}
In fact one can replace $\Co^s(\Omega)$ by $H^{s,p}(\Omega)$ or even $\Fs_{pq}^s(\Omega)$ for $1\le p,q\le \infty$. Indeed by Remark~\ref{Rmk::Space::TLRmk} we have $\Co^s(\Omega)=\Fs_{\infty\infty}^s(\Omega)$ %(which is $\Bs_{\infty\infty}^s(\Omega)$ in the common notation)
and the embedding $\Co^{s+1}(\Omega)\subset \Fs_{pq}^s(\Omega)\subset\Co^{s-N}(\Omega)$.
\begin{proof}
By Remark~\ref{Rmk::Space::TLRmk} \ref{Item::Space::TLRmk::SobHold} $\Co^s(\Omega)=\Fs_{\infty\infty}^s(\Omega)=\Fs_{\infty\infty}^s(\R^N)|_\Omega\subset\Ss'(\R^N)|_\Omega=\Ss'(\Omega)$ for all $s$, thus $\Ss'(\Omega)\supseteq\bigcup_{s<0}\Co^s(\Omega)$.

Conversely, for a $f\in\Ss'(\Omega)$, take an extension $\tilde f\in\Ss'(\R^N)$. We can assume that $\tilde f$ has compact support, which can be done by replacing $\tilde f$ with $\chi\tilde f$, where $\chi\in C_c^\infty(\R^N)$ satisfies $\chi|_{\Omega}\equiv1$. Now by the structure theorem of distributions (see \cite[Theorem~6.27]{GrandpaRudin} for example), there are $M\ge0$ and $\{g_\alpha\}_{|\alpha|\le M}\subset C_c^0(\R^N)$ such that $\tilde f=\sum_{|\alpha|\le M}D^\alpha g_\alpha$. Clearly $g_\alpha\in \Co^0(\R^N)$ thus $\tilde f\in\Co^{-M}(\R^N)$ and we see that $f\in\Co^{-M}(\Omega)\subset\bigcup_{s<0}\Co^s(\Omega)$. This completes the proof.
\end{proof}

\section{Tangential Commutator Estimate and Strong Hardy-Littlewood Lemma}\label{Section::LPThm}

To reduce Theorem~\ref{MainThm} to Theorem~\ref{Thm::WeiEst}, we need the following two results, the Propositions \ref{Prop::LPThm::TangComm} and \ref{Prop::LPThm::HLLem}.

We need to show that the tangential part of commutator $[\dbar,\Ec]^\top f=([\dbar,\Ec]f)^\top$ does not lose derivative.
\begin{prop}[Tangential projection of commutator]\label{Prop::LPThm::TangComm}
    Let $\omega=\{x_1>\sigma(x')\}\subset\R^N$ be a special Lipschitz domain. Let $E=E_\omega$ be the extension operator in Definition~\ref{Defn::Space::ExtOp}.
    
    Let $X=\sum_{\nu=1}^NX^\nu\Coorvec{x_\nu}$ be a smooth vector field on $\R^N$ such that $X(x)\in T_x(b\omega)$ for almost every $x\in b\omega$ (in the sense of surface measure).
    Then we have the following boundedness:
    \begin{equation*}
        \big\langle X,[d,E]\big\rangle=\sum_{\nu=1}^NX^\nu\cdot[D_\nu,E]:\Fs_{pp}^s(\omega)\to\widetilde\Fs_{pp}^s(\omega^c),\qquad\text{for every }1\le p\le \infty,\quad s\in\R.
    \end{equation*}
    
    In particular $\langle X,[d,E]\rangle:\Fs_{p\infty}^s(\omega)\to\widetilde\Fs_{p\eps}^{s-\delta}(\omega^c)$ for every $\eps,\delta>0$.
\end{prop}

\begin{rem}\label{Rmk::LPThm::TangCommRmk}
    \begin{enumerate}[(i)]
        \item There is a different kind of commutator estimate in \cite[Theorem~4.1]{ShiYaoC2}, where we prove a smoothing estimate $[D,E]:\Fs_{p\infty}^s(\omega)\to L^1(\R^N;\C^N)$ for $1\le p\le \infty$ and $s>1-\frac1p$.
        \item\label{Item::LPThm::TangCommRmk::FBdd} In fact $\langle X,[d,E]\rangle:\Fs_{p\infty}^s(\omega)\to\widetilde\Fs_{p\eps}^{s}(\omega^c)$ is bounded, i.e. we can take $\delta=0$. In application a $\delta$-loss is enough and the proof for $\Fs_{pp}^s$ spaces is simpler.
        \item Proposition~\ref{Prop::LPThm::TangComm} can be intuitively understood if we let $E$ be the standard half-space extension for $\omega=\R^N_+:=\{x_1>0\}$. Recall that for an integer $M\ge1$, the half-space extension is given by
        \begin{equation*}
            E^Mf(x_1,x'):=\begin{cases}
        \sum_{j=-M}^Ma_jf(-b_jx_1,x')&x_1<0
        \\
        f(x)&x_1>0
        \end{cases},\text{ where }b_j>0\text{ and }\sum_{j=-M}^Ma_j(-b_j)^k=1\text{ for all }|k|\le M.
        \end{equation*}
        We have boundedness $E^M:\Fs_{pq}^s(\R^N_+)\to\Fs_{pq}^s(\R^N)$ for all $1\le p,q\le\infty$ and $-M<s<M$, see \cite[Theorem~2.9.2]{TriebelTheoryOfFunctionSpacesI}. One can achieve the construction with $M=\infty$ and have the boundedness $E^\infty:\Fs_{pq}^s(\R^N_+)\to\Fs_{pq}^s(\R^N)$ as well. See the recent paper \cite{LuYaoSeeley} by Lu and the author. It is clear that $[D_\nu,E^\infty]\equiv0$ in the domain for all $2\le \nu\le N$, thus $\langle X,[d,E^\infty]\rangle\equiv0$ if $X^1\equiv0$.

        One can use $E^\infty$ with partition of unity (see \eqref{Eqn::Space::ExtOmega}) or the method in \cite[Theorem~26]{LuYaoSeeley} to define the extension operator $\Ec$ for the smooth domain $\Omega\subset\C^n$. It is possible that a modification of Corollary~\ref{Cor::LPThm::BddDom}~\ref{Item::LPThm::BddDom::TangComm} still holds.
        % it is easy to show that $[\dbar,\Ec]^\top:\Fs_{p\infty}^s\to\Fs_{p\infty}^s$ is bounded for all $1\le p\le\infty$ and  $-M<s<M$, see also Corollary~\ref{Cor::LPThm::BddDom} \ref{Item::LPThm::BddDom::TangComm}. However if one use $E^M$ to define our solution operator $\Hc_q$, then $\Hc_q$ is not bounded outside the range $-M<s<M$.
        %\item However it is not known to the author whether we are allow to take $M=\infty$, i.e. whether we have a $E^\infty$ such that $E^\infty:\Fs_{pq}^s(\R^N_+)\to\Fs_{pq}^s(\R^N)$ for all $s$. By \cite{Seeley} one is allowed to have the boundedness for $-M<s\le\infty$ where $M$ is large but fixed.
        % \item Moreover when $X$ is a non-smooth vector field such that $X^1|_{x_1=0}=0$ but $X^1\neq0$, the multiplier $[f\mapsto f\cdot X]$ is not defined if $f$ is too rough, leading the undefinedness of $\langle X,[d,E^M]\rangle$ on some distribution spaces. On the other hand, from the proof of Proposition~\ref{Prop::LPThm::TangComm} for Rychkov's extension $\langle X,[d,E]\rangle$ is defined for distributions. This is due to the smoothing effects of $E$ on $\overline{\omega}^c$.
    \end{enumerate}
\end{rem}

We also need the Hardy-Littlewood type lemma that gives embeddings between the fractional Sobolev spaces (or more generally the Triebel-Lizorkin spaces for Theorem~\ref{StrMainThm}) and the weighted Sobolev spaces.

\begin{prop}[Strong Hardy-Littlewood lemma]\label{Prop::LPThm::HLLem}
Let $\omega\subset\R^N$ be a special Lipschitz domain. Let $\delta(x):=\min(1,\dist(x,b\omega))$ for $x\in\R^N$. We have the following embeddings:
\begin{enumerate}[(i)]
    \item\label{Item::LPThm::HLLem::FtoW} $\widetilde\Fs_{p\infty}^s(\overline\omega)\hookrightarrow L^p(\omega,\delta^{-s})$ for all $1\le p\le\infty$ and $s>0$.
    \item\label{Item::LPThm::HLLem::WtoF} $W^{m,p}(\omega,\delta^{m-s})\hookrightarrow\Fs_{p\eps}^s(\omega)$ for all $\eps>0$, $1\le p\le\infty$, $m\in\N$ and $s<m$.
\end{enumerate}
\end{prop}

We are going to prove Proposition~\ref{Prop::LPThm::HLLem} first then Proposition~\ref{Prop::LPThm::TangComm}.

\begin{rem}The proof of Proposition~\ref{Prop::LPThm::HLLem} is standard if we replace $\Fs_{p\eps}^s$ and $\Fs_{p\infty}^s$ by classical Sobolev or H\"older spaces, see \cite{LigockaHLLem} for example. Our result is stronger, since from Remark~\ref{Rmk::Space::TLRmk} we have $\Fs_{p\eps}^s\subsetneq H^{s,p}=\Fs_{p2}^s\subsetneq\Fs_{p\infty}^s$ for $1<p<\infty$ and $\Fs_{\infty\eps}^s\subsetneq\Co^s=\Fs_{\infty\infty}^s$.

The result \ref{Item::LPThm::HLLem::FtoW} is not new. See \cite{HLGeneral} for a more general version. We also refer \cite[Chapter 5.8]{TriebelStructureFunctions} for a proof on smooth domains which contains the discussion of the case $p<1$.
\end{rem}

We first give the application of Propositions \ref{Prop::LPThm::TangComm} and \ref{Prop::LPThm::HLLem} in our setting.

\begin{cor}\label{Cor::LPThm::BddDom}
Let $\Omega=\{\varrho<0\}\subset\C^n$ be a bounded Lipschitz domain and let $\Uc\Supset\Omega$ be a bounded open neighborhood so that $\varrho$ has non-vanishing gradient in $\Uc\backslash\overline\Omega$. Let $\dist(z):=\dist(z,b\Omega)$ for $z\in\C^n$ and let $\Ec=\Ec_\Omega$ be given in Definition~\ref{Defn::Space::ExtOmega} whose images are supported in $\Uc$. The following linear maps are bounded.
\begin{enumerate}[(i)]
    \item\label{Item::LPThm::BddDom::FtoW} $\widetilde\Fs_{p\infty}^s(\overline{\Uc\backslash\Omega})\hookrightarrow L^p(\Uc\backslash\overline\Omega,\dist^{-s})$ for all $1\le p\le\infty$ and $s>0$.
    \\
    In particular $\widetilde H^{s,p}(\overline{\Uc\backslash\Omega})\hookrightarrow L^p(\Uc\backslash\overline\Omega,\dist^{-s})$, $\widetilde\Co^s (\overline{\Uc\backslash\Omega})\hookrightarrow L^\infty(\Uc\backslash\overline\Omega,\dist^{-s})$ for $1<p<\infty$, $s>0$.
    \item\label{Item::LPThm::BddDom::WtoF} $W^{k,p}(\Omega,\dist^{k-s})\hookrightarrow\Fs_{p\eps}^s(\Omega)$ for all $\eps>0$, $1\le p\le\infty$, $k\in\N$ and $s<k$.
    \\
    In particular $W^{k,p}(\Omega,\dist^{k-s})\hookrightarrow H^{s,p}(\Omega)$ and $W^{k,\infty}(\Omega,\dist^{k-s})\hookrightarrow \Co^s(\Omega)$ for $1<p<\infty$, $s<k$.
    \item\label{Item::LPThm::BddDom::TangComm}  If $\Omega$ is a smooth domain, then $[\dbar,\Ec]^\top:\Fs_{pp}^s(\Omega)\to\widetilde\Fs_{pp}^{s}(\overline{\Uc\backslash\Omega};\C^n)$ for all $1\le p\le\infty$, $s\in\R$.
    \\
    In particular $[\dbar,\Ec]^\top:H^{s,p}(\Omega)\to\widetilde H^{s-\delta,p}(\overline{\Uc\backslash\Omega};\C^n)$ and $[\dbar,\Ec]^\top:\Co^s(\Omega)\to\widetilde \Co^{s}(\overline{\Uc\backslash\Omega};\C^n)$ for $p\in(1,\infty)$ and $\delta>0$.
\end{enumerate}
\end{cor}
\noindent   Here $\widetilde H^{s,p}(\overline{U}):=\{f\in H^{s,p}(\R^N):f|_{\overline U^c}=0\}\subset H^{s,p}(\R^N)$ and $\widetilde \Co^s(\overline{U}):=\{f\in \Co^s(\R^N):f|_{\overline U^c}=0\}\subset \Co^s(\R^N)$ follow the notations in Definitions \ref{Defn::Space::Sob}, \ref{Defn::Space::Hold} and \ref{Defn::Space::TLSpace}.

Recall that $H^{s,p}=\Fs_{p2}^s\hookrightarrow\Fs_{pp}^{s-\delta/2}\hookrightarrow\Fs_{p2}^{s-\delta}=H^{s-\delta,p}$ and $\Co^s=\Fs_{\infty\infty}^s$ from Remark~\ref{Rmk::Space::TLRmk}. The boundededness for Sobolev and H\"older spaces hold immediately.
\begin{proof}Recall Notation~\ref{Note::Space::PartUnity} that we have $U_\nu\cap\Omega=U_\nu\cap\Phi_\nu(\omega_\nu)$ for $1\le \nu\le M$ and $f=\sum_{\nu=0}^M\chi_\nu^2f$. By Lemma~\ref{Lem::Space::PartUnityLem} $[f\mapsto\chi_\nu f]:\Fs_{pq}^s(\Omega)\to\Fs_{pq}^s(U_\nu\cap\Omega)$ are all bounded for $0\le \nu\le M$.

    For each $1\le \nu\le M$, by Proposition~\ref{Prop::LPThm::HLLem} we have $\widetilde\Fs_{p\infty}^s(\overline{U_\nu\cap\Omega})\hookrightarrow L^p(U_\nu\cap\Omega,\dist_{b\Omega}^{-s})$ for $s>0$, and $W^{m,p}(\Phi_\nu(\omega_\nu),\dist_{b\Omega}^{m-s})\xrightarrow{(-)|_{U_\nu}}\Fs_{p\eps}^s(U_\nu\cap\Omega)$ for $s<m$.
    
    For $\nu=0$ we have the trivial estimates $[f\mapsto\chi_0f]:\Fs_{p\infty}^s(\Omega)\to L^p_c(U_0)\hookrightarrow L^p(\Omega,\dist^{-s})$ for $s<m$, and $[f\mapsto\chi_0f]:W^{m,p}(\Omega,\dist^{m-s})\to W^{m,p}_c(U_0)\hookrightarrow \Fs_{p\eps}^s(\Omega)$ for $s<m$. 
    
    Therefore for every $1\le p\le\infty$, $m\in\N$ and $\eps>0$,
    \begin{gather*}
    \begin{aligned}
        \widetilde\Fs_{p\infty}^s(\overline\Omega)&\xrightarrow{f\mapsto(\chi_\nu f)_{\nu=0}^M}\bigoplus_{\nu=0}^M\widetilde\Fs_{p\infty}^s(\overline{U_\nu\cap\Omega})\hookrightarrow \bigoplus_{\nu=0}^ML^p(U_\nu\cap\Omega,\dist_{b\Omega}^{-s})\\
        &\xrightarrow{(g_\nu)_{\nu=0}^M\mapsto\sum_{\nu=0}^M\chi_\nu g_\nu}L^p(\Omega,\dist^{-s})
    \end{aligned}
        ,\quad s>0;
        \\
    \begin{aligned}
        W^{m,p}(\Omega,\dist^{m-s})&\xrightarrow{f\mapsto(\chi_\nu f)_\nu}\bigoplus_{\nu=0}^MW^{m,p}(U_\nu\cap\Omega,\dist_{b\Omega}^{m-s})
        \\
        &\hookrightarrow\bigoplus_{\nu=0}^M \Fs_{p\eps}^s(U_\nu\cap\Omega)\xrightarrow{(g_\nu)_\nu\mapsto\sum_{\nu}\chi_\nu g_\nu}\Fs_{p\eps}^s(\Omega)
    \end{aligned}
        ,\quad s<m.
    \end{gather*}

The second composition map gives $W^{m,p}(\Omega,\dist^{m-s})\hookrightarrow\Fs_{p\eps}^s(\Omega)$, finishing the proof of \ref{Item::LPThm::BddDom::WtoF}.

By suitably shrinking $\Uc$ we can assume $\Uc$ bounded Lipschitz, hence $\Uc\backslash\overline\Omega$ is also bounded Lipschitz. The first composition map then gives $\widetilde\Fs_{p\infty}^s(\overline{\Uc\backslash\Omega})\hookrightarrow L^p(\Uc\backslash\overline\Omega,\dist^{-s}_{b\Omega\cup b\Uc})\hookrightarrow L^p(\Uc\backslash\overline\Omega,\dist^{-s}_{b\Omega})$ for $s>0$, finishing the proof of \ref{Item::LPThm::BddDom::FtoW}.

\medskip
To prove \ref{Item::LPThm::BddDom::TangComm}, for convenience we write $E_\nu g:=(E_{\omega_\nu}[g\circ\Phi_\nu])\circ\Phi_\nu^{-1}$ where $g$ is defined on $\Phi_\nu(\omega_\nu)$ (recall from \ref{Item::PartitionUnity::Phi} that $U_\nu\cap\Omega=U_\nu\cap\Phi_\nu(\omega_\nu)$), thus $\Ec=\chi_0^2+\sum_{\nu=1}^M\chi_\nu\circ E_\nu\circ\chi_\nu$ and we have
\begin{equation}\label{Eqn::LPThm::BddDom::TangCommPf::Tmp}
    [\dbar,\Ec]=2(\chi_0\dbar\chi_0)+\sum_{\nu=1}^M\big((\dbar\chi_\nu)\circ E_\nu\circ\chi_\nu+\chi_\nu\circ E_\nu\circ (\dbar\chi_\nu)+\chi_\nu\circ[\dbar,E_\nu]\circ\chi_\nu\big).
\end{equation}
Here the function $\chi_\nu$ stands for the linear map (pointwise multiplier) $[g\mapsto \chi_\nu g]$.

Since $\chi_\nu$ are smooth, by Proposition~\ref{Prop::Space::RychFact} \ref{Item::Space::RychFact::ExtBdd} and Lemma~\ref{Lem::Space::PartUnityLem}, all terms in \eqref{Eqn::LPThm::BddDom::TangCommPf::Tmp} except $\chi_\nu\circ[\dbar,E_\nu]\circ\chi_\nu$ have the boundedness $\Fs_{pp}^s(\Omega)\to\Fs_{pp}^s(\C^n)$ for all $1\le p\le\infty$ and $s\in\R$. It suffices to show that for $1\le \nu\le M$ we have $\chi_\nu\circ[\dbar,E_\nu]\circ\chi_\nu:\Fs_{pp}^s(\Omega)\to\Fs_{pp}^s(\C^n)$.

Recall from Remark~\ref{Rmk::Goal::TopBotFacts}, by partition of unity we can find $M'\ge n-1$, smooth $(0,1)$-vector fields $\overline W_1,\dots,\overline W_{M'}$ and\\ $(0,1)$-forms $\overline\eta_1,\dots,\overline\eta_{M'}$ on $\Uc\backslash\overline\Omega$ such that $\overline W_\mu(\zeta)\in T^{0,1}_\zeta(b\Omega_{\varrho(\zeta)})$ for all $1\le \mu\le M'$ and $\zeta\in\Uc\backslash\overline\Omega$, and $\alpha^\top=\sum_{\mu=1}^{M'}\langle \overline{W}_\mu,\alpha\rangle\cdot\overline\eta_\mu$ holds for all $(0,1)$-form $\alpha$ on $\Uc\backslash\overline\Omega$.

Since $T^{0,1}(b\Omega)\subset\C T(b\Omega)$ and $T_{\zeta}(\Phi_\nu(b\omega_\nu))=T_{\zeta}(b\Omega)$ for $\zeta\in U_\nu\cap b\Omega$, we have $(\chi_\nu W_\mu)(\zeta)\in\C T_{\zeta}(\Phi_\nu(b\omega_\nu))$ for $\zeta\in U_\nu\cap b\Omega$. By Proposition~\ref{Prop::LPThm::TangComm}, $\langle \overline W_\mu,\chi_\nu\circ[d,E_\nu]\rangle=\langle  \chi_\nu \cdot \overline W_\mu,[\dbar,E_\nu]\rangle:\Fs_{pp}^s(\Phi_\nu(\omega_\nu))\to \widetilde\Fs_{pp}^s(\Phi_\nu(\omega_\nu)^c;\C^n)$ is bounded. Therefore 
\begin{align*}
    \langle \overline W_\mu,\chi_\nu\circ[\dbar,E_\nu]\circ\chi_\nu\rangle:&\Fs_{pp}^s(\Omega)\xrightarrow{\chi_\nu}\Fs_{pp}^{s}(\Phi_\nu(\omega_\nu))\xrightarrow{\langle  \chi_\nu \cdot \overline W_\mu,[\dbar,E_\nu]\rangle}\widetilde\Fs_{pp}^{s}(\Phi_\nu(\omega_\nu)^c;\C^n)\xrightarrow{(-)|_{U_\nu}}\Fs_{pp}^{s}(\Omega^c;\C^n).
\end{align*}

We get $\langle \overline W_\mu,\chi_\nu\circ[\dbar,E_\nu]\circ\chi_\nu\rangle:\Fs_{pp}^s(\Omega)\to\widetilde\Fs_{pp}^{s}(\overline{\Uc\backslash\Omega};\C^n)$ since $\supp\chi_\nu\Subset\Uc$, which gives \ref{Item::LPThm::BddDom::TangComm}.
\end{proof}

\begin{rem}
    Using the notations from Remark \ref{Rmk::Goal::TopBotFacts}, for a $(0,q)$-form $f=\sum_Jf_Jd\bar z^J$, we have
    \begin{align*}
        [\dbar,\Ec]^\top f=&\sum_{|K'|=q+1,|J|=q}\langle\overline Z_{K'},[\dbar,\Ec](f_Jd\bar z^J)\rangle
        \\
        =&\sum_{j=1}^n\sum_{k=2}^n\sum_{|J|=|K|=q;\min K\ge2}([\dbar_j,\Ec]f_J)\langle\overline Z_k,d\bar z^j\rangle\langle \overline Z_K,d\bar z^J\rangle\overline\theta^k\wedge\overline\theta^K
        \\
        =&\sum_{k=2}^n\sum_{|K|=q;\min K\ge2}\big(\langle\overline Z_k,[\dbar,\Ec]f_J\rangle\big)\cdot\langle \overline Z_K,d\bar z^J\rangle\overline\theta^k\wedge\overline\theta^K.
    \end{align*}
    Therefore to estimate $[\dbar,\Ec]^\top f$ it suffices to estimate its components $[\dbar,\Ec]^\top f_J$.
\end{rem}

We first prove Proposition~\ref{Prop::LPThm::HLLem}, we only prove the case $\eps=1$ of \ref{Item::LPThm::HLLem::WtoF} here and leave the proof of $\eps<1$ to the appendix.
\begin{proof}[Proof of Proposition~\ref{Prop::LPThm::HLLem}]
Write $\omega=\{x_1>\sigma(x')\}$. We define the outer strips 
\begin{equation}\label{Eqn::Space::HLLem::Sk}
    S_k=S_k^\omega:=\{(x_1,x'):-2^{\frac12-k}<x_1-\sigma(x')<-2^{-\frac12-k}\}\subset\omega^c\text{ for }k\in\Z_+,
\end{equation} and $S_0:=\{(x_1,x'):x_1-\sigma(x')<-2^{-\frac12}\}$. Recall $\delta(x)=\min(1,\dist(x,b\omega))$ in the assumption. Recall that a special Lipschitz domain satisfies $\|\nabla\sigma\|_{L^\infty}<1$, thus (see also \cite[(5.3)]{ShiYaoExt})
\begin{equation}\label{Eqn::Space::HLLem::Tmp1}
    2^{-1-k}\le\delta(x)\le 2^{\frac12-k},\quad\forall k\ge0,\quad x\in S_k.
\end{equation}
Since $\{S_k\}_{k=0}^\infty$ is a partition of $\omega^c$ up to zero measured sets, we have $\|g\|_{L^p(\omega)}=\|(\|g\|_{L^p(S_k)})_{k=0}^\infty\|_{\ell^p(\N)}$.

By the assumption \ref{Item::Space::PhiSupp} we see that $\phi_0,\phi_1$ are supported in $-\Kb\cap\{x_1<-c_1\}$ for some $c_1>0$. By assumption \ref{Item::Space::PhiScal} we get $\supp\phi_j\subset-\Kb\cap\{x_1<-c_12^{-j}\}$ for all $j\ge0$. A simple calculation shows that (see \cite[Lemma~5.3]{ShiYaoExt})
\begin{equation}\label{Eqn::Space::HLLem::Tmp2}
    \exists R\in\Z_+\text{ such that }\supp\phi_j+\omega^c\subseteq\{x_1-\sigma(x')<-2^{-j-R}\}.
\end{equation}

Let $f\in \widetilde\Fs_{p\infty}^s(\omega^c)$, we have $\supp f\subseteq\omega^c$. By \ref{Item::Space::PhiApprox} we have $f=\sum_{j=0}^\infty\phi_j\ast f$. Using \eqref{Eqn::Space::HLLem::Tmp2} we see that  $f(x)=\sum_{j=k-R}^\infty\phi_j\ast f(x)$ for $k\ge0$ and $x\in S_k$. Thus,
\begin{align*}
    &\|f\|_{L^p(\overline\omega^c,\delta^{-s})}=\big\|\big(\|f\|_{L^p(S_k,\delta^{-s})}\big)_{k=0}^\infty\big\|_{\ell^p(\N)}\le \big\|\big(2^{(k+1)s}\|f\|_{L^p(S_k)}\big)_{k=0}^\infty\big\|_{\ell^p(\N)}
    \\
    \le&\bigg\|\Big(\Big\|2^{(k+1)s}\sum_{j=k-R}^\infty|\phi_j\ast f|\Big\|_{L^p(S_k)}\Big)_{k=0}^\infty\bigg\|_{\ell^p(\N)}=\bigg\|\Big(\Big\|\sum_{j=k-R}^\infty 2^{(k-j+1)s}|2^{js}\phi_j\ast f|\Big\|_{L^p(S_k)}\Big)_{k=0}^\infty\bigg\|_{\ell^p(\N)}
    \\
    \le&\sum_{l=-R}^\infty2^{(1-l)s}\Big\|\Big(\Big\|\sup_{j\ge k-R}|2^{js}\phi_j\ast f|\Big\|_{L^p(S_k)}\Big)_{k=0}^\infty\Big\|_{\ell^p(\N)}
    \lesssim_{R,s}\Big\|\sup_{j\in\N}|2^{js}\phi_j\ast f|\Big\|_{L^p(\omega^c)}=\|f\|_{\Fs_{p\infty}^s(\phi)}.
\end{align*}
Here we use convention $\phi_j\ast f=0$ for $j\le -1$. By \cite[Proposition~1.2 (i)]{RychkovExtension} $\|f\|_{\Fs_{p\infty}^s(\phi)}$ is an equivalent norm for $\Fs_{p\infty}^s(\R^N)$. This completes the proof of  \ref{Item::LPThm::HLLem::FtoW}.

\medskip
For \ref{Item::LPThm::HLLem::WtoF} we prove case $\eps=1$ here by duality argument. We leave the proof of $0<\eps<1$ to the appendix, which is a direct proof without using duality.

Let $\mathring{\Fs}_{pq}^s(\R^N)$ be the norm closure of $C_c^\infty(\R^N)$ in $\Fs_{pq}^s(\R^N)$ and let $\mathring{\Fs}_{pq}^s(\overline{\omega}):=\{f\in\mathring\Fs_{pq}^s(\R^N):f|_{\overline{\omega}^c}=0 \}$ be its subspace. Clearly $\mathring{\Fs}_{pq}^s(\overline{\omega})\subseteq\widetilde{\Fs}_{pq}^s(\overline{\omega})$. We see that $\mathring{\Fs}_{pq}^s(\overline{\omega})=\overline{C_c^\infty(\omega)}^{\Fs_{pq}^s(\R^N)}$, whose proof follows the same argument to \cite[Theorem~4.3.2/1 Proof Step 2]{TriebelInterpolation} via partition of unity and translations. % Clearly $\mathring{\Fs}_{pq}^s(\overline{\omega})\subseteq\widetilde {\Fs}_{pq}^s(\overline{\omega})$ is a closed space. 

% When $s>0$ and $1\le p,q\le\infty$, the map $[f\mapsto f|_\omega]:\widetilde {\Fs}_{pq}^s(\overline{\omega})\to\Fs_{pq}^s(\omega)$ is embedding because the kernel of the map equals to $\{f\in\Fs_{pq}^s(\R^N):\supp f\subseteq b\omega\}\subset\{f\in L^1_\loc(\R^N):\supp f\subseteq b\omega\}=\{0\}$. Therefore

% Let $\mathring{\Fs}_{pq}^s(\omega):=\{f\in\mathring\Fs_{pq}^s(\R^N):f|_{\overline{\omega}^c}=0 \}$ For $s>0$

% For $s>0$, since $\Fs_{p'\infty}^s(\R^N)\subset L^1_\loc(\R^N)$, we see that for every $\phi\in\mathring\Fs_{p'\infty}^s(\R^N)$, $\phi|_{\overline\omega^c}=0$ if and only if  
On the other hand by \cite[Remark~1.5]{TriebelTheoryOfFunctionSpacesIV} we have $\Fs_{p1}^{-s}(\R^N)=\mathring \Fs_{p'\infty}^{s}(\R^N)'$ for all $s\in\R$ and $1\le p\le\infty$, where $p'=\frac p{p-1}$. Therefore (see also \cite[Theorem~2.10.5/1]{TriebelInterpolation}) for $s\in\R$ and $1\le p\le\infty$,
\begin{align*}
    \mathring{\Fs}_{p'\infty}^{s}(\overline{\omega})'
    =&\Fs_{p1}^{-s}(\R^N)/\{f:\langle f,\phi\rangle=0, \forall \phi\in \mathring\Fs_{p'\infty}^s(\R^N),\ \phi|_{\overline\omega^c}=0\}
    \\
    =&\Fs_{p1}^{-s}(\R^N)/\{f:\langle f,\phi\rangle=0, \forall \phi\in C_c^\infty(\omega)\}=\Fs_{p1}^{-s}(\R^N)/\{f:f|_\omega=0\}=\Fs_{p1}^{-s}(\omega).
\end{align*}

By result \ref{Item::LPThm::HLLem::FtoW} we have $\Fs_{p'\infty}^t(\overline{\omega})\hookrightarrow L^{p'}(\omega,\delta^{-t})$ for $t>0$. In particular $\mathring\Fs_{p'\infty}^t(\overline{\omega})\hookrightarrow \overline{C_c^\infty(\omega)}^{L^{p'}(\omega,\delta^{-t})}$.

Clearly $\overline{C_c^\infty(\omega)}^{L^{p'}(\omega,\delta^{-t})}=L^{p'}(\omega,\delta^{-t})$ if $1<p\le\infty$, taking the adjoint we get $L^p(\omega,\delta^t)\hookrightarrow\Fs_{p1}^{-t}(\omega)$ for $1<p\le\infty$, which is \ref{Item::LPThm::HLLem::WtoF} at $m=0$ and $s=-t<0$.

For $p=1$ we have $\overline{C_c^\infty(\omega)}^{L^{p'}(\omega,\delta^{-t})}=\{f\in C^0(\overline\omega):\lim_{x\to b\omega}\dist_{b\omega}(x)^{-t}f(x)=0\text{ uniformly}\}$. Thus the adjoint gives the embedding $\{f\in\Ms_\loc(\omega):\|\delta^{t}f\|_{\Ms}<\infty\}\hookrightarrow\Fs_{11}^{-t}(\omega)$ from the space of locally finite Borel measures\footnote{Here we use $\Ms_\loc(\Omega)$ for the space of locally finite signed Borel measures, where the norm $\|\cdot\|_{\Ms}$ is the total variation of a measure.}. Since $L^1(\omega,\delta^t)\subset \{f\in\Ms_\loc(\omega):\|\delta^{t}f\|_{\Ms}<\infty\}$ is a closed subspace, we get $L^1(\omega,\delta^t)\hookrightarrow\Fs_{11}^{-t}(\omega)$, which is \ref{Item::LPThm::HLLem::WtoF} for $m=0$, $p=1$ and $s=-t<0$.

For $m\ge1$, recall that $\|f\|_{W^{m,p}(\omega,\delta^{s-m})}\approx\sum_{|\alpha|\le m}\|D^\alpha f\|_{L^p(\omega,\delta^{s-m})}$. Therefore for every $m\ge0$ and $s<m$ we have $\sum_{|\alpha|\le m}\|D^\alpha f\|_{\Fs_{p1}^{s-m}(\omega)}\lesssim\|f\|_{L^p(\omega,\delta^{s-m})}$. On the other hand by Proposition~\ref{Prop::Space::RychFact} \ref{Item::Space::RychFact::EqvNorm2} we have $\|f\|_{\Fs_{p1}^s(\omega)}\approx\sum_{|\alpha|\le m}\|D^\alpha f\|_{\Fs_{p1}^{s-m}(\omega)}$. Combining them we get \ref{Item::LPThm::HLLem::WtoF} for $\eps=1$ and finish the proof.
\end{proof}

To prove Proposition~\ref{Prop::LPThm::TangComm} we use a version of Heideman-type estimate \cite{Heideman}.

\begin{lem}\label{Lem::LPThm::HeiEst}
    Let $\phi=(\phi_j)_{j=0}^\infty$ and $\psi=(\psi_j)_{j=0}^\infty$ be as in Definition~\ref{Defn::Space::ExtOp}. Then for any $M>0$, $\alpha,\beta\in\N$ and $g\in\Co^\infty(\R^N)$, there is a $C=C_{\phi,\psi,M,\alpha,\beta,g}>0$ such that for every $j,k\in\N$, $1\le p\le\infty$ and $f\in L^p(\R^N)$,
    \begin{gather}\label{Eqn::LPThm::HeiEst}
        %\|\phi_j\ast D^\alpha\psi_k\ast (gf))\|_{L^p(\R^N)}\le C2^{j|\alpha|+k|\beta|-M|j-k|}\|f\|_{L^p(\R^N)};
        %\\
        \|\phi_j\ast(g\cdot(D^\alpha\psi_k\ast f))-\phi_j\ast D^\alpha\psi_k\ast(g f)\|_{L^p(\R^N)}\le C2^{k|\alpha|-M|j-k|-k}\|f\|_{L^p(\R^N)}.
    \end{gather}
\end{lem}
We can write the left hand side of \eqref{Eqn::LPThm::HeiEst} as $\big\|\phi_j\ast\big([g\cdot(-),D^\alpha\psi_k\ast(-)]\{f\}\big)\big\|_{L^p}$ in terms of commutator.
\begin{proof}The direct computation yields
    \begin{align*}
        &\phi_j\ast(g\cdot(D^\alpha\psi_k\ast f))(x)-\phi_j\ast D^\alpha\psi_k\ast(g f)(x)
        \\
        =&\phi_j\ast\Big[t\mapsto\int f(y)(g(t)-g(y))D^\alpha\psi_k(t-y)dy\Big](x)
        \\
        =&\int_{\R^N}f(y)dy\int_{\R^N}(g(t)-g(y))D^\alpha\psi_k(t-y)\phi_j(x-t)dt=:\int_{\R^N}K_{jk}(x,y)f(y)dy.
    \end{align*}
    Here $K_{jk}(x,y)=\int_{\R^N}(g(t)-g(y))D^\alpha\psi_k(t-y)\phi_j(x-t)dt$.

    By Schur's test Lemma~\ref{Lem::PfThm::Schur} with $\gamma=1$ and $(X,\mu)=(Y,\nu)=(\R^N,dx)$, what we need is to prove that$$\sup_x\int |K_{jk}(x,y)|dy+\sup_y\int |K_{jk}(x,y)|dx\lesssim_{\phi,\psi,\alpha,\beta,g,M}2^{k|\alpha|-M|j-k|-k}.$$
    
    Let $M'\ge0$ to be chosen later, by Taylor's expansion we can write
    \begin{align*}
        K_{jk}(x,y)=&\int\Big(\sum_{0<|\gamma|\le M'}\frac{s^\gamma}{\gamma!}D^\gamma g(y)+R_{M'}(y,s)\Big)D^\alpha\psi_k(s)\phi_j(x-y-s)ds,
    \end{align*}
    where $R_{M'}(y,s):=g(y+s)-\sum_{0\le|\gamma|\le M'}\frac{s^\gamma}{\gamma!}D^\gamma g(y)$ is the Tayler's remainder in $s$-variable. Therefore
    \begin{equation}\label{Eqn::LPThm::HeiEst::PfTmp1}
        |K_{jk}(x,y)|\lesssim_{M',g}\sum_{0<|\gamma|\le M'}\Big|\big(\big[s\mapsto s^\gamma D^\alpha\psi_k(s)\big]\ast \phi_j\big)(x-y)\Big|+\Big|\left(\big(R_{M'}(y,\cdot)D^\alpha\psi_k\big)\ast \phi_j\right)(x-y)\Big|.
    \end{equation}
    
    On the other hand, Note that for every $\gamma\in\N^N$, we have scaling $\phi_k(x)=2^{(k-1)N}\phi_1(2^{k-1}x)$ for $k\ge1$ and
    \begin{equation}\label{Eqn::LPThm::HeiEst::TmpScal}
        x^\gamma D^\alpha\psi_k(x)=2^{(k-1)(N+|\alpha|-|\gamma|)}(2^{k-1}x)^\gamma D^\alpha\psi_1(2^{k-1}x),\quad k\ge1.
    \end{equation}
    Since both $\phi_k$ and $x^\gamma D^\alpha\psi_k$ have infinite moment vanishing for $k\ge1$, by \cite[Lemma~2.1]{Bui} again (see also \cite[Lemma~4.4]{ShiYaoC2} with $l\to+\infty$), we have, for every $M>0$ and $|\gamma|>0$,
    \begin{gather}\label{Eqn::LPThm::HeiEst::PfTmp2}
        \|\phi_j\ast(s^\gamma D^\alpha\psi_k)\|_{L^1} \lesssim_{\phi,\psi,\alpha,\beta,\gamma,M}2^{k(|\alpha|-|\gamma|)-M|j-k|}\le 2^{k(|\alpha|-1)-M|j-k|},\quad\text{for all }j,k\ge0.
    \end{gather}
    
    Moreover by \cite[Lemma~2.1]{Bui} (see also \cite[Proposition~3.5]{ShiYaoExt}), we have, for every $M>0$,
    \begin{gather}\label{Eqn::LPThm::HeiEst::PfTmp3}
        \|\phi_j\ast h\|_{L^1}\lesssim_{\phi,\alpha,M}2^{-jM}\sup_{|\gamma|\le 2M+N;x\in\R^N}(1+|x|^{2M+N})|D^\gamma h(x)|,\quad\text{for all }j\ge0,\quad h\in\Ss(\R^N).
    \end{gather}

    Taking $h=h_{M',y}=R_{M'}(y,\cdot)D^\alpha\psi_k$ in \eqref{Eqn::LPThm::HeiEst::PfTmp3}, applying Taylor's theorem to $R_{M'}(y,\cdot)$ we get 
    \begin{equation}\label{Eqn::LPThm::HeiEst::PfTmpTmp}
        \begin{aligned}
        &(1+|s|^{2M+N})\sum_{|\gamma|\le2M+N}|D^\gamma h_{M',y}(s)|\lesssim(1+|s|^{2M+N})\sum_{|\beta|,|\gamma|\le2M+N}|D^\beta_s R_{M'}(y,s)||D^{\alpha+\gamma}\psi_k(s)|
        \\
        &\qquad\lesssim(1+|s|^{2M+N})|s|^{M'-2M-N}\sum_{|\gamma|\le2M+N}|D^{\alpha+\gamma}\psi_k(s)|,
    \end{aligned}
    \end{equation}
    uniformly in $s,y\in\R^N$, whenever $M'>2M+N$. 
    
    For $k\ge1$, using the scaling property \eqref{Eqn::LPThm::HeiEst::TmpScal} and the fact that $\psi_1$ is Schwartz, we have, uniformly for every $s$ and $y$,
    \begin{equation}\label{Eqn::LPThm::HeiEst::PfTmpTmp2}
        \begin{aligned}
        &\textstyle(1+|s|^{2M+N})|s|^{M'-2M-N}\sum_{|\gamma|\le2M+N}|D^{\alpha+\gamma}\psi_k(s)|
        \\
        \lesssim&\textstyle(1+|s|^{2M+N})|s|^{M'-2M-N}\sum_{|\gamma|\le2M+N}2^{kN+k|\alpha+\gamma|}(1+|2^ks|)^{-M'}
        \\
        \lesssim&\textstyle\frac{1+|s|^{2M+N}}{1+|2^ks|^{2M+N}}\cdot\frac{|s|^{M'-2M-N}}{1+|2^ks|^{M'-2M-N}}2^{k(2M+2N)+k|\alpha|}\lesssim2^{-k(M'-2M-N)}2^{k(2M+2N+|\alpha|)}.
    \end{aligned}
    \end{equation}
    To summarize, by \eqref{Eqn::LPThm::HeiEst::PfTmp3}, \eqref{Eqn::LPThm::HeiEst::PfTmpTmp} and \eqref{Eqn::LPThm::HeiEst::PfTmpTmp2},
    \begin{equation}\label{Eqn::LPThm::HeiEst::PfTmp4}
        \sup_{y\in\R^N}\big\|\big(R_{M'}(y,\cdot)D^\alpha\psi_k\big)\ast \phi_j\big\|_{L^1(\R^N)}\lesssim2^{-jM}2^{-kM'}2^{k(4M+3N+|\alpha|)}.
    \end{equation}

    Therefore by \eqref{Eqn::LPThm::HeiEst::PfTmp2} and by taking $M'= 4M+3N+1$ in \eqref{Eqn::LPThm::HeiEst::PfTmp4} we have 
    \begin{align*}
        &\sup_{x\in\R^N}\int_{\R^N}\big( |K_{jk}(x,t)|+|K_{jk}(t,x)|\big)dt
        \\
        &\quad\lesssim\sum_{0<|\gamma|\le M'}\|\phi_j\ast(s^\gamma D^\alpha\psi_k) \|_{L^1(\R^N)}+\sup_y\|\phi_j\ast (R_{M'}(y,\cdot)D^\alpha\psi_k)\|_{L^1(\R^N)}
        \\
        &\quad\lesssim_M2^{k(|\alpha|-1)-M|j-k|}+2^{-jM+k(4M+3N+|\alpha|-M')}
        \\
        &\quad\lesssim_M2^{k(|\alpha|-1)-M|j-k|}+2^{k(|\alpha|-1)-M(j+k)}\lesssim 2^{k(|\alpha|-1)-M|j-k|}.
    \end{align*}
    This completes the proof.
\end{proof}

\begin{proof}[Proof of Proposition~\ref{Prop::LPThm::TangComm}]First we claim that $X\1_\omega=\sum_{\nu=1}^NX^\nu D_\nu\1_\omega=0$ holds in the sense of distributions. We use approximation. The assumption $X(x)\in \C T_x(b\omega)$ for almost every $x\in b\omega$ gives
\begin{equation*}
    \textstyle X^1(\sigma(x'),x')=\sum_{\nu=2}^NX^\nu(\sigma(x'),x')\cdot\frac{\partial\sigma}{\partial x_\nu}(x'),\quad\text{for almost every } x'\in\R^{N-1}.
\end{equation*}

For $\delta>0$ we define 
\begin{equation*}
    h_\delta(x):=0\text{ when }x_1\le\sigma(x')-\tfrac\delta2;\quad h_\delta(x):=\tfrac{x_1-\sigma(x')}\delta\text{ when }|x_1-\sigma(x')|\le\tfrac\delta2;\quad h_\delta(x):=1\text{ when }x_1\ge\sigma(x')+\tfrac\delta2.
\end{equation*}
Clearly $h_\delta\to \1_\omega$ as $\delta\to0$ in $L^p_\loc(\R^N)$ for all $1<p<\infty$. 

Clearly $Xh_\delta(x)=0$ for $|x_1-\sigma(x')|>\frac\delta2$. For $|x_1-\sigma(x')|<\frac\delta2$,
\begin{align*}
    Xh_\delta(x)=&\textstyle\sum_{\nu=1}^NX^\nu(x)D_\nu h_\delta(x)=\delta^{-1}X^1(x)-\delta^{-1}\sum_{\nu=2}^NX^\nu(x)D_\nu\sigma(x')
    \\
    =&\textstyle\frac{X^1(x)-X^1(\sigma(x'),x')}\delta-\sum_{\nu=2}^N\frac{X^\nu(x)-X^\nu(\sigma(x'),x')}\delta D_\nu\sigma(x')
    \\
    =&\textstyle(D_1X^1)(\sigma(x'),x')-\sum_{\nu=2}^N(D_1X^\nu)(\sigma(x'),x')D_\nu\sigma(x')+O(\delta).
\end{align*}

We conclude that $Xh_\delta\in L^\infty(\R^N)$ is uniformly bounded in $\delta$, and $Xh_\delta=0$ outside a $\delta$-neighborhood of $b\Omega $. Therefore $Xh_\delta\xrightarrow{L^p_\loc}0$ for all $p<\infty$, hence $X\1_\omega=\lim_{\delta\to0}Xh_\delta=0$ as distributions. 

\medskip Now we can rewrite $\langle X,[d,E]\rangle$ as 
\begin{align*}
    &\langle X,[d,E]\rangle f=\sum_{\nu=1}^N\sum_{k=0}^\infty  X^\nu\cdot(\psi_k\ast((D_\nu\1_\omega)\cdot(\phi_k\ast f)))
    \\
    =&\sum_{\nu=1}^N\sum_{k=0}^\infty X^\nu\cdot(\psi_k\ast((D_\nu\1_\omega)\cdot(\phi_k\ast f)))-\psi_k\ast((X^\nu D_\nu\1_\omega)\cdot(\phi_k\ast f))
    \\
    =&\sum_{\nu=1}^N\sum_{k=0}^\infty\Big( X^\nu\cdot(D_\nu\psi_k\ast(\1_\omega\cdot(\phi_k\ast f))-D_\nu\psi_k\ast((X^\nu\1_\omega)\cdot(\phi_k\ast f))
    \\
    &
    \quad\qquad-\big(X^\nu\cdot(\psi_k\ast(\1_\omega\cdot(D_\nu\phi_k\ast f)))-\psi_k\ast((X^\nu\1_\omega)\cdot(D_\nu\phi_k\ast f))\big)-\psi_k\ast((D_\nu X^\nu)\cdot\1_\omega\cdot(\phi_k\ast f))\Big)
    \\
    =&\sum_{\nu=1}^N\sum_{k=0}^\infty\Big([X^\nu,D_\nu\psi_k\ast(-)]\big\{\1_\omega(\phi_k\ast f)\big\}-[X^\nu,\psi_k\ast(-)]\big\{\1_\omega(D_\nu\phi_k\ast f)\big\}+\psi_k\ast((D_\nu X^\nu)\1_\omega(\phi_k\ast f))\Big).
\end{align*}

Note that  $(\phi_j)_{j=0}^\infty$ satisfies conditions \ref{Item::Space::PhiScal}, \ref{Item::Space::PhiMomt} and \ref{Item::Space::PhiApprox}. By \cite[Proposition~1.2]{RychkovExtension} we have 
\begin{equation}\label{Eqn::LPThm::TangCommPf::Tmp1}
    \|\langle X,[d,E]\rangle f\|_{\Fs_{pp}^s(\R^N)}\approx\big\|\big(2^{js}\|\phi_j\ast(\langle X,[d,E]\rangle f)\|_{L^p(\R^N)}\big)_{j=0}^\infty\big\|_{\ell^p(\N)}.
\end{equation}
On the other hand, applying Lemma~\ref{Lem::LPThm::HeiEst} we get
\begin{equation}\label{Eqn::LPThm::TangCommPf::Tmp2}
    \big\|\phi_j\ast\big([X^\nu\cdot(-),D^\alpha\psi_k\ast (-)]\big\{\1_\omega(D^\beta\phi_k\ast f)\big\}\big)\big\|_{L^p(\R^N)}\lesssim_M 2^{j|\alpha|-M|j-k|-k}\|D^\beta\phi_k\ast f\|_{L^p(\omega)}.
\end{equation}
And by \cite[Lemma~2.1]{Bui} (see also \cite[Corollary~3.6]{ShiYaoExt}), we have
\begin{equation}\label{Eqn::LPThm::TangCommPf::Tmp3}
    \|\phi_j\ast\psi_k((D_\nu X^\nu)\1_\omega(\phi_k\ast f))\|_{L^p(\R^N)}\le\|\phi_j\ast\psi_k\|_{L^1}\|D_\nu X^\nu\|_{L^\infty}\|\phi_k\ast f\|_{L^p(\omega)}\lesssim_M 2^{-M|j-k|}\|\phi_k\ast f\|_{L^p(\omega)}.
\end{equation}

Plugging \eqref{Eqn::LPThm::TangCommPf::Tmp2} and \eqref{Eqn::LPThm::TangCommPf::Tmp3} to \eqref{Eqn::LPThm::TangCommPf::Tmp1} we get, for every $M>0$,
\begin{align*}
    &\|\langle X,[d,E]\rangle f\|_{\Fs_{pp}^s(\R^N)}
    \\
    \lesssim&_M\Big\|\Big(2^{js}\sum_{k=0}^\infty\big( 2^{-M|j-k|+k-k}\|\phi_k\ast f\|_{L^p(\omega)}+2^{-M|j-k|-k}\|D\phi_k\ast f\|_{L^p(\omega)}+2^{-M|j-k|}\|\phi_k\ast f\|_{L^p(\omega)}\big)\Big)_{j=0}^\infty\Big\|_{\ell^p(\N)}
    \\
    \lesssim&_M\Big\|\Big(\sum_{k=0}^\infty 2^{-(M-|s|)|j-k|}2^{ks}\|\phi_k\ast f\|_{L^p(\omega)}\Big)_{j=0}^\infty\Big\|_{\ell^p(\N)}+\Big\|\Big(\sum_{k=0}^\infty 2^{-(M-|s|)|j-k|}2^{k(s-1)}\|D\phi_k\ast f\|_{L^p(\omega)}\Big)_{j=0}^\infty\Big\|_{\ell^p(\N)}
    \\
    \le&\|(2^{-(M-|s|)|j|})_{j=-\infty}^\infty\|_{\ell^1}\Big(\big\|\big(2^{ks}\|\phi_k\ast f\|_{L^p(\omega)}\big)_{k=0}^\infty\big\|_{\ell^p}+\big\|\big(2^{k(s-1)}\|\phi_k\ast Df\|_{L^p(\omega)}\big)_{k=0}^\infty\big\|_{\ell^p}\Big).
\end{align*}
Here the last inequality above is obtained by Young's inequality on $\Z$. Since $M$ is arbitrary, taking $M>|s|$ we have $\|(2^{-(M-|s|)|j|})_{j=-\infty}^\infty\|_{\ell^1}<\infty$.

By Proposition~\ref{Prop::Space::RychFact} \ref{Item::Space::RychFact::EqvNorm} we have $\big\|\big(2^{ks}\|\phi_k\ast f\|_{L^p(\omega)}\big)_{k=0}^\infty\big\|_{\ell^p}\approx\|f\|_{\Fs_{pp}^s(\omega)}$ and $\big\|\big(2^{k(s-1)}\|\phi_k\ast Df\|_{L^p(\omega)}\big)_{k=0}^\infty\big\|_{\ell^p}\approx \|Df\|_{\Fs_{pp}^{s-1}(\omega)}\lesssim\|f\|_{\Fs_{pp}^s(\omega)}$. Therefore $\|\langle X,[d,E]\rangle f\|_{\Fs_{pp}^s(\R^N)}\lesssim\|f\|_{\Fs_{pp}^s(\omega)}$ and we complete the proof.
\end{proof}

\section{Proof of the Theorems}\label{Section::PfThm}
We now go back to the complex domain $\C^n$ and we assume $\Omega\subset\C^n$ to be a bounded smooth domain of finite type. We let $U_1$ be a fixed neighborhood of $b\Omega$ obtained from Lemma~\ref{Lem::Goal::Glue}, let $m_q$ be the $q$-type of $\Omega$ and $r_q:=(n-q+1)\cdot m_q+2q$, for $1\le q\le n$.

Recall the space $\widetilde \Fs_{p\infty}^s(\overline \Uc)$ in Definition~\ref{Defn::Space::TLSpace}. 
We recall that the Bochner–Martinelli kernels always gain 1 derivative here. 
\begin{lem}\label{Lem::PfThm::BMKernel}
    Let $\Uc\subset\C^n$ be a bounded set. Then the Bochner–Martinelli integral $\Bc_qg(z)=\int_\Uc B_{q-1}(z,\cdot)\wedge g$ has boundedness $\Bc_q:\widetilde\Fs_{pr}^s(\overline \Uc;\wedge^{0,q})\to\Fs_{pr}^{s+1}(\Uc;\wedge^{0,q-1})$ for all $0<p,r\le\infty$ and $s\in\R$ such that $(p,r)\notin\{\infty\}\times(0,\infty)$. 
    
%    In particular when $1\le q\le n-1$, $\Bc_q:\widetilde\Fs_{p\infty}^s(\overline \Uc;\wedge^{0,q})\to\Fs_{p\eps}^{s+\frac1{m_q}}(\Uc;\wedge^{0,q-1})$ is bounded for all $\eps>0$.
\end{lem}
The boundedness $\Bc_q:\widetilde\Fs_{\infty q}^s\to\Fs_{\infty q}^{s+1}$ is also true by using \cite[Theorem~1.22]{TriebelTheoryOfFunctionSpacesIV}.
\begin{proof}
The proof is standard. One can see that $B_{q-1}(z,\zeta)$ is simply the linear combinations of the derivatives of the Newtonian potential $G(z-\zeta):=-\frac{(n-2)!}{4\pi^n}|z-\zeta|^{2-2n}$. We need to prove $[f\mapsto G\ast f]:\widetilde\Fs_{pr}^s(\overline\Uc)\to\Fs_{pr}^{s+2}(\Uc)$ for all $0<p,r\le\infty$ and $s\in\R$ such that $(p,r)\notin\{\infty\}\times(0,\infty)$. 

Indeed let $\chi\in\Sc(\R^{2n})$ be such that its Fourier transform has compact support. We define $G_0:=\chi\ast G$ and $G_\infty=G-G_0$. We see that the Fourier transform $\hat G_\infty(\xi)=(1-\hat\chi(\xi))|\xi|^{-2}$ ($\xi\in\R^{2n}$) is a bounded smooth function. Therefore by H\"ormander-Mikhlin multiplier theorem (see \cite[Theorem~2.3.7]{TriebelTheoryOfFunctionSpacesI}), we have $[f\mapsto (I-\Delta)G_\infty\ast f]:\Fs_{pr}^s(\R^{2n})\to\Fs_{pr}^s(\R^{2n})$. Note that $(I-\Delta)^{-2}:\Fs_{pr}^s(\R^{2n})\to\Fs_{pr}^{s+2}(\R^{2n})$ is bounded (see \cite[Theorem~2.3.8]{TriebelTheoryOfFunctionSpacesI}). Therefore $[f\mapsto G_\infty\ast f]:\widetilde\Fs_{pr}^s(\overline\Uc)\to\Fs_{pr}^{s+2}(\Uc)$ is bounded as well. 

On the other hand the Fourier support $\supp\hat G_0\subseteq\supp\hat\chi$ is compact, we see that $G_0\in C^\infty_\loc(\R^{2n})$. Since $\Uc$ is a bounded set, we have $[f\mapsto G_0\ast f]:\{f\in\Ss'(\R^{2n}):\supp f\subseteq\overline{\Uc}\}\to C^\infty_\loc(\R^{2n})$, in particular $[f\mapsto G_0\ast f]:\widetilde\Fs_{pr}^s(\overline\Uc)\to\Fs_{pr}^{s+2}(\Uc)$.

Now $[f\mapsto G\ast f]:\widetilde\Fs_{pr}^s(\overline\Uc)\to\Fs_{pr}^{s+2}(\Uc)$ is bounded. The boundedness of $\Bc_q:\Fs_{pr}^s\to\Fs_{pr}^{s+1}$ follows from the fact that $\nabla:\Fs_{pr}^{s+2}(\R^{2n})\to\Fs_{pr}^{s+1}(\R^{2n};\C^{2n})$ (see \cite[Theorem~2.3.8]{TriebelTheoryOfFunctionSpacesI}).
%
%By Remark~\ref{Rmk::Space::TLRmk} \ref{Item::Space::TLRmk::Embed} we have elementary embedding $\Fs_{p\infty}^{s+1}(\R^{2n})\hookrightarrow\Fs_{p\eps}^{s+1-\delta}(\R^{2n})$ for all $\eps,\delta>0$. 
%Since $m_q\ge2$ when $q<n$, taking restrictions to $\Uc$ we have embedding $\Fs_{p\infty}^{s+1}(\Uc)\hookrightarrow\Fs_{p\eps}^{s+\frac1{m_q}}(\Uc)$.
%The boundedness of $\Bc_q:\Fs_{p\infty}^s\to\Fs_{p\eps}^{s+\frac1{m_q}}$ then follows.
\end{proof}

\begin{proof}[Proof of Theorems \ref{MainThm} and \ref{StrMainThm}]
We prove the definedness of $\Hc_q$ on $\Ss'(\Omega;\wedge^{0,q})$ and the homotopy formula $f=~\dbar\Hc_q f+\Hc_{q+1}\dbar f$ for $f\in\Ss'$ after the proof of boundedness of $\Hc_q$ on Triebel-Lizorkin spaces.

Recall the Rychkov's extension operator $\Ec=\Ec_\Omega$ in Definition~\ref{Defn::Space::ExtOp} and the anti-derivative operators $\Sc^{k,\alpha}=~\Sc^{k,\alpha}_\Omega$ in Proposition~\ref{Prop::Space::ATDOp} (see also Remark~\ref{Rmk::Space::ATDOp}). For $k\ge0$, we define $\Hc_q^k:=\Bc_q\circ\Ec +\Hc_q^{k,\top}+\Hc_q^{k,\bot}$ by the following, where $\Bc_q$ is in Lemma~\ref{Lem::PfThm::BMKernel} and $\Kc_{q,\alpha}^\top,\Kc_{q,\alpha}^\bot$ are in Corollary~\ref{Cor::PfThm::WeiSobEst},
\begin{align*}
    \Hc_q^{k,\top}f(z)&:=\sum_{|\alpha|\le k}(-1)^{|\alpha|}\Kc_{q,\alpha}^\top\circ\Sc^{k,\alpha} [\dbar,\Ec] f=\sum_{|\alpha|\le k}(-1)^{|\alpha|}\int_{U_1\backslash\overline\Omega}(D^\alpha_\zeta (K_{q-1}^\top))(z,\cdot)\wedge\Sc^{k,\alpha} [\dbar,\Ec] f;
    \\
    \Hc_q^{k,\bot}f(z)&:=\sum_{|\alpha|\le k}(-1)^{|\alpha|}\Kc_{q,\alpha}^\bot\circ\Sc^{k,\alpha} [\dbar,\Ec]^\top f=\sum_{|\alpha|\le k}(-1)^{|\alpha|}\int_{U_1\backslash\overline\Omega}(D^\alpha_\zeta (K_{q-1}^\bot))(z,\cdot)\wedge\Sc^{k,\alpha} [\dbar,\Ec]^\top f.
\end{align*}
Clearly $\Hc_q^0=\Hc_q$. We are going to prove $\Hc_q^{k,(\top,\bot)}=\Hc_q^{0,(\top,\bot)}$ for all $k\ge0$, in particular $\Hc_q^k=\Hc_q$ holds.

\medskip
Applying  Lemma~\ref{Lem::PfThm::BMKernel}, for every $1\le p,r\le\infty$ and $s\in\R$ such that $(p,r)\notin\{\infty\}\times[1,\infty)$, we have
\begin{equation}\label{Eqn::PfThm::BddB}
    \Bc_q\circ\Ec:\ \Fs_{pr}^s(\Omega;\wedge^{0,q})\xrightarrow{\Ec}\widetilde\Fs_{pr}^s(\overline{U_1\cup\Omega};\wedge^{0,q})\xrightarrow{\Bc_q}\Fs_{pr}^{s+1}(U_1\cup\Omega;\wedge^{0,q-1})\xrightarrow{(-)|_\Omega}\Fs_{pr}^{s+1}(\Omega;\wedge^{0,q-1}).
\end{equation}
When $q=n$ we have $\Hc_n=\Bc_n\circ\Ec$ since $K_n(z,\zeta)\equiv0$. Therefore by Remark~\ref{Rmk::Space::TLRmk} \ref{Item::Space::TLRmk::SobHold} we get the boundedness $\Hc_n:H^{s,p}(\Omega;\wedge^{0,n})\to H^{s+1,p}(\Omega;\wedge^{0,n-1})$ and $\Hc_n:\Co^s(\Omega;\wedge^{0,n})\to \Co^{s+1}(\Omega;\wedge^{0,n-1})$ for all $1< p<\infty$ and $s\in\R$, which is Theorem~\ref{MainThm} \ref{Item::MainThm::Sob} at $q=n$. 

We have Sobolev embedding $H^{s+1,p}(\C^n)\hookrightarrow H^{s,\frac{2np}{p-2n}}(\C^n)$ for $1< p< 2n$, see \cite[Corollary~2.7]{TriebelTheoryOfFunctionSpacesIV}. Taking restrictions to $\Omega$ we have $\Hc_n:H^{s,p}(\Omega;\wedge^{0,n})\to H^{s,\frac{2np}{p-2n}}(\Omega;\wedge^{0,n-1})$. Since $r_n=1+2n>2n$, the embedding $ H^{s,\frac{2np}{p-2n}}(\Omega;\wedge^{0,n-1})\hookrightarrow  H^{s,\frac{pr_n}{p-r_n}}(\Omega;\wedge^{0,n-1})$ is immediately. This proves  Theorem~\ref{MainThm} \ref{Item::MainThm::Lp} at $q=n$.

\medskip
Now we assume $q\le n-1$ below.

Applying Lemma~\ref{Lem::PfThm::BMKernel}, Remarks \ref{Rmk::Space::ExtBdd} and \ref{Rmk::Space::TLRmk} \ref{Item::Space::TLRmk::Embed} to \eqref{Eqn::PfThm::BddB} we see that
\begin{align*}
    \Bc_q\circ\Ec: \Fs_{p\infty}^s(\Omega;\wedge^{0,q})&\to \widetilde\Fs_{p\infty}^s(\overline{U_1\cup\Omega};\wedge^{0,q})\to \Fs_{p\infty}^{s+1}(\Omega;\wedge^{0,q-1})
    \\
    &\hookrightarrow\begin{cases}\Fs_{p\eps}^{s+\frac1{m_q}}(\Omega;\wedge^{0,q-1})&\text{for all }1\le p\le\infty
    \\
    \Fs_{\frac{pr_q}{r_q-p},\eps}^s(\Omega;\wedge^{0,q-1})&\text{when }1\le p\le r_q
    \end{cases}.
\end{align*}

Note that if we write $f=\sum_{|I|=q}f_Id\overline\zeta^I$, then $[\dbar,\Ec]^\top f$ is the linear combinations of $[\dbar,\Ec]^\top f_I$. Therefore by Corollary~\ref{Cor::LPThm::BddDom} \ref{Item::LPThm::BddDom::TangComm} we have $[\dbar,\Ec]^\top :\Fs_{p\infty}^s(\Omega;\wedge^{0,q})\to\widetilde\Fs_{p\infty}^{s-1/{m_q}}(\overline{U_1\backslash\Omega};\wedge^{0,q+1})$.

For every $s>1-k$ and integer $l>\max(0,s+1)$, applying Remarks \ref{Rmk::Space::ExtBdd} and \ref{Rmk::Space::ATDOp}, Corollaries \ref{Cor::LPThm::BddDom} and \ref{Cor::PfThm::WeiSobEst}, we have, for every $1\le p\le\infty$ and $\eps>0$,
\begin{align*}
    \Hc_q^{k,\top}: &\ \Fs_{p\infty}^s(\Omega;\wedge^{0,q})\xrightarrow{[\dbar,\Ec]}\widetilde\Fs_{p\infty}^{s-1}(\overline{U_1\backslash\Omega};\wedge^{0,q+1})
    \\
    &\qquad\xrightarrow{\Sc^{k,\alpha}}\widetilde\Fs_{p\infty}^{s-1+k}(\overline{U_1\backslash\Omega};\wedge^{0,q+1})\xhookrightarrow{s>1-k} L^p({U_1}\backslash\overline\Omega,\dist^{1-k-s};\wedge^{0,q+1})
    \\
    &\qquad\xrightarrow{\Kc_{q,\alpha}^\top}\begin{cases}
    W^{l,p}(\Omega,\dist^{l-\frac1{m_q}-s};\wedge^{0,q-1})\xhookrightarrow{l>s+\frac1{m_q}}\Fs_{p\eps}^{s+\frac1{m_q}}(\Omega,\wedge^{0,q-1})&\text{for all }1\le p\le\infty
    \\
    W^{l,\frac{pr_q}{r_q-p}}(\Omega,\dist^{l-s};\wedge^{0,q-1})\xhookrightarrow{l>s}\Fs_{\frac{pr_q}{r_q-p},\eps}^{s}(\Omega,\wedge^{0,q-1})&\text{when }1\le p\le r_q
    \end{cases};
    \\
    \Hc_q^{k,\bot}: &\ \Fs_{p\infty}^s(\Omega;\wedge^{0,q})\hookrightarrow\Fs_{pp}^{s-\frac1{m_q}}(\Omega;\wedge^{0,q})\xrightarrow{[\dbar,\Ec]^\top}\widetilde\Fs_{pp}^{s-\frac1{m_q}}(\overline{U_1\backslash\Omega};\wedge^{0,q+1})\hookrightarrow\widetilde\Fs_{p\infty}^{s-\frac1{m_q}}(\overline{U_1\backslash\Omega};\wedge^{0,q+1})
    \\
    &\qquad\xrightarrow{\Sc^{k,\alpha}}\widetilde\Fs_{p\infty}^{s+k-\frac1{m_q}}(\overline{U_1\backslash\Omega};\wedge^{0,q+1})\xhookrightarrow{s>\frac1{m_q}-k} L^p({U_1}\backslash\overline\Omega,\dist^{\frac1{m_q}-k-s};\wedge^{0,q+1})
    \\
    &
    \qquad\xrightarrow{\Kc_{q,\alpha}^\bot}\begin{cases}
    W^{l,p}(\Omega,\dist^{l-\frac1{m_q}-s};\wedge^{0,q-1})\xhookrightarrow{l>s+\frac1{m_q}}\Fs_{p\eps}^{s+\frac1{m_q}}(\Omega,\wedge^{0,q-1})&\text{for all }1\le p\le\infty
    \\
    W^{l,\frac{pr_q}{r_q-p}}(\Omega,\dist^{l-s};\wedge^{0,q-1})\xhookrightarrow{l>s}\Fs_{\frac{pr_q}{r_q-p},\eps}^{s}(\Omega,\wedge^{0,q-1})&\text{when }1\le p\le r_q
    \end{cases}.
\end{align*}
In particular we see that $\Hc_q^{k,\top},\Hc_q^{k,\bot}$ are both defined on $\bigcup_{s>1-k;1\le p\le\infty}\Fs_{p\infty}^s(\Omega;\wedge^{0,q})$.
This completes the proof of Theorem~\ref{StrMainThm} once we show $\Hc_q=\Hc_q^k$ for all $k$. 

Taking integration by parts, for $f\in\Co^\infty(\Omega;\wedge^{0,q})$,
\begin{equation}\label{Eqn::PfThm::Tk=T0}
    \Hc_q^{k,\top}f(z)=\sum_{|\alpha|\le k}\int_{{U_1}\backslash\overline\Omega} K_{q-1}^\top(z,\cdot)\wedge D^\alpha\Sc^{k,\alpha} [\dbar,\Ec] f=\Hc_q^{0,\top}f(z).
\end{equation}
There is no boundary term because of Proposition~\ref{Prop::Space::ATDOp} \ref{Item::Space::ATDOp::Supp}.

The same way we get $\Hc_q^{k,\bot}f=\Hc_q^{0,\bot}f$. Therefore $\Hc_q^kf=\Hc_q^0f(=\Hc_qf)$ for all $f\in\Co^\infty(\Omega;\wedge^{0,q})$. 

For $s>1-k$, $1\le p\le\infty$ and $f\in\Fs_{p\infty}^s(\Omega;\wedge^{0,q})$, we can find an smooth sequence $\{f_j\}_{j=0}^\infty\subset\Co^\infty(\Omega;\wedge^{0,q})$ such that $f_j\xrightarrow{\Fs_{p\infty}^{s'}}f$ for some $s'\in(1-k,s)$. Therefore $\lim_{j\to\infty}\Hc_qf_j=\lim_{j\to\infty}\Hc_q^kf_j=\Hc_q^kf$ give the definedness of $\Hc_qf$, and we see that $\Hc_q=\Hc_q^k$ for all $k\ge0$. This completes the proof of Theorem~\ref{StrMainThm}.

Moreover for such $(f_j)\subset\Co^\infty$, by Lemma~\ref{Lem::Goal::HomotopyFormula} we see that $f_j=\dbar\Hc_q^{k+1} f_j+\Hc_{q+1}^{k+1}\dbar f_j=\dbar\Hc_q f_j+\Hc_{q+1}\dbar f_j$. Therefore taking the limit we see that $f=\dbar\Hc_q f+\Hc_{q+1}\dbar f$ holds as well. 

By Lemma~\ref{Lem::Space::SsUnion}, $\Hc_q$ is defined on $\bigcup_{s}\Co^s(\Omega;\wedge^{0,q})=\Ss'(\Omega;\wedge^{0,q})$. Therefore $f=\dbar\Hc_q f+\Hc_{q+1}\dbar f$ holds for all $f\in\Ss'(\Omega;\wedge^{0,q})$. Now we prove Theorem~\ref{MainThm} \ref{Item::MainThm::Homo}.

Theorem~\ref{MainThm} \ref{Item::MainThm::Sob} and \ref{Item::MainThm::Lp} follow from the inclusions $\Fs_{p1}^s(\Omega)\hookrightarrow H^{s,p}(\Omega)=\Fs_{p2}^s(\Omega)\hookrightarrow\Fs_{p\infty}^s(\Omega)$ and $\Fs_{\infty1}^s(\Omega)\hookrightarrow \Co^s(\Omega)=\Fs_{\infty\infty}^s(\Omega)$, as discussed in Remark~\ref{Rmk::Space::TLRmk} \ref{Item::Space::TLRmk::Embed} and \ref{Item::Space::TLRmk::SobHold}. 
\end{proof}

\section{An Additional Result for Smooth Strongly Pseudoconvex Domains}\label{Section::SPsiCX}

To summarize the technics from Sections \ref{Section::Basis} and \ref{Section::Space}. We see that the correspondence results of Theorem~\ref{MainThm} for bounded smooth strongly pseudoconvex domains also hold.

\begin{thm}\label{Thm::SPsiCX}
    Let $\Omega\subset\C^n$ be a bounded smooth strongly pseudoconvex domain. There are operators $\Hc_q:\Ss'(\Omega;\wedge^{0,q})\to\Ss'(\Omega;\wedge^{0,q-1})$ for $1\le q\le n$, such that $f=\dbar\Hc_q f+\Hc_{q+1}\dbar f$ for all $f\in\Ss'(\Omega;\wedge^{0,q})$ (we set $\Hc_{n+1}=0$). 
    
    Moreover for every $1\le q\le n$, $s\in\R$ and $\eps>0$, $\Hc_q$ has the following boundedness:
    \begin{align}
    \label{Eqn::SPsiCX::Thm::T+}
        &\Hc_q:\Fs_{p,\infty}^s(\Omega;\wedge^{0,q})\to\Fs_{p,\eps}^{s+\frac12}(\Omega;\wedge^{0,q-1}),&\forall\  1\le p\le\infty; 
    \\
    \label{Eqn::SPsiCX::Thm::T0}
        &\Hc_q:\Fs_{p,\infty}^s(\Omega;\wedge^{0,q})\to\Fs_{\frac{(2n+2)p}{2n+2-p},\eps}^s(\Omega;\wedge^{0,q-1}),&\forall\  1\le p\le 2n+2.
    \end{align}

In particular for every $1\le q\le n$ and $s\in\R$, we have boundedness $\Hc_q:\Co^s(\Omega;\wedge^{0,q})\to \Co^{s+\frac12}(\Omega;\wedge^{0,q-1})$, $\Hc_q:H^{s,p}(\Omega;\wedge^{0,q})\to H^{s+\frac12,p}(\Omega;\wedge^{0,q-1})$ for all $1< p<\infty$, and $\Hc_q:H^{s,p}(\Omega;\wedge^{0,q})\to H^{s,\frac{(2n+2)p}{2n+2-p}}(\Omega;\wedge^{0,q-1})$ for all $1< p<2n+2$.
\end{thm}

\begin{rem}
    Theorem \ref{Thm::SPsiCX} improves the result in \cite[Theorem~1.2]{ShiYaoCk}, which proves
$\Hc_q:H^{s,p}(\Omega;\wedge^{0,q})\to H^{s+\frac12,p}(\Omega;\wedge^{0,q-1})$ for all $s\in\R$ and  $1< p<\infty$. For negative $s$ the boundedness $\Hc_q:H^{s,p}\to H^{s,\frac{(2n+2)p}{2n+2-p}}$ is new. Recall that these two results are not comparable since the Sobolev embedding $ H^{s+\frac12,p}\hookrightarrow H^{s,\frac{4np}{4n-p}}$ is not contained in $ H^{s,\frac{(2n+2)p}{2n+2-p}}$. 

By keeping check of the proof using regularized distance functions, one could show that the results for non-smooth domains are true: if $k\ge0$ is an integer and $b\Omega\in C^{k+2}$, then $\Hc_q:H^{s,p}\to H^{s,\frac{(2n+2)p}{2n+2-p}}$ is still true for all $1< p<2n+2$ and $s>\frac1p-k$. We refer the reader to \cite{ShiYaoC2,ShiYaoCk}.
\end{rem}

To prove Theorem~\ref{Thm::SPsiCX} we repeat the arguments in Section~\ref{Section::Basis}. Note that for $1\le q\le n-1$, the (D'Angelo or Catlin) $q$-type of $\Omega$ is always $2$. We shall see later in the proof that the $\top$ and $\bot$ projections are not needed for the estimates.

Recall that we can choose a smooth defining function for $\Omega$ such that it is plurisubharmonic in a neighborhood of $b\Omega$, see for example \cite[Theorem~3.4.4]{ChenShawBook}. In particular there is a $T_0>0$ such that $\Omega_t:=\{\varrho<t\}$ is smooth strongly pseudoconvex for all $-T_0<t<T_0$.

We recall the standard Henkin-Ram\'irez function for strongly pseudoconvex domains: 
\begin{prop}\label{Prop::SPsiCX::LeviMap}
Let $\Omega\subset\C^n$ be a smooth strongly pseudoconvex domain. There are a number $T_1\in (0,T_0]$ associated with a neighborhood $U_1:=\{|\varrho|<T_1\}$ of $b\Omega$, a $c\in(0,\frac12T_1)$, and a map $\widehat Q\in\Co^\infty(\Omega\times U_1;\C^n)$ which is holomorphic in $z$, such that the associated support function $\widehat S(z,\zeta):=\widehat Q(z,\zeta)\cdot(z-\zeta)$ satisfies:
\begin{align}
\label{Eqn::SPsiCX::SBdd}
    &|\widehat S(z,\zeta)|\ge\varrho(\zeta)-\varrho(z)+c|z-\zeta|^2,&\forall z\in\Omega,\ \zeta\in U_1\backslash\overline\Omega\text{ such that }|z-\zeta|\le c;
    \\
    &|\widehat S(z,\zeta)|\ge c^3,&\forall z\in\Omega,\ \zeta\in U_1\backslash\overline\Omega\text{ such that }|z-\zeta|\ge c.
\end{align}
\end{prop}
% \begin{prop}
% Let $\Omega\subset\C^n$ be a smooth pseudoconvex domain. There are a number $T_1\in (0,T_0]$ associated with a neighborhood $U_1:=\{|\varrho|<T_1\}$ of $b\Omega$, a constant $0<c<\frac12T_1$, and a smooth function $\widehat S\in\Co^\infty(\Omega\times U_1;\C)$ which is holomorphic in $z$, such that 
% \begin{itemize}
%     \item $|\widehat S(z,\zeta)|\ge c$ for all $z\in\Omega$, $\varrho\in U_1\backslash\overline\Omega$ such that $|z-\zeta|\ge c$.
%     \item There is a smooth function $A\in\Co^\infty(\Omega\times U_1;\C)$ such that $\widehat S(z,\zeta)=A(z,\zeta)S(z,\zeta)$ when $z\in\Omega$, $\zeta\in U_1\backslash\overline\Omega$, $|z-\zeta|\le c$, where
% \begin{align}
% \label{Eqn::SPsiCX::DefS}
%     S(z,\zeta)&:=\sum_{j=1}^n\frac{\partial\varrho}{\partial\zeta_j}(\zeta)\cdot(z_j-\zeta_j)-\frac12\sum_{j,k=1}^n\frac{\partial^2\varrho}{\partial\zeta_j\partial\zeta_k}(\zeta)\cdot(z_j-\zeta_j)(z_k-\zeta_k),
% \end{align}
% satisfies
% \begin{equation}\label{Eqn::SPsiCX::SBdd}
%     \re S(z,\zeta)\le\varrho(z)-\varrho(\zeta)-c|z-\zeta|^2,\qquad\forall z\in\Omega,\ \zeta\in U_1\backslash\overline\Omega\text{ such that }|z-\zeta|\le c.  
% \end{equation}
% \end{itemize}
% \end{prop}
See \cite[Theorem~III.7.15]{LiebMichelBook}, \cite[Theorem~2.4.3]{HenkinLeitererBook} or \cite[Proposition~5.1]{GongHolderSPsiCXC2}. 

% We similarly define
% \begin{equation}\label{Eqn::SPsiCX::DefQ}
%     \widehat Q(z,\zeta):=\int_0^1(\partial_z\widehat S)(\zeta+t(z-\zeta),\zeta)dt=\sum_{j=1}^n\Big(\int_0^1\tfrac{\partial \widehat S}{\partial z_j}(\zeta+t(z-\zeta),\zeta)dt\Big)d\zeta_j.
% \end{equation}

For such $\widehat S$ and $\widehat Q$, we still use $K(z,\zeta)$ from \eqref{Eqn::Goal::DefK}. We define our solution operators $\Hc_q$ in Definition~\ref{Defn::Goal::DefT}, where the image function of the Rychkov's extension operator $\Ec$ is always supported in $U_1\cup\Omega$. In other words,
%Recall from Definition~\ref{Defn::Goal::BotTop} and Lemma~\ref{Lem::Goal::TopBotDecomp} we have
\begin{equation}\label{Eqn::SPsiCX::DefT}
    \Hc_qf(z)=\int_{U_1\cup\Omega}B_{q-1}(z,\cdot)\wedge\Ec f+\int_{U_1\backslash\overline\Omega}K_{q-1}(z,\cdot)\wedge[\dbar,\Ec]f.
\end{equation}

We define a naive version of $P_\eps(\zeta)$ adapted to $(\Omega,\varrho)$ (cf. Definition~\ref{Defn::Basis::Basis}):
\begin{defn}
    For $\zeta\in \C^n$ and $\eps>0$, we define
    \begin{equation*}
        P_\eps(\zeta):=\{z\in B(\zeta,\eps^\frac12):|\partial\varrho(\zeta)(\zeta-z)|<\eps\}.
    \end{equation*}
\end{defn}
Clearly when $\zeta\in U_1$ and $\eps<\frac12T_1$ the set $P_\eps(\zeta)$ is non-empty. Moreover from \eqref{Eqn::Basis::Prem::Set} there is a $C>0$ such that:
\begin{equation}\label{Eqn::SPsiCX::PSym}
    \zeta\in P_{C^{-1}\eps}(z)\quad\Longrightarrow\quad z\in P_\eps(\zeta)\quad \Longrightarrow\quad\zeta\in P_{C\eps}(z).
\end{equation}

Informally speaking we are setting $\tau_1(\zeta,\eps):=\frac c2\eps$ and $\tau_2(\zeta,\eps)=\dots=\tau_n(\zeta,\eps):=\eps^\frac12$. In the strongly pseudoconvex case there is no difference between using $\eps$-extremal basis and using $\eps$-minimal basis.

Now \eqref{Eqn::SPsiCX::SBdd} implies that there is an $\eps_0>0$ such that the corresponding result for Lemma~\ref{Lem::Basis::EstSQ} \ref{Item::Basis::EstSQ::S} holds:
\begin{equation*}
    |\widehat S(z,\zeta)|\gtrsim\eps,\quad\text{for all }0<\eps\le\eps_0,\quad\zeta\in U_1\backslash\overline\Omega\text{ and }z\in\Omega\backslash P_\eps(\zeta).
\end{equation*}

We do not need the corresponding estimates in Lemma~\ref{Lem::Basis::EstSQ} \ref{Item::Basis::EstSQ::Q}: we see that the trivial estimates for Corollary~\ref{Cor::Basis::EstQ2} hold (see also Remark~\ref{Rmk::Basis::EstQ2::SPsiCX}):
\begin{equation}\label{Eqn::SPsiCX::EstQ}
    |\widehat Q\wedge(\dbar\widehat Q)^k|+|(\dbar\widehat Q)^k|\lesssim1.
\end{equation}

That means a simpler version of Lemma~\ref{Lem::Basis::FinalBasisEst} and Corollary~\ref{Cor::Basis::FinalBasisEst} holds: for $\eps\in(0,\eps_0]$, $1\le k\le n$, $j\ge0$, and for all $(z,\zeta)\in\Omega\times (U_1\backslash\overline\Omega)$ such that $z\notin P_\eps(\zeta)$, we have $\big|D^j_{z,\zeta}\big(\frac{\widehat Q\wedge(\dbar\widehat Q)^{k-1}}{\widehat S^k}\big)(z,\zeta)\big|\lesssim_j\eps^{-j-k}$. Therefore
\begin{equation}\label{Eqn::SPsiCX::EstDK}
    |D^j_{z,\zeta}K_{q-1}(z,\zeta)|\lesssim\sum_{k=1}^{n-q}\frac{\eps^{-j-k}}{|z-\zeta|^{2n-2k-1}},\qquad\forall(z,\zeta)\in\Omega\times (U_1\backslash\overline\Omega)\text{ such that } z\notin P_\eps(\zeta).
\end{equation}
Recall \eqref{Eqn::SPsiCX::PSym}, the above is the same as saying $(z,\zeta)\in\Omega\times (U_1\backslash\overline\Omega)$ such that $\zeta\notin P_\eps(z)$.

Taking integrations on $P_\eps(\zeta)$ and $P_\eps(z)$ we have the correspondent estimate to \eqref{Eqn::Basis::IntEst::Top+} and \eqref{Eqn::Basis::IntEst::Bot+} (recall \eqref{Eqn::Basis::PfIntEst::Pf+} with $\tau_1\approx\eps$ and $\tau_2=\dots=\tau_n=\eps^\frac12$):
\begin{align*}
    &\int_{\Omega \cap P_\eps(\zeta)\backslash {P_{\frac\eps2}(\zeta)}}|D^jK_{q-1}(w,\zeta)|d\Vol_w+\int_{P_\eps(z)\backslash ({P_{\frac\eps2}(z)\cup \Omega})}|D^jK_{q-1}(z,w)|d\Vol_w
    \\
    \lesssim&_j\sum_{k=1}^{n-q}\int_{|w_1|<\eps,|w_2|<\eps^\frac12,\dots,|w_n|<\eps^\frac12}\frac{d\Vol(w_1,\dots,w_n)}{\eps^{j+k}\big(\sum_{l=k+1}^n|w_l|\big)^{2n-2k-1}}\lesssim\eps^{\frac32-j}.
\end{align*}

In our case the types for $\Omega$ are $m_1=\dots=m_{n-1}=2$. Therefore the numbers $r_q=(n-q+1)m_q+2q$ and $\gamma_q=\frac{r_q}{r_q-1}$ are indeed $r_q=2n+2$ and $\gamma_q=\frac{2n+2}{2n+1}$ for all $1\le q\le n-1$. Using \eqref{Eqn::Basis::PfIntEst::Pf0} and \eqref{Eqn::Basis::PfIntEst::Pf0Num} we have the correspondent estimate to \eqref{Eqn::Basis::IntEst::Top0} and \eqref{Eqn::Basis::IntEst::Bot0}:
\begin{align*}
    &\int_{\Omega \cap P_\eps(\zeta)\backslash P_{\frac\eps2}(\zeta)}|D^jK_{q-1}(w,\zeta)|^{\frac{2n+2}{2n+1}}d\Vol_w+\int_{P_\eps(z)\backslash (P_{\frac\eps2}(z)\cup \Omega)}|D^jK_{q-1}(z,w)|^{\frac{2n+2}{2n+1}}d\Vol_w
    \\
    \lesssim&_j\sum_{k=1}^{n-q}\int_{|w_1|<\eps,|w_2|<\eps^\frac12,\dots,|w_n|<\eps^\frac12}\frac{d\Vol(w_1,\dots,w_n)}{\big(\eps^{j+k}\sum_{l=k+1}^n|w_l|\big)^{(2n-2k-1)\gamma_q}}\lesssim\eps^{(1-j)\gamma_q}=\eps^{(1-j)\frac{2n+2}{2n+1}}.
\end{align*}

Therefore, by integrating on the dyadic shells $P_{2^{1-j}\eps}(\zeta)\backslash P_{2^{-j}\eps}(\zeta)$ or $P_{2^{1-j}\eps}(z)\backslash P_{2^{-j}\eps}(z)$, and using \eqref{Eqn::Basis::PfWeiEst::Pf+} and \eqref{Eqn::Basis::PfWeiEst::Pf0}, we obtain the weighted estimates (cf. Theorem~\ref{Thm::WeiEst}): for every $k\ge2$, $0<s<k-\frac32$ and $1\le q\le n-1$,
\begin{align*}
    \int_{U_1\backslash\overline\Omega}\dist(\zeta)^s|D^k_{z,\zeta}K_{q-1}(z,\zeta)|d\Vol(\zeta)&\lesssim_{k,s}\dist(z)^{s+\frac32-k},&&\forall z\in \Omega;
    \\
    \int_{\Omega}\dist(z)^s |D^k_{z,\zeta}K_{q-1}(z,\zeta)|d\Vol(z)&\lesssim_{k,s}\dist(\zeta)^{s+\frac32-k},&&\forall\zeta\in U_1\backslash\overline\Omega;
    \\
    \int_{U_1\backslash\overline\Omega}|\dist(\zeta)^s D^k_{z,\zeta}K_{q-1}(z,\zeta)|^\frac{2n+2}{2n+1}d\Vol(\zeta)&\lesssim_{k,s}\dist(z)^{(s+1-k)\frac{2n+2}{2n+1}},&&\forall z\in \Omega;
    \\
    \int_{\Omega}|\dist(z)^s D^k_{z,\zeta}K_{q-1}(z,\zeta)|^\frac{2n+2}{2n+1}d\Vol(z)&\lesssim_{k,s}\dist(\zeta)^{(s+1-k)\frac{2n+2}{2n+1}},&&\forall \zeta\in U_1\backslash\overline\Omega.
\end{align*}

By Schur's test (Lemma~\ref{Lem::PfThm::Schur}), for $\alpha\in\N^{2n}_\zeta$ and $1\le q\le n-1$, the integral operator 
\begin{equation*}
    \Kc_{q,\alpha}g(z):=\int_{U_1\backslash\overline\Omega}(D^\alpha_{z,\zeta}K_{q-1})(z,\cdot)\wedge g\quad\big(=(-1)^{|\alpha|}\Kc_{q,0}\circ D^\alpha g(z)\big)
\end{equation*}
has the boundedness (cf. Corollary~\ref{Cor::PfThm::WeiSobEst}): for every $k\ge0$ and $1<s<k+|\alpha|-\frac12$ (in particular $k+|\alpha|\ge2$),
\begin{align}\label{Eqn::SPsiCX::WBdd+}
        &\Kc_{q,\alpha}:L^p(U_1\backslash\overline\Omega,\dist^{1-s};\wedge^{0,q+1})\to W^{k,p}(\Omega,\dist^{k+|\alpha|-\frac12-s};\wedge^{0,q-1}),
    & \forall 1\le p\le \infty;
    \\\label{Eqn::SPsiCX::WBdd0}
    &\Kc_{q,\alpha}:L^p(U_1\backslash\overline\Omega,\dist^{1-s};\wedge^{0,q+1})\to W^{k,\frac{(2n+2)p}{p-2n-2}}(\Omega,\dist^{k+|\alpha|-s};\wedge^{0,q-1}),
    & \forall 1\le p\le 2n+2;
\end{align}

Recall the notations $\Bc_q$ in Lemma~\ref{Lem::PfThm::BMKernel} and $\Sc^{k,\alpha}$ in Proposition~\ref{Prop::Space::ATDOp}. Rewriting \eqref{Eqn::SPsiCX::DefT}, and by the same argument to \eqref{Eqn::PfThm::Tk=T0} we have
\begin{equation}\label{Eqn::SPsiCX::TDecomp}
    \Hc_qf=\Bc_q\circ\Ec f+\sum_{|\alpha|\le k}(-1)^{|\alpha|}\Kc_{q,\alpha}\circ\Sc^{k,\alpha}\circ[\dbar,\Ec]f,\qquad\forall k\ge0.
\end{equation}

By Remark~\ref{Rmk::Space::ExtBdd}, Lemma~\ref{Lem::PfThm::BMKernel} and Remark~\ref{Rmk::Space::TLRmk} \ref{Item::Space::TLRmk::Embed}, for every $1\le q\le n$, $s\in\R$, $1\le p\le\infty$ and $\eps>0$,
\begin{align*}
    \Bc_q\circ\Ec:\Fs_{p\infty}^s(\Omega;\wedge^{0,q})&\xrightarrow{\Ec}\widetilde\Fs_{p\infty}^s(\overline{U_1\cup\Omega};\wedge^{0,q})\xrightarrow{\Bc_q}\Fs_{p\infty}^{s+1}(\Omega;\wedge^{0,q-1})
    \\
    &\hookrightarrow
    \begin{cases}
    \Fs_{p\eps}^{s+\frac12}(\Omega;\wedge^{0,q-1})&\forall p\in[1,\infty]
    \\
    \Fs_{\frac{(2n+2)p}{2n+2-p},\eps}^{s}(\Omega;\wedge^{0,q-1})&\forall p\in[1,2n+2]
    \end{cases}.
\end{align*}
In particular we have \eqref{Eqn::SPsiCX::Thm::T+} and \eqref{Eqn::SPsiCX::Thm::T0} for $q=n$, since $\Kc_{n,\alpha}\equiv0$.

Choose integers $k,l\ge0$ such that $k>1-s$ and $l>s+\frac12$. Applying \eqref{Eqn::SPsiCX::WBdd+}, \eqref{Eqn::SPsiCX::WBdd0} and Corollary~\ref{Cor::LPThm::BddDom} to the summands in \eqref{Eqn::SPsiCX::TDecomp}, we have, for $1\le q\le n-1$, $1\le p\le\infty$ and $\eps>0$,
\begin{align*}
    \Fs_{p\infty}^s(\Omega;\wedge^{0,q})&\xrightarrow{[\dbar,\Ec]}\widetilde\Fs_{p\infty}^{s-1}(\overline{U_1\backslash\Omega};\wedge^{0,q+1})\xrightarrow{\Sc^{k,\alpha}}\widetilde\Fs_{p\infty}^{s-1+k}(\overline{U_1\backslash\Omega};\wedge^{0,q+1})\xhookrightarrow{s>1-k} L^p(U_1\backslash\overline\Omega,\dist^{k-1-s};\wedge^{0,q+1})
    \\
    &\xrightarrow{\Kc_{q,\alpha}}\begin{cases}
    W^{l,p}(\Omega,\dist^{l-\frac12-s};\wedge^{0,q-1})\xhookrightarrow{l>s+\frac12}\Fs_{p\eps}^{s+\frac12}(\Omega,\wedge^{0,q-1})&\text{for all }1\le p\le\infty
    \\
    W^{l,\frac{(2n+2)p}{2n+2-p}}(\Omega,\dist^{l-s};\wedge^{0,q-1})\xhookrightarrow{l>s}\Fs_{\frac{(2n+2)p}{2n+2-p},\eps}^{s}(\Omega,\wedge^{0,q-1})&\text{when }1\le p\le 2n+2
    \end{cases}.
\end{align*}

This completes the proof of \eqref{Eqn::SPsiCX::Thm::T+} and \eqref{Eqn::SPsiCX::Thm::T0}.

The H\"older bound $\Hc_q:\Co^s\to\Co^{s+\frac12}$ and Sobolev bounds $\Hc_q:H^{s,p}\to H^{s+\frac12,p}$, $\Hc_q:H^{s,p}\to H^{s,\frac{(2n+2)p}{2n+2-p}}$ follow from the inclusions $\Fs_{p1}^s(\Omega)\hookrightarrow H^{s,p}(\Omega)=\Fs_{p2}^s(\Omega)\hookrightarrow\Fs_{p\infty}^s(\Omega)$ and $\Fs_{\infty1}^s(\Omega)\hookrightarrow \Co^s(\Omega)=\Fs_{\infty\infty}^s(\Omega)$, as discussed in Remark~\ref{Rmk::Space::TLRmk} \ref{Item::Space::TLRmk::Embed} and \ref{Item::Space::TLRmk::SobHold}. \qed
% We see that \eqref{Eqn::SPsiCX::DefS} and \eqref{Eqn::SPsiCX::DefQ} imply Lemma~\ref{Lem::Basis::EstSQ} \ref{Item::Basis::EstSQ::Q}: that is, for $\zeta_0\in U_1$ and for a unitary matrix $\Psi_0\in\C^{n\times n}$ that satisfies $\Psi_0\cdot[1,0,\dots,0]^\intercal=\frac{\dbar\varrho}{|\dbar\varrho|}(\zeta_0)$, then for $\widehat Q_{\Psi_0}(z,\zeta):=\overline\Psi_0\cdot \widehat Q(\Psi_0\cdot z,\Psi_0\cdot\zeta)$ we have
% \begin{equation*}
%         |\widehat Q_{\Psi_0,j}(z,\zeta_0)|\lesssim\eps^\frac12,\quad\big|\tfrac\partial{\partial\overline\zeta_k}\widehat Q_{\Psi_0,j}(z,\zeta_0)\big|\lesssim1,\qquad\text{for }0<\eps\le\eps_0,\quad z\in P_\eps(\zeta_0),\quad 1\le j,k\le n.
% \end{equation*}

\appendix
\section{Proof of Proposition~\ref{Prop::LPThm::HLLem} \ref{Item::LPThm::HLLem::WtoF} for $\eps<1$}

Let $\omega=\{x_1>\sigma(x')\}\subset\R^N$ be a special Lipschitz domain. To give a direct proof of Proposition~\ref{Prop::LPThm::HLLem} \ref{Item::LPThm::HLLem::WtoF}, we define the inner strips (cf. \eqref{Eqn::Space::HLLem::Sk}) of $\omega$: we define $P_0=P_0^\omega:=\{(x_1,x'):x_1-\sigma(x')>2^{-\frac12}\}$ and
\begin{gather}\label{Eqn::App::Pk}
    P_k=P_k^\omega:=\{(x_1,x'):2^{-\frac12-k}<x_1-\sigma(x')<2^{\frac12-k}\}\subset\omega^c\text{ for }k\in\Z_+;
    \\\label{Eqn::App::P<k}
    P_{<k}=P_{<k}^\omega:=\{(x_1,x'):x_1-\sigma(x')>2^{\frac12-k}\}\subset\omega^c\text{ for }k\in\Z_+.
\end{gather}
Now  $\{P_k\}_{k=0}^\infty$ is a partition of $\omega$ up to zero measured sets. Similar to \eqref{Eqn::Space::HLLem::Tmp1} (recall that $\delta(x)=\min(1,\dist(x,b\omega))$), 
\begin{equation}\label{Eqn::App::HLLem::PkTmp}
    2^{-1-k}\le\delta(x)\le2^{\frac12-k}\text{ for all }k\ge0,\quad x\in P_k.
\end{equation}Similar to \eqref{Eqn::Space::HLLem::Tmp2} (see \cite[Lemma~5.3]{ShiYaoExt}) we can also find a $R>0$ such that $\supp\phi_j+P_{k}\subseteq\omega^c$ whenever $k\ge j+R$. Thus
\begin{equation}\label{Eqn::App::HLLem::Tmp1}
    \1_{P_k}\cdot(\phi_j\ast f)=\1_\omega\cdot(\phi_j\ast(f\cdot\1_{P_{<\min(j+R,k)}})),\quad\forall j\ge0.
\end{equation}

For the case $\eps<1$ the duality argument does not work, since $\Fs_{p\eps}^s$ is no longer a locally convex space. In the proof of the case $1\le p<\infty$ we need a version of locally constant principle.
\begin{lem}[Locally constant]\label{Lem::App::LocConst}
    Let $\phi=(\phi_j)_{j=0}^\infty\subset\Ss(\R^N)$ satisfies the scaling condition \ref{Item::Space::PhiScal} (in Definition~\ref{Defn::Space::ExtOp}). Then for any $M>0$ there is a $C_{\phi,M}>0$ such that for every $1\le p\le\infty$ and $f\in L^p(\R^N)$,
    \begin{equation}\label{Eqn::App::LocConst::Infty}
        |\phi_j\ast f(x)|\le C_{\phi, M}2^\frac{Nj}p\sum_{v\in\Z^N}\frac1{(1+|v|)^M}\|f\|_{L^p(x+2^{-j}v+[0,2^{-j}]^N)},\quad\forall j\ge0,\quad x\in\R^N.
    \end{equation}
    In particular if $\phi$ also satisfies the support condition \ref{Item::Space::PhiSupp}, then there is a $C_{\phi}>0$, for every $1\le p\le\infty$, $j\ge0$ and $0\le k\le j+R$,
    \begin{equation}\label{Eqn::App::LocConst::P}
        \|\phi_j\ast f\|_{L^p(P_k)}\le C_\phi \min(1,2^\frac{j-k}p)\|f\|_{L^p(P_{<j+R})},\quad\forall f\in L^p(P_{<j+R}),
    \end{equation}
    where $R>0$ is given in \eqref{Eqn::App::HLLem::Tmp1} and $\dist(x)=\dist(x,b\omega)$.
    
    In particular $\|\phi_j\ast f\|_{L^\infty(P_k)}\lesssim \|f\|_{L^\infty(P_{<j+R})}$.
\end{lem}
\begin{proof}
We denote $Q_{j,v}:=2^{-j}v+[0,2^{-j}]^N$ for $j\ge0$ and $v\in\Z^N$. Therefore for every $x\in\R^N$,
\begin{equation*}
    \phi_j\ast f(x)=\sum_{v\in\Z^N}\phi_j\ast (f\cdot \1_{(x+Q_{j,v})})(x)=\sum_{v\in\Z^N}(\phi_j\cdot\1_{B(0,2^{-j}\max(0,|v|-\sqrt N))^c})\ast (f\cdot \1_{(x+Q_{j,v})})(x).
\end{equation*}
Therefore by H\"older's inequality $|\phi_j\ast f(x)|\le\sum_v\|\phi_j\|_{L^{p'}(B(0,2^{-j}\max(0,|v|-\sqrt N))^c)}\|f\|_{L^p(x+Q_{j,v})}$.

Note that $\phi_0,\phi_1$ are Schwartz, so for $j\in\{0,1\}$ and $l\in\Z$ we have $\int_{|x|>2^l}|\phi_j(x)|dx\lesssim_M2^{-M\max(0,l)}$, which means $\|\phi_j\|_{L^{p'}(B(0,2^l)^c)}\lesssim_M2^{-M\max(0,l)}$ for $j\in\{0,1\}$. Therefore by the scaling condition \ref{Item::Space::PhiScal} we have, for every $j\ge0$,
\begin{equation*}
    \|\phi_j\|_{L^{p'}(B(0,2^{-j}\max(0,|v|-\sqrt N))^c)}\le 2^{\frac {Nj}p}(\|\phi_0\|_{L^{p'}(B(0,\max(|v|-\sqrt N))^c)}+ \|\phi_0\|_{L^{p'}(B(0,\frac12\max(|v|-\sqrt N))^c)})\lesssim_M\tfrac{2^{ {Nj}/p}}{(1+|v|)^M}.
\end{equation*}
Therefore $|\phi_j\ast f(x)|\lesssim_M2^{ {Nj}/p}\sum_v(1+|v|)^{-M}\|f\|_{L^p(x+Q_{j,v})}$ which gives \eqref{Eqn::App::LocConst::Infty}.

\smallskip
To prove \eqref{Eqn::App::LocConst::P}, by Fubini's theorem, 
\begin{align*}
    \int_{P_k}|\phi_j\ast f|^p=&\int_{\R^{N-1}}\int_{2^{-\frac12-k}}^{2^{\frac12-k}}|\phi_j\ast f(t+\sigma(x'),x')|^pdtdx'
    \\
    \le&2^{-k}\sup_{2^{-1/2-k}<t<2^{1/2-k}}\int_{\R^{N-1}}\big|(\phi_j\ast f)(t+\sigma(x'),x')\big|^pdx'.
\end{align*} Therefore by taking $M>N$,
\begin{align*}
    &\|\phi_j\ast f\|_{L^p(P_k)}\le2^{-\frac kp}\sup_{2^{-1/2-k}<t<2^{1/2-k}}\big\|x'\mapsto \big|(\phi_j\ast f)(t+\sigma(x'),x')\big|\big\|_{L^p(\R^{N-1}_{x'})}
    \\
    \le&2^{-\frac kp}2^{\frac jp}\sup_{2^{-1/2-k}<t<2^{1/2-k}}\big\|u'\mapsto \big\|\phi_j\ast (f\1_{P_{<j+R}})\big\|_{L^p((t+\sigma(u'),u')+B(0,2^{-j}\sqrt N))}\big\|_{\ell^p(2^{-j}\Z^{N-1}_{u'})}
    \\
    \le&2^{\frac {j-k}p}\big\|u'\mapsto \big\|\phi_j\ast (f\1_{P_{<j+R}})\big\|_{L^p((\sigma(u'),u')+B(0,2^{1-j}\sqrt N))}\big\|_{\ell^p(2^{-j}\Z^{N-1}_{u'})}
    \\
    \lesssim&_{M}2^{\frac{j-k}p}2^\frac{Nj}p\sum_{v\in\Z^N}(1+|v|)^{-M}\big\|u'\mapsto \big\|f\1_{P_{<j+R}}\big\|_{L^p((\sigma(u'),u')+B(0,2^{1-j}\sqrt N)+Q_{j,v})}\big\|_{\ell^p(2^{-j}\Z^{N-1}_{u'})}
    \\
    \le&2^{\frac{j-k}p}2^\frac{Nj}p\sum_{v\in\Z^N}(1+|v|)^{-M}\big\|u\mapsto \big\|f\1_{P_{<j+R}}\big\|_{L^p(u+B(0,2^{2-j}\sqrt N)+Q_{j,v})}\big\|_{\ell^p(2^{-j}\Z^N_{u})}
    \\
    \lesssim&_{M}2^{\frac{j-k}p}2^\frac{Nj}p\big|B(0,2^{2-j}\sqrt N)\big|\big\|\big((1+|v|)^{-M}\big)_{v\in\Z^N}\big\|_{\ell^1(\Z^N)}\|f\1_{P_{<j+R}}\|_{L^p(\R^N)}\lesssim_{M}2^{\frac{j-k}p}\|f\|_{L^p(P_{<j+R})}.
\end{align*}
This completes the proof of \eqref{Eqn::App::LocConst::P}. % Taking $M=0$ we get $\|\phi_j\ast f\|_{L^p(P_k)}\lesssim \min(1,2^\frac{j-k}p)\|f\|_{L^p(P_{<j+R})}$.
\end{proof}

\begin{proof}[Proof of Proposition~\ref{Prop::LPThm::HLLem} \ref{Item::LPThm::HLLem::WtoF}]We only need to prove $m=0$. Indeed if the case $m=0$ is true, then by Proposition~\ref{Prop::Space::RychFact} \ref{Item::Space::RychFact::EqvNorm2}, we have for every $s<m$,
\begin{equation*}\textstyle
    \|f\|_{\Fs_{p\eps}^s(\omega)}\approx_{s,m,\eps}\sum_{|\alpha|\le m}\|D^\alpha f\|_{\Fs_{p\eps}^{s-m}(\omega)}\overset{\text{case }m=0}{\lesssim}\sum_{|\alpha|\le m}\|D^\alpha f\|_{L^p(\omega,\delta^{m-s})}\approx\|f\|_{W^{m,p}(\delta^{m-s})}.
\end{equation*}

Now assume $m=0$ and $0<\eps<1$. By Proposition~\ref{Prop::Space::RychFact} \ref{Item::Space::RychFact::EqvNorm} we have $\|f\|_{\Fs_{p\eps}^s(\omega)}\approx\|(2^{js}\phi_j\ast f)_{j=0}^\infty\|_{L^p(\omega;\ell^\eps)}$ for $1\le p<\infty$ and $\|f\|_{\Fs_{\infty\eps}^s(\omega)}\approx\sup_{x,J}2^{NJ\frac1\eps}\|(2^{js}\phi_j\ast f)_{j=\max(0,J)}^\infty\|_{L^\eps(\omega\cap B(x,2^{-J});\ell^\eps)}$. We need to show that for every $t>0$ and $\eps>0$,
\begin{gather}
\label{Eqn::App::PfHLLem::<infty}
    \big\|\big(2^{-jt}\phi_j\ast f\big)_{j=0}^\infty\big\|_{L^p(\omega;\ell^\eps(\N))}\lesssim_{\phi,p,\eps,t}\|\delta^tf\|_{L^p(\omega)},\quad 1<p<\infty;
    \\
\label{Eqn::App::PfHLLem::=infty}
    2^{NJ}\sup_{x\in\omega,J\in\Z}\big\|\big(2^{-jt}\phi_j\ast f\big)_{j=\max(0,J)}^\infty\big\|_{L^\eps(\omega\cap B(x,2^{-J});\ell^\eps(\N))}^\eps\lesssim_{\phi,\eps,t}\|\delta^t f\|_{L^\infty(\omega)}^\eps.
\end{gather}

We first prove \eqref{Eqn::App::PfHLLem::<infty} using Lemma~\ref{Lem::App::LocConst}. For $k\ge0$, recall $P_k,P_{<k}$ from \eqref{Eqn::App::Pk}:
\begin{equation*}%\label{Eqn::App::PfHLLem::Main}
    \begin{aligned}
    &\big\|\big(2^{-jt}\phi_j\ast f\big)_{j=0}^\infty\big\|_{L^p(P_{k};\ell^\eps(\N))}=\big\|\big(2^{-jt}\phi_j\ast (f\1_{P_{<\min(j+R,k)}})\big)_{j=0}^\infty\big\|_{L^p(P_{k};\ell^\eps(\N))}
    \\
    \le&\big\|\big(2^{-jt}\phi_j\ast (f\1_{P_{<\min(j+R,k)}})\big)_{j=0}^\infty\big\|_{\ell^\eps(\N;L^p(P_{ k}))}&(\text{by }\eqref{Eqn::App::PfHLLem::Tmp1}\text{ below})
    \\
    \le&2^{\frac1\eps}\Big(\big\|\big(2^{-jt}\|\phi_j\ast (f\1_{P_{<j+R}})\|_{L^p(P_{k})}\big)_{j=0}^{k-R}\big\|_{\ell^\eps}+\big\|\big(2^{-jt}\|\phi_j\ast (f\1_{P_{<k}})\|_{L^p(P_{k})}\big)_{j=\max(0,k-R)}^\infty\big\|_{\ell^\eps}\Big)\hspace{-0.7in}
    \\
    \lesssim&_{\eps,\phi,M}\big\|\big(2^{-jt}2^\frac{j-k}p\|f\|_{L^p(P_{<j+R})}\big)_{j=0}^{k-R}\big\|_{\ell^\eps}+\big\|\big(2^{-jt}\|f\|_{L^p(P_{<k})}\big)_{j=\max(0,k-R)}^\infty\big\|_{\ell^\eps}\hspace{-0.2in}&(\text{by }\eqref{Eqn::App::LocConst::P})
    \\
    \lesssim&_{R}\big\|\big(2^\frac{j-k}p2^{-jt}\|f\|_{L^p(P_{<j})}\big)_{j=1}^k\big\|_{\ell^\eps}+\|(2^{ls})_{l=-R}^\infty\|_{\ell^\eps}2^{-tk}\|f\|_{L^p(P_{<k})}
    \\
    \lesssim&_{\eps,R,t}\big\|\big( 2^{-\frac{|j-k|}p}2^{-jt}\|f\|_{L^p(P_{<j})}\big)_{j=1}^\infty\big\|_{\ell^\eps}.
%    \\
%    =&\textstyle\big(\sum_{j=0}^\infty2^{-\frac\eps p|j-k|}2^{-jt\eps}\|f\|_{L^p(P_{<j})}^\eps\big)^{1/\eps}.
\end{aligned}
\end{equation*}

Here the first inequality is a variant of Minkowski's inequality: since we assume $0<\eps\le 1$,
\begin{equation}\label{Eqn::App::PfHLLem::Tmp1}
    \|(g_j)_j\|_{L^p(\ell^\eps)}=\|(|g_j|^\eps)_j\|_{L^{p/\eps}(\ell^1)}^\frac1\eps=\|(|g_j|^\eps)_j\|_{\ell^1(L^{p/\eps})}^\frac1\eps=\|(g_j)_j\|_{\ell^\eps(L^p)}.
\end{equation}
Therefore
\begin{align*}
    &\|f\|_{\Fs_{p\eps}^{-t}(\omega)}\approx\big\|\big(2^{-jt}\phi_j\ast f\big)_{j=0}^\infty\big\|_{L^p(\omega;\ell^\eps(\N))}=\big\|\big(\|(2^{-jt}\phi_j\ast f)_{j=0}^\infty\|_{L^p(P_k;\ell^\eps(\N)}\big)_{k=0}^\infty\big\|_{\ell^p(\N)}\hspace{-0.3in}
    \\
    \lesssim&_{\eps,p,t}\textstyle\big\|\big(\sum_{j=0}^\infty2^{-\frac\eps p|j-k|}2^{-jt\eps}\|f\|_{L^p(P_{<j})}^\eps\big)_{k=0}^\infty\big\|_{\ell^{p/\eps}(\N)}^{1/\eps}
    \\
    \le&\big\|(2^{-\frac\eps p|j|})_{j=-\infty}^\infty\big\|_{\ell^1}^{1/\eps}\big\|\big(2^{-jt\eps}\|f\|_{L^p(P_{<j})}^\eps\big)_{j=0}^\infty\big\|_{\ell^{p/\eps}(\N)}^{1/\eps}&(\text{by Young's inequality})
    \\
    \lesssim&_{\eps,p}\textstyle\big\|\big(2^{-jt}\|f\|_{L^p(P_{<j})}\big)_{j=0}^\infty\big\|_{\ell^{p}(\N)}=\big\|\big(\sum_{j=k}^\infty2^{-jt}\|f\|_{L^p(P_k)}\big)_{k=0}^\infty\big\|_{\ell^{p}(\N)}
    \\
    \approx&\big\|\big(2^{-kt}\|f\|_{L^p(P_k)}\big)_{k=0}^\infty\big\|_{\ell^{p}(\N)}\approx\big\|\big(\|\delta^tf\|_{L^p(P_k)}\big)_{k=0}^\infty\big\|_{\ell^{p}(\N)}=\|f\|_{L^p(\omega,\delta^t)}.&(\text{by }t>0\text{ and }\eqref{Eqn::App::HLLem::PkTmp})
\end{align*}
This completes the proof of \eqref{Eqn::App::PfHLLem::<infty}.

To prove \eqref{Eqn::App::PfHLLem::=infty}, let $x\in\omega$ and let $k_0\ge0$ be such that $x\in\overline {P_{k_0}}$. We separate the discussion of the norm $\|(2^{-jt}\phi_j\ast f)_{j\ge J}\|_{L^\eps(\omega\cap B(x,2^{-J});\ell^\eps)}$ between $J\le k_0+1$ and $J\ge k_0+2$.

When $J\le k_0+1$, we have $|P_k\cap B(x,2^{-J})|\lesssim 2^{-k}2^{-(N-1)J}$ if $k\ge J-2$ and $P_k\cap B(x,2^{-J})=\varnothing$ if $k\le J-3$. By \eqref{Eqn::App::HLLem::Tmp1} and \eqref{Eqn::App::LocConst::P}, $\|\phi_j\ast f\|_{L^\infty(P_k)}\lesssim_{\phi,R}2^{t\min(j+R,k)}\|f\|_{L^\infty(\omega,\delta^t)}$, therefore
\begin{align*}
    &\textstyle2^{NJ}\sum_{j=\max(0,J)}^\infty\int_{\omega\cap B(x,2^{-J})}2^{-jt\eps}|\phi_j\ast f|^\eps=2^{NJ}\sum_{j=\max(0,J)}^\infty\sum_{k=J-2}^\infty\int_{P_k\cap B(x,2^{-J})}2^{-jt\eps}|\phi_j\ast f|^\eps
    \\
    \lesssim&\textstyle2^{NJ}\|f\|_{L^\infty(\omega,\delta^t)}^\eps\sum_{k=J-2}^\infty\sum_{j=\max(0,J)}^\infty|P_k\cap B(x,2^{-J})|2^{-jt\eps}2^{t\eps\min(j+R,k)}
    \\
    =&\textstyle2^{NJ}\|f\|_{L^\infty(\omega,\delta^t)}^\eps\sum_{k=J-2}^\infty2^{-k}2^{(N-1)J}\big(k-(J-2)+\sum_{j=k}^\infty2^{t\eps(k-j)}\big)
    \\
    \approx&\|f\|_{L^\infty(\omega,\delta^t)}^\eps2^{tJ\eps}\lesssim\|f\|_{L^\infty(\omega,\delta^t)}^\eps2^{tk_0\eps}.
\end{align*}

When $J\ge k_0+2$, we have $B(x,2^{-J})\subset P_{>k_0+2}$, therefore
\begin{align*}
    &\textstyle2^{NJ}\sum_{j=\max(0,J)}^\infty\int_{\omega\cap B(x,2^{-J})}2^{-jt\eps}|\phi_j\ast f|^\eps=2^{NJ}\sum_{j=\max(0,J)}^\infty\int_{P_{>k_0+2}\cap B(x,2^{-J})}2^{-jt\eps}|\phi_j\ast f|^\eps
    \\
    \lesssim&\textstyle2^{NJ}|B(x,2^{-J})|\sum_{j=k_0+2}^\infty2^{-jt\eps}2^{k_0t\eps}\|f\|_{L^\infty(P_{>k_0+2})}^\eps\lesssim\|f\|_{L^\infty(\omega,\delta^t)}^\eps 2^{tk_0\eps}.
\end{align*}
This completes the proof of \eqref{Eqn::App::PfHLLem::=infty} and hence the whole proof.
\end{proof}

\section{Boundedness of a Skew Bergman Projection}

Note: the Sobolev estimate of the following result appears in the preprint \cite{YaoZhangProduct1}, but not in the journal version of the current paper. We attach the statement and proof here for completeness.
\begin{thm}\label{Thm::SkewBerg}
    Let $(\Hc_q)_{q=1}^n$ be the homotopy operators given in Theorem~\ref{MainThm}, the skew Bergman projection $\Pc f:=f-\Hc_1\dbar f$ for $f\in\Ss'(\Omega)$ shares the boundedness
    \begin{equation}\label{Eqn::SkewBerg::TLBdd}
        \Pc:\Fs_{p\infty}^s(\Omega)\to\Fs_{p\eps}^s(\Omega),\quad\forall\ \eps>0,\quad s\in\R,\quad 1\le p\le\infty.
    \end{equation}
    In particular $\Pc:H^{s,p}(\Omega)\to H^{s,p}(\Omega)$ and $\Pc:\Co^s(\Omega)\to\Co^s(\Omega)$ for all $s\in\R$ and $1<p<\infty$.
\end{thm}

The proof of Theorem~\ref{Thm::SkewBerg} follows from the estimate of the correspondent \textit{Cauchy-Fantappi\`e form}: we set
\begin{equation}\label{Eqn::SkewBerg::CF}
    F(z,\zeta):=B(z,\zeta)-\dbar_{z,\zeta}K(z,\zeta)\qquad\text{for }\quad z\in\Omega,\quad \zeta\in\Uc\backslash\overline\Omega, 
\end{equation}
   where $B(z,\zeta)$ and $K(z,\zeta)$ are from \eqref{Eqn::Goal::DefB} and \eqref{Eqn::Goal::DefK}. Recall from \cite[Lemma~11.1.1]{ChenShawBook} we have
\begin{equation}\label{Eqn::SkewBerg::DefF}
    F(z,\zeta)=\frac{\widehat Q\wedge(\dbar\widehat Q)^{n-1}}{(2\pi i)^n\widehat S(z,\zeta)^n}=\frac{\widehat Q(z,\zeta)\wedge (\dbar\widehat Q(z,\zeta))^{n-1}}{(2\pi i)^n(\widehat Q(z,\zeta)\cdot(\zeta-z))^n},\quad z\in\Omega,\quad\zeta\in\Uc\backslash\overline\Omega.
\end{equation}
Note that $F$ is a $(n,n-1)$ bi-degree form which has degree $(0,0)$ in $z$ and $(n,n-1)$ in $\zeta$. Using Notation~\ref{Note::Goal::KTopBot} we have decomposition $F=F^\top+F^\bot$ in $\zeta$-variable.

Theorem~\ref{Thm::SkewBerg} follows from the correspondent weighted estimates.
\begin{prop}
    Let $\delta(w):=\dist(w,b\Omega)$. Then for any $s\in\R$ and $k\in\Z_+$ such that $0<s<k-1$, there is a $C=C(\Omega,\Uc,\widehat Q,k,s)>0$ such that
\begin{align}
    \label{Eqn::SkewBerg::FTopz}
    \int_{\Uc\backslash\overline\Omega}\delta(\zeta)^s |D^k_{z,\zeta}(F^\top)(z,\zeta)|d\Vol(\zeta)&\le C\delta(z)^{s+1-k},&&\forall z\in \Omega;
    \\
    \label{Eqn::SkewBerg::FTopzeta}
    \int_{\Omega}\delta(z)^s |D^k_{z,\zeta}(F^\top)(z,\zeta)|d\Vol(z)&\le C\delta(\zeta)^{s+1-k},&&\forall \zeta\in \Uc\backslash\overline\Omega;
    \\
    \label{Eqn::SkewBerg::FBotz}
    \int_{\Uc\backslash\overline\Omega}\delta(\zeta)^s|D^k_{z,\zeta}(F^\bot)(z,\zeta)|d\Vol(\zeta)&\le C\delta(z)^{s+\frac1{m_1}-k},&&\forall z\in \Omega;
    \\
    \label{Eqn::SkewBerg::FBotzeta}
    \int_{\Omega}\delta(z)^s |D^k_{z,\zeta}(F^\bot)(z,\zeta)|d\Vol(z)&\le C\delta(\zeta)^{s+\frac1{m_1}-k},&&\forall\zeta\in \Uc\backslash\overline\Omega.
\end{align}
As a result if we define for every $\alpha\in\N^{2n}_{\zeta,\bar\zeta}$,
$$\Fc_\alpha^\top g(z):=\int_{\Uc\backslash\overline\Omega}D_\zeta^\alpha (F^\top)(z,\cdot)\wedge g,\quad \Fc_\alpha^\bot g(z):=\int_{\Uc\backslash\overline\Omega}D_\zeta^\alpha (F^\bot)(z,\cdot)\wedge g,\qquad g\in L^1(\Uc\backslash\overline\Omega; \wedge^{0,1}),$$
then in terms of Definition~\ref{Defn::Space::TLSpace}, for every $s>0$, $1\le p\le\infty$ and $\eps>0$, 
\begin{equation}\label{Eqn::SkewBerg::FOpBdd}
    \Fc_\alpha^{\top}:\widetilde \Fs_{p\infty}^s(\overline\Uc\backslash\Omega;\wedge^{0,1})\to \Fs_{p\eps}^{s+1-|\alpha|}(\Omega),\qquad \Fc_\alpha^\bot:\widetilde \Fs_{p\infty}^s(\overline\Uc\backslash\Omega;\wedge^{0,1})\to \Fs_{p\eps}^{s+\frac1{m_1}-|\alpha|}(\Omega).
\end{equation}
\end{prop}
\begin{proof}
    Notice that for $0<\eps<\eps_0$ and $\zeta\in\Uc\backslash\Omega$ we have $|P_\eps(\zeta)\backslash P_{\eps/2}(\zeta)|\le \prod_{l=1}^n\tau_l(\zeta,\eps)^2$. By \eqref{Eqn::Basis::Prem::Set} for $-\eps_0<\varrho(z)<0$ we have $|P_\eps(z)\backslash P_{\eps/2}(z)|\le \prod_{l=1}^n\tau_l(z,\eps)^2$ as well. Therefore by \eqref{Eqn::Basis::FinalBasisEst::Top}, \eqref{Eqn::Basis::FinalBasisEst::Bot} and \eqref{Eqn::SkewBerg::DefF}, for every $k\ge0$ there is a $C_k>0$, such that for every $0<\eps\le\eps_0$, $z\in\Omega$ and $\zeta\in\Uc\backslash\Omega$,
    \begin{gather}
        \label{Eqn::SkewBerg::BasisTop}
    \int_{\Omega \cap P_\eps(\zeta)\backslash P_{\frac\eps2}(\zeta)}|D^k(F^\top)(w,\zeta)|d\Vol(w)+\int_{P_\eps(z)\backslash (P_{\frac\eps2}(z)\cup \Omega)}|D^k(F^\top)(z,w)|d\Vol(w)\le C_k\eps^{1-k};
    \\
\label{Eqn::SkewBerg::BasisBot}
    \int_{\Omega \cap P_\eps(\zeta)\backslash P_{\frac\eps2}(\zeta)}|D^k(F^\bot)(w,\zeta)|d\Vol(w)+\int_{P_\eps(z)\backslash(P_{\frac\eps2}(z)\cup \Omega)}|D^k(F^\bot)(z,w)|d\Vol(w)\le C_k\eps^{\frac1{m_1}-k}.
    \end{gather}

    For $z\in\Omega$, since $F$ is bounded  and smooth uniformly for $\delta(z)\ge \eps_0$, it suffices to prove the estimates for $\delta(z)<\eps_0$. Let $J\in\Z$ be the unique number such that $2^{-J}\eps_0\le\varrho(z)<2^{1-J}\eps_0$. Therefore for all $0<\eps\le \eps_0$, one has $P_{\eps_0}(z)\backslash\overline\Omega \subset \cup_{j=1}^J  P_{2^{1-j}\eps_0}(z)\backslash(P_{2^{-j}\eps_0}(z)\cup\Omega)$.
    
     Clearly $\int_{\Uc\backslash(P_{\eps_0}(z)\cup\Omega)}\delta(\zeta)^s|D^kF^\top(z,\zeta)|d\Vol_\zeta\lesssim1\le\delta(z)^{s+1-k}$ since $F$ is bounded smooth when $\zeta\notin P_{\eps_0}(z)$. Applying \eqref{Eqn::SkewBerg::BasisTop} we get \eqref{Eqn::SkewBerg::FTopz}:
\begin{equation}\label{Eqn::SkewBerg::PfTop}
    \begin{aligned}
    &\int_{P_{\eps_0}(z)\backslash\overline\Omega}\delta(\zeta)^s|D^kF^\top(z,\zeta)|d\Vol_\zeta\lesssim_{k}\sum_{j=1}^J\int_{P_{2^{1-j}\eps_0}(z)\backslash(P_{2^{-j}\eps_0}(z)\cup\Omega)}(2^{-j}\eps_0)^s|D^kF^\top(z,\zeta)|d\Vol_\zeta
    \\
    &\qquad\lesssim_k\sum_{j=1}^J(2^{-j}\eps_0)^s(2^{-j}\eps_0)^{1-k}\lesssim_{\eps_0}2^{-J(s+1-k)}\approx\delta(z)^{s+1-k}.
\end{aligned}
\end{equation}
By swapping $z$ and $\zeta$, the same argument yields \eqref{Eqn::SkewBerg::FTopzeta}.
Replacing \eqref{Eqn::SkewBerg::BasisTop} by \eqref{Eqn::SkewBerg::BasisBot}, the same computation as in \eqref{Eqn::SkewBerg::PfTop} yields \eqref{Eqn::SkewBerg::FBotz} and \eqref{Eqn::SkewBerg::FBotzeta}.

Applying Schur's test, Lemma~\ref{Lem::PfThm::Schur} to \eqref{Eqn::SkewBerg::FTopz} - \eqref{Eqn::SkewBerg::FBotzeta}, the same argument to the proof of Corollary~\ref{Cor::PfThm::WeiSobEst} yields that for $k\ge0$ and $0<s<k+|\alpha|-1$,
\begin{align}\label{Eqn::SkewBerg::WeiBddTop}
    \Fc_\alpha^\top:&L^p(\Uc\backslash\overline\Omega,\dist^{-s};\wedge^{0,1})\to W^{k,p}(\Omega,\dist^{k-1+|\alpha|-s}),&\forall1\le p\le\infty;
    \\\label{Eqn::SkewBerg::WeiBddBot}
    \Fc_\alpha^\bot:&L^p(\Uc\backslash\overline\Omega,\dist^{-s};\wedge^{0,1})\to W^{k,p}(\Omega,\dist^{k-\frac1{m_1}+|\alpha|-s}),&\forall1\le p\le\infty.
\end{align}

Applying Corollary~\ref{Cor::LPThm::BddDom}~\ref{Item::LPThm::BddDom::FtoW} and \ref{Item::LPThm::HLLem::WtoF}, we have the following embeddings for $s>0$ and $k>s+1-|\alpha|$: 
\begin{itemize}
    \item $\widetilde \Fs_{p\infty}^{s}(\Uc\backslash\overline\Omega)\hookrightarrow L^p(\Uc\backslash\overline\Omega,\dist^{-s})$;
    \item $W^{k,p}(\Omega,\dist^{k-1+|\alpha|-s})\hookrightarrow\Fs_{p\eps}^{s+1-|\alpha|}(\Omega)$;
    \item $W^{k,p}(\Omega,\dist^{k-\frac1{m_1}+|\alpha|-s})\hookrightarrow\Fs_{p\eps}^{s+\frac1{m_1}-|\alpha|}(\Omega)$.
\end{itemize} 
Taking compositions with \eqref{Eqn::SkewBerg::WeiBddTop} and \eqref{Eqn::SkewBerg::WeiBddBot} we obtain \eqref{Eqn::SkewBerg::FOpBdd}.
\end{proof}

\begin{proof}[Proof of Theorem~\ref{Thm::SkewBerg}]
    Since $B-F=\dbar_{z,\zeta}K$ by \eqref{Eqn::SkewBerg::CF}, by separating the degrees we have $F=B_0-\dbar_\zeta K_0$. For the same extension operator $\Ec$ in Lemma~\ref{Lem::Goal::HomotopyFormula} (see Definition~\ref{Defn::Space::ExtOmega}) we have (see also \cite[Proposition~2.1]{GongHolderSPsiCXC2}) for every $f\in\Ss'(\Omega)$,
    \begin{align*}
        \Pc f(z)=&f(z)-\Hc_1\dbar f(z)=\Ec f(z)-\int_\Uc B_0(z,\cdot)\wedge\Ec\dbar f-\int_{\Uc\backslash\overline\Omega} K_0(z,\cdot)\wedge[\dbar,\Ec]\dbar f
        \\
        =&\int_\Uc \dbar_\zeta B_0(z,\cdot)\wedge \Ec f-\int_\Uc B_0(z,\cdot)\wedge\dbar\Ec f+\int_{\Uc\backslash\overline\Omega} B_0(z,\cdot)\wedge[\dbar,\Ec] f+\int_{\Uc\backslash\overline\Omega} K_0(z,\cdot)\wedge\dbar[\dbar,\Ec] f
        \\
        =&\int_{\Uc\backslash\overline\Omega} B_0(z,\cdot)\wedge[\dbar,\Ec]f-\int_{\Uc\backslash\overline\Omega}\dbar_\zeta K_0(z,\cdot)\wedge[\dbar,\Ec]f=\int_{\Uc\backslash\overline\Omega}F(z,\cdot)\wedge[\dbar,\Ec]f.
    \end{align*}
    
    Take $k\in\Z_+$ such that $k>1-s$. Using Proposition~\ref{Prop::Space::ATDOp}, Remark~\ref{Rmk::Goal::TopBotFacts}, and integration by parts,
    \begin{align*}
        \Pc f(z)=&\int_{\Uc\backslash\overline\Omega}F^\top(z,\cdot)\wedge[\dbar,\Ec]^\bot f+F^\bot(z,\cdot)\wedge[\dbar,\Ec]^\top f 
        \\
        =&\sum_{|\alpha|\le k}\int_{\Uc\backslash\overline\Omega}F^\top(z,\cdot)\wedge D^\alpha\Sc^{k,\alpha}[\dbar,\Ec]^\bot f+F^\bot(z,\cdot)\wedge D^\alpha\Sc^{k,\alpha}[\dbar,\Ec]^\top f          \\
        =&\sum_{|\alpha|\le k}(-1)^{|\alpha|}\int_{\Uc\backslash\overline\Omega}D_\zeta^\alpha( F^\top)(z,\cdot)\wedge\Sc^{k,\alpha}[\dbar,\Ec]^\bot f+D_\zeta^\alpha (F^\bot)(z,\cdot)\wedge\Sc^{k,\alpha}[\dbar,\Ec]^\top f
        \\
        =&\sum_{|\alpha|\le k}(-1)^{|\alpha|}\Big(\Fc_\alpha^\top\Sc^{k,\alpha}[\dbar,\Ec]^\bot+\Fc_\alpha^\top\Sc^{k,\alpha}[\dbar,\Ec]^\top\Big)[f].
    \end{align*}
    
    Recall from Corollary~\ref{Cor::LPThm::BddDom}~\ref{Item::LPThm::BddDom::TangComm} and Remark~\ref{Rmk::Space::TLRmk}~\ref{Item::Space::TLRmk::Embed} that $[\dbar,\Ec]^\top:\Fs_{p\infty}^s(\Omega)\to\Fs_{p\eps}^{s-1+\frac1{m_1}}(\Omega;\wedge^{0,1})$. Therefore, for every $|\alpha|\le k$
    \begin{align*}
        \Fc_\alpha^\top\Sc^{k,\alpha}[\dbar,\Ec]^\bot:&\Fs_{p\infty}^s(\Omega)\xrightarrow{[\dbar,\Ec]^\bot}\widetilde \Fs_{p\infty}^{s-1}(\overline{\Uc\backslash\Omega})\xrightarrow{\Sc^{k,\alpha}}\widetilde \Fs_{p\infty}^{s-1+k}(\overline{\Uc\backslash\Omega})\xrightarrow[\eqref{Eqn::SkewBerg::FOpBdd}]{\Fc_\alpha^\top}\Fs_{p\eps}^s(\Omega);
        \\
        \Fc_\alpha^\bot\Sc^{k,\alpha}[\dbar,\Ec]^\top:&\Fs_{p\infty}^s(\Omega)\xrightarrow{[\dbar,\Ec]^\top}\widetilde \Fs_{p\infty}^{s-\frac1{m_1}}(\overline{\Uc\backslash\Omega})\xrightarrow{\Sc^{k,\alpha}}\widetilde \Fs_{p\infty}^{s-\frac1{m_1}+k}(\overline{\Uc\backslash\Omega})\xrightarrow[\eqref{Eqn::SkewBerg::FOpBdd}]{\Fc_\alpha^\bot}\Fs_{p\eps}^s(\Omega).
    \end{align*}
    Taking sums over $\alpha$ we obtain \eqref{Eqn::SkewBerg::TLBdd}. 
    
    The $H^{s,p}$ and $\Co^s$ bounds follow from the inclusions $\Fs_{p1}^s(\Omega)\hookrightarrow H^{s,p}(\Omega)=\Fs_{p2}^s(\Omega)\hookrightarrow\Fs_{p\infty}^s(\Omega)$ and $\Fs_{\infty1}^s(\Omega)\hookrightarrow \Co^s(\Omega)=\Fs_{\infty\infty}^s(\Omega)$, as discussed in Remark~\ref{Rmk::Space::TLRmk} \ref{Item::Space::TLRmk::Embed} and \ref{Item::Space::TLRmk::SobHold}. 
\end{proof}

\bibliographystyle{amsalpha}
\bibliography{reference} 
\end{document}